\documentclass[11pt]{article}
\usepackage{amssymb, amsthm, amsmath, amscd}
\newtheorem{thm}{Theorem}[section]
\newtheorem{lem}[thm]{Lemma}
\newtheorem{prop}[thm]{Proposition}
\newtheorem{cor}[thm]{Corollary}
\newtheorem{NN}[thm]{}
\theoremstyle{definition}\newtheorem{df}[thm]{Definition}
\theoremstyle{definition}\newtheorem{rem}[thm]{Remark}
\theoremstyle{definition}

\renewcommand{\phi}{\varphi}

\newcommand{\N}{\mathbb{N}}
\newcommand{\Z}{\mathbb{Z}}
\newcommand{\Q}{\mathbb{Q}}
\newcommand{\R}{\mathbb{R}}
\newcommand{\C}{\mathbb{C}}
\newcommand{\T}{\mathbb{T}}

\newcommand{\morp}{contractive completely positive linear map}

\newcommand{\hm}{homomorphism}
\newcommand{\dt}{\delta}
\newcommand{\ep}{\epsilon}
\newcommand{\andeqn}{\,\,\,{\rm and}\,\,\,}
\newcommand{\rforal}{\,\,\,{\rm for\,\,\,all}\,\,\,}
\newcommand{\CA}{$C^*$-algebra}
\newcommand{\SCA}{$C^*$-subalgebra}

\newcommand{\af}{{\alpha}}
\newcommand{\bt}{{\beta}}

\newcommand{\beq}{\begin{eqnarray}}
\newcommand{\eneq}{\end{eqnarray}}
\newcommand{\tforal}{\,\,\,\text{for\,\,\,all}\,\,\,}
\newcommand{\tand}{\,\,\,\text{and}\,\,\,}

\title{Asymptotically Unitary Equivalence and Classification of Simple Amenable $C^*$-algebras
}
\author{Huaxin Lin
 }
\date{}

\begin{document}

\maketitle

\begin{abstract}
Let $C$  and $A$ be two unital separable amenable simple \CA s
with tracial rank no more than one. Suppose that $C$ satisfies the
Universal Coefficient Theorem and suppose that $\phi_1, \phi_2:
C\to A$ are two unital monomorphisms. We show that there is a
continuous path of unitaries $\{u_t: t\in [0, \infty)\}$ of $A$
such that
$$
\lim_{t\to\infty}u_t^*\phi_1(c)u_t=\phi_2(c)\tforal c\in C
$$
if and only if $[\phi_1]=[\phi_2]$ in $KK(C,A),$
$\phi_1^{\ddag}=\phi_2^{\ddag},$ $(\phi_1)_T=(\phi_2)_T$ and a
rotation related map $\overline{R}_{\phi_1,\phi_2}$ associated
with $\phi_1$ and $\phi_2$ is zero.

Applying this result together with a result of W. Winter, we give
a classification theorem for a class ${\cal A}$ of unital
separable simple amenable \CA s which is strictly larger than the
class of separable \CA s whose tracial rank are zero or one.
Tensor products of two \CA s in ${\cal A}$ are again in ${\cal
A}.$ Moreover, this class is closed under inductive limits and
contains
all unital simple ASH-algebras whose state spaces of
$K_0$ is the same as the tracial state spaces as well as some unital simple ASH-algebras
whose $K_0$-group is $\Z$ and tracial state spaces are any metrizable Choquet simplex.  One consequence
of the main result is that
all unital simple AH-algebras which are ${\cal Z}$-stable are
 isomorphic to ones with no dimension growth.

\end{abstract}

\section{Introduction}

\CA s are often viewed as non-commutative topological spaces.
Indeed, by a classical theorem of Gelfand, all unital commutative
\CA s are isomorphic to $C(X),$ the algebra of all continuous
functions on some compact Hausdorff space $X.$ We are interested
in the non-commutative cases, in particular, in the extremal end
of non-commutative cases, namely, unital simple \CA s. As in the
commutative case, one studies continuous maps from one space to
another, here we are interested in the \hm s from one unital
simple \CA\, to another unital simple \CA. In this paper, we study
the problem when two unital \hm s from one unital simple \CA\, to
another are asymptotically unitarily equivalent. To be more
precise,  let $C$ and $A$ be two unital separable simple \CA s and
let $\phi_1, \phi_2: C\to A$ be two unital \hm s. We consider the
question when there exists a continuous path of unitaries $\{u(t):
t\in [0,1)\}$ such that
\beq\label{AA}
\lim_{t\to 1} u(t)^* \phi_1(c) u(t)=\phi_2(c)\tforal c\in C.
\eneq

In this paper, we assume that both $A$ and $C$ have tracial rank
no more than one (see \ref{Dtr1} below). We prove the following:
Suppose that $C$ also satisfies the
 Universal Coefficient Theorem.
 Then (\ref{AA}) holds if and only if $[\phi_1]=[\phi_2]$
in $KK(C,A),$ $\phi_1^{\ddag}=\phi^{\ddag},$
$(\phi_1)_T=(\phi_2)_T$ and a rotation related map
$\overline{R}_{\phi_1, \phi_2}=0.$ These conditions are all
$K$-theory related. Except, perhaps, the condition
$\overline{R}_{\phi_1,\phi_2}=0,$ all others are obviously
necessary. The technical detail of these conditions will be
discussed later. The result has been established when $A$ is
assumed to be a unital simple separable \CA\, with tracial rank
zero. In fact, in that case $C$ can be any unital AH-algebra. As
in the case that $A$ has tracial rank zero, one technical problem
involved in proving the above mentioned asymptotic unitary
equivalence theorem is the so-called Basic Homotopy Lemma (see
\cite{Lnasym}). A version of the Basic Homotopy Lemma for the case
that $A$ has tracial rank one was recently established in
\cite{Lnnhomp} which paves the way for this paper.

 As in the commutative case, understanding continuous maps is
 important, the answer to the above mentioned question has many applications.
 For example, a special version of this was used to study the
 embedding of crossed products into a unital simple AF-algebras (see \cite{LnZ2af}).
However, in this paper, we will only give an  application to the
classification of simple amenable \CA s.

In the study of classification of amenable \CA s, or otherwise
known as the Elliott program, two types of results are
particularly important: the so-called uniqueness theorem and
existence theorem. Briefly, the uniqueness theorem usually means a
theorem which states that, for certain \CA s $A$ and $B,$ if two
unital monomorphisms from $A$ to $B$ carry the same $K$-theory (or
rather $KK$-theory) related information, then they are
approximately unitarily equivalent. The existence theorem on the
other hand usually states that given certain $K$-theory related
maps from the relevant $K$-theory information of $A$ to those of
$B$ (for example, maps from $\kappa: K_*(A)\to K_*(B)$),  there is
indeed a map (often a unital monomorphism) from $A$ to $B$ which
induces the given map (for example, a \hm\, $\phi: A\to B$ such
that $\phi_*=\kappa$). These two types of results play the crucial
roles in all classification results in the Elliott program.

In a recent paper, W. Winter (\cite{W}) demonstrated a new
approach to the Elliott program. Let $A$ and $B$ be two unital
separable amenable \CA s which satisfy the UCT and which are
${\cal Z}$-stable. Suppose that $A\otimes M_{\mathfrak{p}}$ and
$B\otimes M_{\mathfrak{p}}$ are isomorphic for any supernatural
number $\mathfrak{p}$ and $M_{\mathfrak{p}}$ is the UHF-algebra
associated with ${\mathfrak{p}}.$ W. Winter provided a way to show
that $A$ and $B$ are actually isomorphic. He also showed how it
works for the case that both $A\otimes M_{\mathfrak{p}}$ and
$B\otimes M_{\mathfrak{p}}$ have tracial rank zero. In order to
have Winter's method to work, in general, one also needs a
uniqueness as well as an existence theorem. However, one requires
an even finer uniqueness theorem as well as a finer existence
theorem. Instead of approximate unitarily equivalent, the
uniqueness theorem now requires
 two unital monomorphisms to be asymptotically unitarily
equivalent. Thus the above mentioned theorem for asymptotic
unitary equivalence serves as the required uniqueness theorem.

The existence theorem is equally important. Let $\kappa\in
KK(C,A)$ be  such that $\kappa(K_0(C)_+\setminus\{0\})\subset
K_0(A)_+\setminus\{0\}$ with $\kappa([1_C])=[1_A].$ Let $\gamma:
T(A)\to T(C)$ be an affine continuous map. We say that $\gamma$
and $\kappa$ are compatible if
$\tau(\kappa([p]))=\gamma(\tau)([p])$ for all tracial states
$\tau\in T(A)$ and all projections $p\in M_n(C),$ $n=1,2,....$
Denote by $U(C)$ and $U(A)$ the unitary groups of $C$ and $A$ and
by $CU(C)$ and $CU(A)$ the closures of the subgroups generated by
commutators of $U(C)$ and $U(A),$ respectively. A \hm\, $\alpha:
U(C)/CU(C)\to U(A)/CU(C)$ is said to be compatible with $\kappa$
if the induced maps from $K_1(C)$ into $K_1(A)$ determined by
$\alpha$ and $\kappa|_{K_1(C)}$ are the same. One also requires
that $\af,$ $\gamma$ and $\kappa$ are compatible (see 8.1 below).
As the first part of the existence theorem, we show that given
such a triple, there is indeed a unital \hm\, $\phi: C\to A$ such
that $\phi$ induces $\kappa, $ $\alpha$ and $\gamma,$ i.e.,
$[\phi]=\kappa,$ $\phi^{\ddag}=\alpha$ and $\phi_T=\gamma.$ The
previous results of this kind can be found in L. Li's paper
(\cite{Li} with a special case in \cite{EL}) and \cite{LnKT} as
well as \cite{LnKK}. One should notice that it is relatively easy
(or at least known) to find a \hm\, $\phi$ such that
$[\phi]=\overline{\kappa}$ in $KL(C,A),$ where $\overline{\kappa}$
is the image of $\kappa$ in $KL(C,A)$ once the case that $K_*(C)$
is finitely generated is established (as in \cite{Li}). It is
quite different to match the $KK$-element when $K_*(C)$ is not
finitely generated given the fact that
$\lim_nKK(C_n,A)\not=KK(\lim_n C_n, A)$ in general.  A general
method to realize $\kappa$ can be found in \cite{LN} which can be
traced back to \cite{KK2}. It again involves the Basic Homotopy
Lemma among other things. Trace formulas and de la Harp and
Scandalis determinants play important roles to match maps $\gamma$
and $\alpha.$

 However, more are required this time for existence theorem.
We show that, for any \\$\lambda\in Hom(K_1(C),
\overline{\rho_A(K_0(A))})),$ there is a unital \hm\, $\psi:
\phi(C)\to A$ for which there is a sequence of unitaries
$\{u_n\}\subset A$ such that $\psi(a)=\lim_{n\to\infty}{\rm ad}\,
u_n(a)$ for all $a\in \phi(C)$ and the rotation related map
$\overline{R}_{\psi\circ \phi, \phi}$ is induced by $\lambda.$

Let ${\cal A}$ be the class of all unital separable simple
amenable \CA s $A$ which satisfy the UCT such that $A\otimes
M_{\mathfrak{p}}$ have tracial rank one or zero for any
supernatural number $\mathfrak{p}.$  Using the above mentioned the
strategy as well as the so-called uniqueness and existence
theorem, we show that two \CA s $A$ and $B$ in ${\cal A}$ which
are ${\cal Z}$-stable are isomorphic if and only if their Elliott
invariant are the same. This considerably enlarges the class of
unital simple \CA s that can be classified by their Elliott
invariant. Class ${\cal A}$  contains all unital simple
AH-algebras. It also contains all unital simple separable amenable
\CA s with tracial rank one or zero which satisfy the UCT and all
unital simple separable \CA s which are inductive limits of type
$I$ \CA s with the unique tracial states. In particular, it
contains the Jiang-Su algebra ${\cal Z}.$ Moreover, it contains
those unital simple ASH-algebras $A$ for which
$T(A)=S_{[1]}(K_0(A))$ (the latter is the state space of
$K_0(A)$). It is closed under inductive limits and  tensor
products. As a consequence, it shows that any two unital simple
AH-algebras are ${\cal Z}$-stably isomorphic if and only if they
have the same Elliott invariant. In a subsequent paper
(\cite{LN2}), it will show that the class contains all unital
simple AD-algebras as well as unital simple A$\T$D-algebras. Let
$(G_0, (G_0)_+, u, G_1, S, r)$ be a six-tuple, where $S$ is any
metrizable Choquet simplex, $G_1$ is any countable abelian group,
$(G_0, (G_0)_+, u)$ is any weakly unperforated Riesz group with
order unit $u$ and any pairing $r,$ it will show that there is a
unital simple separable amenable \CA\, $A\in {\cal A}$ so that its
Elliott invariant is precisely the given six-tuple.  In that
paper, we will also show that there are \CA s in ${\cal A}$ whose
$K_0$-group may not be a Riesz groups (\cite{LN2}). Furthermore,
the range of the Elliott invariants of \CA s in ${\cal A}$ will be
determined.

\vspace{0.2in}

{\bf Acknowledgments} This work is partially supported by a NSF
grant, the Shanghai Priority Academic disciplines and Chang-Jiang
Professorship from East China Normal University in the summer of
2008.

\section{Preliminaries}

\begin{NN}\label{NN1}
{\rm Let $A$ be a stably finite \CA. Denote by $T(A)$ the tracial
state space of $A$ and denote by $Aff(T(A))$ the space of all real
affine continuous functions on $T(A).$ Suppose $\tau\in T(A)$ is a
tracial state. We will also use $\tau$ for the trace $\tau\otimes
Tr$ on $A\otimes M_k=M_k(A)$ (for every integer $k\ge 1$), where
$Tr$ is the standard trace on $M_k.$

Define $\rho_A: K_0(A)\to Aff(T(A))$ to be the positive \hm\,
defined by $\rho_A([p])(\tau)=\tau(p)$ for each projection $p$ in
$M_k(A).$

}

\end{NN}

\begin{NN}\label{DDCU}
{\rm Let $A$ be a unital \CA. Denote by $U(A)$ the group of
unitaries in $A$ and denote by $U_0(A)$ the path connected
component of $U(A)$ containing the identity. Denote by $CU(A)$ the
{\it closure} of the subgroup of $U(A)$ generated by commutators.
It is a normal subgroup of $U(A).$ If $u\in U(A),$ $\overline{u}$ will be used for its image in $U(A)/CU(A).$
If $x, y\in U(A)/CU(A),$ define
$$
{\rm dist}(x, y)=\inf\{\|v^*u-1\|: {\bar u}=x\andeqn {\bar v}=y\}.
$$

Suppose that $K_1(A)=U(A)/U_0(A).$ 
We denote by $\Delta_A: U_0(A)/CU(A)\to \text{Aff}(T(A))/\overline{\rho_A(K_0(A))}$ the 
de la Harp and Skandalis 
determinant (see \cite{HS} and \cite{Th1}).  Let 
$$
d_A'({\bar f},{\bar g})=\inf\{\|f-g-h\|: h\in \overline{\rho_A(K_0(A))}\}
$$
for $f, g\in \text{Aff}(T(A))$ (${\bar f}$ and ${\bar g}$ are images of $f$ and $g$ in $\text{Aff}(T(A))/\overline{\rho_A(T(A))},$
respectively). Define 
$$
d_A(\bar{f}, \bar{g})=\begin{cases} 2 & \text{if $d_A'({\bar f},{\bar g})\ge 1/2$ }\\
                               |e^{2\pi d_A'({\bar f}, {\bar g})}-1| &\text{if  $d_A'({\bar f},{\bar g})<1/2.$}
                               \end{cases}
$$
Note that $d_A({\bar f},{\bar g})<2\pi d_A'({\bar f}, {\bar g}).$  As in 3.4 of \cite{EGL}, $\Delta_A$ gives an 
isometric isomorphism (The isomorphism follows \cite{Th1} and the isometric part follows  (a slight modification)  proof of 3.1 of \cite{NT}). 

  }
\end{NN}

\begin{NN}
 {\rm By $Aut(A)$ we mean the group of automorphisms on $A.$ Let
$u\in U(A).$ We write ${\rm ad}\, u$ the inner automorphism
defined by ${\rm ad}\, u(a)=u^*au$ for all $a\in A.$

Let $C\subset A$ be a \SCA. Denote by $\overline{\text{Inn}}(C,A)$
the set of those monomorphisms $\phi: C\to A$ such that
$\phi(c)=\lim_{n\to\infty}u_n^*cu_n$ for some sequence of
unitaries $\{u_n\}$ of $A$ and for all $c\in C.$

}
\end{NN}

\begin{NN}
{\rm Suppose that $A$ and $B$ are two unital \CA s and $\phi: A\to
B$ is a \hm. Then $\phi$ maps $CU(A)$ to $CU(B).$ Denote by
$\phi^{\ddag}: U(A)/CU(A)\to U(B)/CU(B)$ the induced \hm. Suppose
that $T(A)\not=\emptyset$ and $T(B)\not=\emptyset.$ Denote by
$\phi_T: T(B)\to T(A)$ the continuous affine map defined by
$$
\phi_T(\tau)(a)=\tau(\phi(a))\tforal a\in A\andeqn \tau\in T(B).
$$

}
\end{NN}

\begin{NN}
{\rm  Let $A$ be a \CA. Denote by $SA=C_0((0,1), A)$ the
suspension of $A.$ }
\end{NN}

\begin{NN}

{\rm A \CA\, $A$ is an AH-algebra if $A=\lim_{n\to\infty}(A_n,
\psi_n),$ where each $A_n$ has the form $P_nM_{k(n)}(C(X_n))P_n,$
where $X_n$ is a finite CW complex (not necessarily connected) and
$P_n\in M_{k(n)}(C(X_n))$ is a projection. We use $\psi_{n,
\infty}: A_n\to A$ for the induced \hm. }
\end{NN}

\begin{NN}
{\rm Denote by ${\cal N}$ the class of separable amenable \CA\,
which satisfies the Universal Coefficient Theorem. }

\end{NN}

\begin{df}\label{Dtr1}

{\rm Recall (see Theorem 6.13 of \cite{Lnplms}) that a unital
simple \CA\, $A$ is said to have tracial rank no more than one
(written $TR(A)\le 1$), if for any $\ep>0,$ any $a\in
A_+\setminus\{0\}$ and any finite subset ${\cal F}\subset A,$
there exists a projection $p\in A$ and a \SCA\, $C$ of $A$ with
$1_C=p$ and $C=\oplus_{i=1}^mC(X_i, M_{r(i)}),$ where $X_i$ is a
connected finite CW complex with dimension ${\rm dim} X\le 1,$
such that

{\rm (1)} $\|px-xp\|<\ep\tforal x\in {\cal F},$

{\rm (2)} ${\rm dist}(pxp, C)<\ep\rforal x\in {\cal F},$

{\rm (3)} $1-p$ is von-Neumann equivalent to a projection in
$\overline{aAa}.$

Denote by ${\cal I},$ the class of those \CA s\, with the form
$\oplus_{i=1}^m C([0,1], M_{r(i)})\oplus \oplus_{j=1}^{m'}
M_{R(j)}.$ It was proved (Theorem 7.1 of \cite{Lnplms}) that if
$TR(A)\le 1,$ then in the above definition, we can choose $C\in
{\cal I}.$ If furthermore, $C$ can be replaced by finite
dimensional \SCA s, then $TR(A)=0.$ If $TR(A)\le 1$ and
$TR(A)\not=0,$ we write $TR(A)=1.$ By Theorem 2.5 of
\cite{Lnctr1}, if $B$ is a unital simple AH-algebra with very slow
dimension growth, then $TR(B)\le 1.$

}

\end{df}

\begin{NN}
{\rm  Let $A$ and $B$ be two unital \CA s and let $\phi: A\to B$
be a \hm. One can extend $\phi$ to a \hm\, from $M_k(A)$ to
$M_k(B)$ by $\phi\otimes {\rm id}_{M_k}.$ In what follows we will
use $\phi$ again for this extension without further notice.}

\end{NN}

\begin{NN}

{\rm Let $A$ and $B$ be two \CA s and let $L_1, L_2: A\to B$ be a
map.  Suppose that ${\cal F}\subset A$ is a subset and  $\ep>0.$
We write
$$
L_1\approx_{\ep} L_2\,\,\,\text{on}\,\,\, {\cal F},
$$
if
$$
\|L_1(a)-L_2(a)\|<\ep\tforal a\in {\cal F}.
$$

}

\end{NN}

\begin{NN}\label{Dppu}
{\rm Let $A$ and $B$ be two unital \CA s and let $L: A\to B$ be a
unital \morp. Let ${\cal P}\subset \underline{K}(A)$ be a finite
subset. It is well known that, for some small $\dt$ and large finite
subset ${\cal G}\subset A,$ if $L$ is also $\dt$-${\cal
G}$-multiplicative, then $[L]|_{\cal P}$ is well defined. In what
follows whenever we write $[L]|_{\cal P}$ we mean $\dt$ is
sufficiently small and ${\cal G}$ is sufficiently large so that it
is well defined (see 2.3 of \cite{Lnhomp}).

If $u\in U(A),$ we will use $ \langle L\rangle (u) $ for the unitary
$L(u)|L(u)^*L(u)|^{-1}.$

For an integer $m\ge 1$ and a  finite subset ${\cal U}\subset
U(M_m(A)),$ let $F\subset U(A)$ be the subgroup generated by
${\cal U}.$ As in 6.2 of \cite{Lnctr1}, there exists a finite
subset ${\cal G}$ and a small $\dt>0$ such that a $\dt$-${\cal
G}$-multiplicative \morp\, $L$ induces a \hm\, $L^{\ddag}:
{\overline{F}}\to U(M_m(B))/CU(M_m(B)).$  Moreover, we may assume,
${\overline{\langle L\rangle (u)}}=L^{\ddag}({\bar u}).$
If there are $L_1, L_2: A\to B$ and $\ep>0$ is given. Suppose that
both $L_1$ and $L_2$ are $\dt$-${\cal G}$-multiplicative and
$L_1^{\ddag}$ and $L_2^{\ddag}$ are well defined on
${\overline{F}},$ whenever, we write
$$
{\rm dist}(L_1^{\ddag}({\bar u}), L_2^{\ddag}({\bar u}))<\ep
$$
for all $u\in {\cal U},$ we also assume that $\dt$ is sufficiently
small and ${\cal G}$ is sufficiently large so that
$$
{\rm dist}{\overline{\langle L_1\rangle(u)}}, \overline{\langle
L_2\rangle(u)})<\ep
$$
for all $u\in {\cal U}.$

}

\end{NN}

\begin{df}\label{Kund}
{\rm Let $A$ be a \CA.  Following Dadarlat and Loring (\cite{DL}),
denote
$$
\underline{K}(A)=\bigoplus_{i=0,1}
K_i(A)\bigoplus_{i=0,1}\bigoplus_{k\ge 2}K_i(A,\Z/k\Z).
$$
Let $B$ be a unital \CA. If furthermore, $A$ is assumed to be
separable and satisfy the Universal Coefficient Theorem
(\cite{RS}), by \cite{DL},
$$
Hom_{\Lambda}(\underline{K}(A), \underline{K}(B))=KL(A,B).
$$
Here $KL(A,B)=KK(A,B)/Pext(K_*(A), K_*(B)).$ (see \cite{DL} for
details).
Let $k\ge 1$ be an integer. Denote
$$
F_k\underline{K}(A)=\oplus_{i=0,1}K_i(A)\bigoplus_{n|k}K_i(A,\Z/k\Z).
$$
Suppose that $K_i(A)$ is finitely generated ($i=0,2$). It follows
from \cite{DL} that there is an integer $k\ge 1$ such that
\beq\label{dkl1}
Hom_{\Lambda}(F_k\underline{K}(A), F_k\underline{K}(B))
=Hom_{\Lambda}(\underline{K}(A), \underline{K}(B)).
\eneq
}
\end{df}

\begin{df}
{\rm Denote by $KK(A,B)^{++}$ those elements $x\in KK(A,B)$ such
that $x(K_1(A)_+\setminus \{0\})\subset K_1(B)_+\setminus\{0\}.$
Suppose that both $A$ and $B$ are unital. Denote by
$KK_e(A,B)^{++}$ the set of those elements $x$ in $KK(A,B)^{++}$
such that $x([1_A])=[1_B].$

}
\end{df}

\begin{df}\label{Dbot2}
{\rm Let
$A$ and $B$ be  two unital \CA s.  Let $h: A\to B$ be a \hm\, and
$v\in U(B)$ such that
$$
h(g)v=vh(g)\,\rforal\, g\in A.
$$
 Thus we
obtain a \hm\, ${\bar h}: A\otimes C(S^1)\to B$ by ${\bar
h}(f\otimes g)=h(f)g(v)$ for $f\in A$ and $g\in C(S^1).$ The
tensor product induces two injective \hm s:
\beq\label{dbot01}
\bt^{(0)}&:& K_0(A)\to K_1(A\otimes C(S^1))\\
 \bt^{(1)}&:&
K_1(A)\to K_0(A\otimes C(S^1)).
\eneq
The second one is the usual Bott map. Note, in this way, one
writes
$$K_i(A\otimes C(S^1))=K_i(A)\oplus \bt^{(i-1)}(K_{i-1}(A)).$$
We use $\widehat{\bt^{(i)}}: K_i(A\otimes C(S^1))\to
\bt^{(i-1)}(K_{i-1}(A))$ for the projection to
$\bt^{(i-1)}(K_{i-1}(A)).$

For each integer $k\ge 2,$ one also obtains the following
injective \hm s:
\beq\label{dbot02}
\bt^{(i)}_k: K_i(A, \Z/k\Z))\to K_{i-1}(A\otimes C(S^1), \Z/k\Z),
i=0,1.
\eneq
Thus we write
\beq\label{dbot02-1}
K_{i-1}(A\otimes C(S^1), \Z/k\Z)=K_{i-1}(A,\Z/k\Z)\oplus
\bt^{(i)}_k(K_i(A, \Z/k\Z)),\,\,i=0,1.
\eneq
Denote by $\widehat{\bt^{(i)}_k}: K_{i}(A\otimes C(S^1),
\Z/k\Z)\to \bt^{(i-1)}_k(K_{i-1}(A,\Z/k\Z))$ similarly to that of
$\widehat{\bt^{(i)}}.,$ $i=1,2.$ If $x\in \underline{K}(A),$ we
use ${\boldsymbol{\beta}}(x)$ for $\bt^{(i)}(x)$ if $x\in K_i(A)$
and for $\bt^{(i)}_k(x)$ if $x\in K_i(A, \Z/k\Z).$ Thus we have a
map ${\boldsymbol{ \bt}}: \underline{K}(A)\to
\underline{K}(A\otimes C(S^1))$ as well as
$\widehat{\boldsymbol{\bt}}: \underline{K}(A\otimes C(S^1))\to
 {\boldsymbol{ \bt}}(\underline{K}(A)).$ Therefore one may write
 $\underline{K}(A\otimes C(S^1))=\underline{K}(A)\oplus {\boldsymbol{ \bt}}( \underline{K}(A)).$

On the other hand ${\bar h}$ induces \hm s ${\bar h}_{*i,k}:
K_i(A\otimes C(S^1)), \Z/k\Z)\to K_i(B,\Z/k\Z),$ $k=0,2,...,$ and
$i=0,1.$
We use $\text{Bott}(h,v)$ for all  \hm s ${\bar h}_{*i,k}\circ
\bt^{(i)}_k.$ We write
$$
\text{Bott}(h,v)=0,
$$
if ${\bar h}_{*i,k}\circ \bt^{(i)}_k=0$ for all $k\ge 1$ and
$i=0,1.$
We will use $\text{bott}_1(h,v)$ for the \hm\, ${\bar
h}_{1,0}\circ \bt^{(1)}: K_1(A)\to K_0(B),$ and
$\text{bott}_0(h,u)$ for the \hm\, ${\bar h}_{0,0}\circ \bt^{(0)}:
K_0(A)\to K_1(B).$
Since $A$ is unital, if $\text{bott}_0(h,v)=0,$ then $[v]=0$ in
$K_1(B).$

In what follows, we will use $z$ for the standard generator of
$C(S^1)$ and we will often identify $S^1$ with the unit circle
without further explanation. With this identification $z$ is the
identity map from the circle to the circle.

}
\end{df}

\begin{NN}\label{ddbot}
{\rm Given a finite subset ${\cal P}\subset \underline{K}(A),$
there exists a finite subset ${\cal F}\subset A$ and $\dt_0>0$
such that
$$
\text{Bott}(h, v)|_{\cal P}
$$
is well defined, if
$$
\|[h(a),\, v]\|=\|h(a)v-vh(a)\|<\dt_0\tforal a\in {\cal F}
$$
(see 2.10 of \cite{Lnhomp}). There is $\dt_1>0$ (\cite{Lo}) such
that $\text{bott}_1(u,v)$ is well defined for any pair of
unitaries $u$ and $v$ such that $\|[u,\, v]\|<\dt_1.$ As in 2.2 of
\cite{ER}, if $v_1,v_2,...,v_n$ are unitaries such that
$$
\|[u, \, v_j]\|<\dt_1/n,\,\,\,j=1,2,...,n,
$$
then
$$
\text{bott}_1(u,\,v_1v_2\cdots v_n)=\sum_{j=1}^n\text{bott}_1
(u,\, v_j).
$$
By considering  unitaries $z\in  {\widetilde{A\otimes C}}$
($C=C_n$ for some commutative \CA\, with torsion $K_0$ and
$C=SC_n$), from the above, for a given unital separable amenable
\CA\, $A$ and a given finite subset ${\cal P}\subset
\underline{K}(A),$ one obtains a universal constant $\dt>0$ and a
finite subset ${\cal F}\subset A$ satisfying the following:
\beq\label{ddbot-1}
\text{Bott}(h,\, v_j)|_{\cal P}\,\,\, \text{is well
defined}\andeqn \text{Bott}(h,\, v_1v_2\cdots v_n)=\sum_{j=1}^n
\text{Bott}(h,\, v_j),
\eneq
for any unital \hm\, $h$ and unitaries $v_1, v_2,...,v_n$ for
which
\beq\label{ddbot-2}
\|[h(a),\, v_j]\|<\dt/n,\,\,\,j=1,2,...,n
\eneq
for all $a\in {\cal F}.$
If furthermore, $K_i(A)$ is finitely generated, then (\ref{dkl1})
holds. Therefore, there is a finite subset ${\cal Q}\subset
\underline{K}(A),$ such that
$$
\text{Bott}(h,v)
$$
is well defined if $\text{Bott}(h, v)|_{\cal Q}$ is well defined
(see also 2.3 of \cite{Lnhomp}).

See section 2 of \cite{Lnhomp} for the further information. }

\end{NN}

We will use  the following the theorems frequently.

\begin{thm}{\rm ( Theorem 8.4 of \cite{Lnnhomp})}\label{LNHOMP}
Let $A$ be a unital simple \CA\, in ${\cal N}$   with $TR(A)\le 1$
and let $\ep>0$ and ${\mathcal F}\subset A$ be a finite subset.
Suppose that $B$ is a unital separable simple \CA\, with $TR(B)\le
1$ and $h: A\to B$ is a unital monomorphism. Then there exists
$\dt>0,$ a finite subset ${\mathcal G}\subset A$ and a finite
subset ${\mathcal P}\subset \underline{K}(A)$ satisfying the
following: Suppose that there is a unitary $u\in B$ such that
\beq\label{NCT1C-1}
\|[h(a), u]\|<\dt\tforal f\in {\mathcal G}\andeqn
\rm{Bott}(h,u)|_{{\mathcal P}}=0.
\eneq
 Then there exists a piece-wise smooth and
continuous path of unitaries $\{u_t:t\in [0,1]\}$ such that
$$
u_0=u,\,\,\, u_1=u, \|[h(a), v_t]\|<\ep\rforal f\in {\mathcal
F}\andeqn t\in [0,1].
$$

\end{thm}

We will also use the following

\begin{thm}{\rm (Corollary 11.8 of \cite{Lnn1})}\label{Tnn1}
Let $C$ be a unital simple \CA\, in ${\cal N}$ with $TR(C)\le 1$
and let $B$ be a unital separable simple \CA\, with $TR(B)\le 1.$
Suppose that $\phi_1, \phi_2: C\to B$ are two unital
monomorphisms. Then there exists a sequence of unitaries $\{u_n\}$
of $A$ such that
$$
\lim_{n\to\infty}{\rm ad}\, u_n\circ \phi_1(a)=\phi_2(a)\tforal
a\in C
$$
if and only if
$$
[\phi_1]=[\phi_2]\,\,\,\text{in}\,\,\, KL(C,B),\,\,\,
\phi_1^{\ddag}=\phi_2^{\ddag}\andeqn (\phi_1)_T=(\phi_2)_T.
$$
\end{thm}

\section{Rotation maps and  Exel's trace formula}

\begin{NN}\label{Mtorus}
{\rm Let $A$ and $B$ be two unital \CA s. Suppose that $\phi,
\psi: A\to B$ are two monomorphisms. Define
\beq\label{Dmt-1}
M_{\phi, \psi}=\{x\in C([0,1], B): x(0)=\phi(a)\andeqn
x(1)=\psi(a) \,\,\,\text{for some}\,\,\, a\in A\}.
\eneq
When $A=B$ and $\phi={\rm id}_A,$ $M_{\phi, \psi}$ is the usual
mapping torus.  We may call $M_{\phi, \psi}$ the (generalized)
mapping torus of $\phi$ and $\psi.$ This notation will be used
throughout of this article.  Thus one obtains an exact sequence:
\beq\label{Dmt-2}
0\to SB \stackrel{\imath}{\to} M_{\phi, \psi}\stackrel{\pi_0}{\to}
A\to 0,
\eneq
where $\pi_0: M_{\phi, \psi}\to A$ is  identified with the
point-evaluation at the point $0.$

Suppose that $A$ is a separable amenable \CA.  From
(\ref{Dmt-2}), one obtains an element in $Ext(A,SB).$ In this case
we identify $Ext(A,SB)$ with $KK^1(A,SB)$ and $KK(A,B).$

Suppose that $[\phi]=[\psi]$ in $KL(A,B).$ The mapping torus
$M_{\phi, \psi}$ corresponds a trivial element in $KL(A, B).$ It
follows that there are two exact sequences:
\beq\label{Dmt-3}
&&0\to K_1(B)\stackrel{\imath_*}{\to} K_0(M_{\phi,
\psi})\stackrel{(\pi_0)_*}{\to} K_0(A)\to 0\andeqn\\\label{Dmt-4}
&&0\to K_0(B)\stackrel{\imath_*}{\to}K_1(M_{\phi, \psi})
\stackrel{(\pi_0)_*}{\to} K_1(A)\to 0.
\eneq
which are pure extensions of abelian groups.

 }

\end{NN}

\begin{df}\label{Dr}
{\rm Suppose that $T(B)\not=\emptyset.$  Let $u\in M_l(M_{\phi,
\psi})$ be a unitary which is a piecewise smooth function on
$[0,1].$ For each $\tau\in T(B),$ we denote by $\tau$  the trace
$\tau\otimes Tr$ on $M_l(B),$ where $Tr$ is the standard trace on
$M_l$ as in \ref{NN1}. Define \beq\label{Dr-1}
R_{\phi,\psi}(u)(\tau)={1\over{2\pi i}}\int_0^1
\tau({du(t)\over{dt}}u(t)^*)dt. \eneq It is easy to see that
$R_{\phi, \psi}(u)$ has real values.
If
\beq\label{Dr-2}
\tau(\phi(a))=\tau(\psi(a))\rforal a\in A\andeqn \tau\in T(B),
\eneq
then
 there exists a \hm\,
$$
R_{\phi, \psi}: K_1(M_{\phi, \psi})\to Aff(T(B))
$$
defined by
$$
R_{\phi, \psi}([u])(\tau)={1\over{2\pi i}}\int_0^1
\tau({du(t)\over{dt}}u(t)^*)dt.
$$
}

\end{df}

If $p$ is a projection in $M_l(B)$ for some integer $l\ge 1,$ one
has $\imath_*([p])=[u],$ where $u\in M_{\phi, \psi}$ is a unitary
defined by
$
u(t)=e^{2\pi it}p+(1-p)\,\,\,\text{for}\,\,\, t\in [0,1].
$
It follows that
$$
R_{\phi, \psi}(\imath_*([p]))(\tau)=\tau(p)\tforal \tau\in T(B).
$$
In other words,
$$
R_{\phi, \psi}(\imath_*([p]))=\rho_B([p]).
$$
Thus one has, exactly as in 2.2 of \cite{KK2},  the following:

\begin{lem}\label{DrL} When
{\rm (\ref{Dr-2})} holds, the following diagram commutes:
$$
\begin{array}{ccccc}
K_0(B) && \stackrel{\imath_*}{\longrightarrow} && K_1(M_{\phi, \psi})\\
& \rho_B\searrow && \swarrow R_{\phi,\psi} \\
& & Aff(T(B)) \\
\end{array}
$$

\end{lem}

\vspace{0.2in}

\begin{df}\label{eta}
If furthermore, $[\phi]=[\psi]$ in $KK(A,B)$ and $A$ satisfies the
Universal Coefficient Theorem, using Dadarlat-Loring's notation,
one has the following splitting exact sequence:
\beq\label{eta-1-}
0\to \underline{K}(SB)\,{\stackrel{[\imath]}{\to}}\,
\underline{K}(M_{\phi,\psi})\,{\stackrel{[\pi_0]}{\rightleftarrows}}_{\theta}
\,\,\underline{K}(A)\to 0.
\eneq
In other words, there is  $\theta\in
{\rm Hom}_{\Lambda}(\underline{K}(A), \underline{K}(M_{\phi,\psi}))$
such that $[\pi_0]\circ \theta=[\rm id_A].$ In particular, one has
a monomorphism $\theta|_{K_1(A)}: K_1(A)\to K_1(M_{\phi, \psi})$
such that $[\pi_0]\circ \theta|_{K_1(A)}=({\rm id}_A)_{*1}.$ Thus,
one may write
\beq\label{eta-1}
K_1(M_{\phi, \psi})=K_0(B)\oplus K_1(A).
\eneq
Suppose also that $\tau\circ \phi=\tau\circ \psi$ for all $\tau\in
T(A).$  Then  one obtains the \hm\,
\beq\label{eta-2}
R_{\phi,\psi}\circ \theta|_{K_1(A)}: K_1(A)\to Aff(T(B)).
\eneq

We say a  rotation related map  vanishes,  if there exists a such
splitting map $\theta$ such that
$$R_{\phi, \psi}\circ \theta|_{K_1(A)}=0.$$

Denote by ${\cal R}_0$ the set of those \hm s $\lambda \in
{\rm Hom}(K_1(A), \text{Aff}(T(B)))$ for which there is a \hm\, $h:
K_1(A)\to K_0(B)$ such that $\lambda=\rho_A\circ h.$ It is a
subgroup of ${\rm Hom}(K_1(A), \text{Aff}(T(B))).$
Let $\theta, \theta'\in
{\rm Hom}_{\Lambda}(\underline{K}(A),\underline{K}(M_{\phi, \psi}))$
such that $[\pi_0]\circ [\theta]=[{\rm id}_C]=[\pi_0]\circ
[\theta'].$ Then $(\theta-\theta')(K_1(A)\subset K_0(B).$ In other
words,
$$
R_{\phi, \psi}\circ (\theta-\theta')|_{K_1(A)}\in {\cal R}_0.
$$
Thus, we obtain a well-defined element $\overline{R}_{\phi,
\psi}\in {\rm Hom}(K_1(A),\text{Aff}(T(B)))/{\cal R}_0$ (which does not
depend on the choices of $\theta$).

{\it In this case, if there is a \hm\, $\theta'_1: K_1(A)\to
K_1(M_{\phi, \psi})$ such that $(\pi_0)_{*1}\circ\theta'_1={\rm
id}_{K_1(A)}$ and
$$
R_{\phi, \psi}\circ \theta'_1\in {\cal R}_0,
$$
then there is $\Theta\in {\rm Hom}_{\Lambda}(\underline{K}(A),
\underline{K}(M_{\phi, \psi}))$ such that
$$
[\pi_0]\circ \Theta=[{\rm id}_A]\tand R_{\phi, \psi}\circ
\Theta=0.
$$
}

To see this, let $\theta\in {\rm Hom}_{\Lambda}(\underline{K}(A),
\underline{K}(M_{\phi, \psi}))$ such that $[\pi_0]\circ
\theta=[{\rm id}_A].$ There is a \hm\, $h: K_1(A)\to K_0(B)$ such
that
$
\rho_A\circ h=R_{\phi, \psi}\circ \theta'_1.
$
Define $\theta_1'': K_1(A)\to K_1(M_{\phi, \psi})$ by
$$
\theta_1''(x)=\theta_1'(x)-h(x)\tforal x\in K_1(A).
$$
Then $(\pi_0)_{*1}\circ \theta_1''=(\pi_0)_{*1}\circ \theta={\rm
id}_{K_1(A)}.$ Moreover,
$$
R_{\phi,\psi}\circ \theta_1''=0.
$$

To lift $\theta_1''$ to an element in $KL(C, K_1(M_{\phi,
\psi})),$ define $\theta_0''=\theta|_{K_0(A)}.$  By the UCT, there
exists $\theta''\in {KL}(C,M_{\phi, \psi})$ such that
$\Gamma(\theta'')=\theta_*'',$  where $\Gamma$ is the map from
${KL}(C, M_{\phi, \psi})$ onto $\mathrm{Hom}({K}_*(C),
{K}_*(M_{\phi, \psi})).$

Put
$$
x_0=[\pi_0]\circ
\theta''-[\rm{id}_{\underline{{K}}(C)}]\,\,\,\,{\rm in}\,\,\,
KL(C,C).
$$
Then $\Gamma(x_0)=0.$ Define $\Theta=\theta''-\theta\circ x_0\in
\mathrm{Hom}_{\Lambda}(\underline{{K}}(C),
\underline{{K}}(M_{\phi,\psi})).$ Then one computes that
\beq\label{Lr0-3}
[\pi_0]\circ \Theta &=& [\pi_0]\circ \theta''-[\pi_0]\circ \theta\circ x_0=([\rm{id}_{\underline{{K}}(C)}]+x_0)-[\rm{id}_{\underline{{K}}(C)}]\circ x_0\\&=&[\rm{id}_{\underline{{K}}(C)}]+x_0-x_0=[\rm{id}_{\underline{{K}}(C)}].
\eneq
Moreover,
$$
\Theta|_{{K}_1(C)}=\theta_1''.
$$
Therefore,
$$
R_{\phi, \psi}\circ \Theta|_{{K}_1(C)}=0.
$$

In particular, if $\overline{R}_{\phi, \psi}=0,$  there exists
$\Theta\in {\rm Hom}_{\Lambda}(\underline{K}(A),
\underline{K}(M_{\phi,\psi}))$ such that $[\pi_0]\circ
\Theta=[{\rm id}_A]$ and
$$
R_{\phi, \psi}\circ \Theta=0.
$$


When $\overline{R}_{\phi, \psi}=0,$ $\theta(K_1(A))\in {\rm
ker}R_{\phi,\psi}$ for some $\theta$ so that (\ref{eta-1-}) holds.
In this case $\theta$ also gives the following:
$$
{\rm ker}R_{\phi,\psi}={\rm ker}\rho_B\oplus K_1(A).
$$


\end{df}

\vspace{0.2in}

The following is a generalization of the Exel trace formula for
the Bott element.

\begin{thm}{\rm (see also Theorem 3.5 of \cite{Lnasym})}\label{Exel}
There is $\dt>0$ satisfying the following:
 Let $A$ be a unital
separable simple \CA\, with $TR(A)\le 1$ and let $u, v\in U(A)$ be
two unitaries such that
\beq\label{Exel-1}
\|uv-vu\|<\dt.
\eneq
Then $\text{bott}_1(u,v)$ is well defined and
\beq\label{Exel-2}
\rho_A({\rm{bott}_1}(u,v))(\tau)={1\over{2\pi
i}}(\tau(\log(vuv^*u^*)))\tforal \tau\in T(A).
\eneq

\end{thm}

\begin{proof}
There is $\dt_1>0$ (one may choose $\dt_1=2$) such that, for any
pair of unitaries for which $\|uv-vu\|<\dt_1,$
$\text{bott}_1(u,v)$ is well defined (see \ref{ddbot}). There is
also $\dt_2>0$ satisfying the following: if two pair of unitaries
$u_1, v_1, u_2, v_2$ such that
\beq\label{exl-1}
\|u_1-u_2\|<\dt_2,\,\,\,\|v_1-v_2\|<\dt_2
\eneq
as well as
$$
\|[u_1,\, v_1]\|<\dt_1/2\andeqn \|[u_2,\,v_2]\|<\dt_1/2,
$$
then
\beq\label{exl-2}
\text{bott}_1(u_1,v_1)=\text{bott}_1(u_2, v_2).
\eneq
We may also assume that
$$
\|v_1u_1v_1^*u_1-1\|<1\,\,\,\text{whenever}\,\,\,
\|[u_1,\,v_1]\|<\dt_2.
$$

 Let
$$
F=\{z\in S^1: |z-1|<1+1/2\}.
$$
Let $\log: F\to (-\pi, \pi)$ be a smooth branch of logarithm.
We choose
$\dt=\min\{\dt_1/2, \dt_2/2\}.$ Now fix a pair of unitaries $u,
v\in A$ with
\beq\label{exl-3}
\|[u,\,v]\|<\dt.
\eneq
For $\ep>0,$ there is $\dt_3>0$ such that
\beq\label{exl-3+1}
|\log(t)-\log (t')|<\ep
\eneq
provided that $|t-t'|<\dt_3$ and $t,t'\in F.$

Choose $\ep_1=\min\{\ep, \dt_3/4, \dt/2\}.$
Since $TR(A)\le 1,$ there is a \SCA\, $B\in {\cal I}$ of $A,$ a
projection $p\in A$ with $1_B=p,$  unitaries $u', v'\in B,$ and
$u'', v''\in (1-p)A(1-p)$ such that \beq\label{exl-4}
\|u'+u''-u\|<\ep_1,\,\,\, \|v'+v''-v\|<\ep_1\\\label{exl-5} \andeqn
\tau(1-p)<\ep_1\tforal \tau\in T(A). \eneq In particular,
\beq\label{exl-6} \text{bott}_1(u'+u'',v'+v'')=\text{bott}_1(u,v).
\eneq Therefore \beq\label{exl-6+1}
\rho_A(\text{bott}_1(u'+u'',v'+v''))(\tau)=\rho_A(\text{bott}_1(u,v))(\tau)
\eneq for all $\tau\in T(A).$

 Note that
\beq\label{exl-7}
\|[u',\, v']\|<\dt\andeqn \|[u'',\,v'']\|<\dt.
\eneq
We also have
\beq\label{exl-7+1}
\|v'u'(v')^*(u')^*-p\|<1+1/2\andeqn
\|v''u''(v'')^*(u'')^*-(1-p)\|<1+1/2.
\eneq

Write $B=\oplus_{j=1}^r C_j,$ where $C_j=C([0,1],M_{l(j)}),$ or
$C_j=M_{l(j)}$ and
\beq\label{exl-8}
u'=\oplus_{j=1}^r u'(j)\andeqn v'=\oplus_{j=1}^rv'(j),
\eneq
where $u'(j), v'(j)\in C_j$ are unitaries. If $C_j=M_{l(j)},$ by
Exel's trace formula (\cite{Ex}), one has
\beq\label{exl-9}
Tr_j(\text{bott}_1(u'(j), v'(j)))={1\over{2\pi i}}
Tr_j(\log(v'(j)u'(j)v(j)^*u(j)^*))
\eneq
where $Tr_j$ is the standard trace on $M_{l(j)}.$
 If
$C_j=C([0,1], M_{l(j)}),$ define $\pi_s: C_j\to M_{l(j)}$ by the
point-evaluation at $s\in [0,1].$ Then we have
\beq\label{exl-9+1}
Tr_j(\text{bott}_1(\pi_s(u'(j)), \pi_s(v'(j))))={1\over{2\pi i}}
Tr_j(\log(\pi_s(v'(j)u'(j)v(j)^*u(j)^*)))
\eneq
for all $ s\in [0,1].$ It follows that, for any tracial state
$t\in T(C_j),$
\beq\label{exl-9+2}
t(\text{bott}_1(u'(j), v'(j)))={1\over{2\pi i}}
t(\log(v'(j)u'(j)v(j)^*u(j)^*)).
\eneq

Suppose that $\tau\in T(A).$ Then there are $\lambda_j\ge 0$ such
that \beq\label{exl-10} \sum_{j=1}^r {\lambda_j\over{l(j)}}=1\andeqn
\tau|_{B}=\sum_{j=1}^r{\lambda_j\over{l(j)}}t_j, \eneq where $t_j\in
T(C_j),$ $j=1,2,...,r.$
 It follows that
\beq\label{exl-11}
\tau(\text{bott}_1(u',v'))={1\over{2\pi
i}}\tau(\log(v'u'(v')^*(u')^*)).
\eneq

We also have
\beq\label{exl-12}
&&\hspace{-0.3in}\tau(\log((v'+v'')(u'+u'')(v'+v'')^*(u'+u'')^*)))\\
&&= \tau(\log(v'u'(v')^*(u')^*))+\tau(\log(v''u''(v'')^*(u'')^*)).
\eneq
Note that
\beq\label{exl-13}
&&|{1\over{2\pi
i}}\tau(\log(v''u''(v'')^*(u'')^*))|<\tau(1-p)<\ep\\
&&\hspace{-0.2in}\andeqn \tau(\text{bott}_1(u'', v''))<\ep
\eneq
for all $\tau\in T(A).$ It follows that
\beq\label{exl-14}
|\rho_A(\text{bott}(u,v))(\tau)-{1\over{2\pi
i}}\tau(\log((v'+v'')(u'+u'')(v'+v'')^*(u'+u'')^*)))| <2\ep
\eneq
for all $\tau\in T(A).$

By(\ref{exl-4}), we have
\beq\label{exl-14+1}
\|vuv^*u^*-(v'+v'')(u'+u'')(v'+v'')^*(u'+u'')^*\|<4\ep_1.
\eneq
Thus, by the choice of $\ep_1$ and  by (\ref{exl-4}),
\beq\label{exl-15}
|{1\over{2\pi
i}}[\tau(\log((v'+v'')(u'+u'')(v'+v'')^*(u'+u'')^*))-
\tau(\log(vuv^*u))]|<\ep
\eneq
for all $\tau\in T(A).$ Thus, by (\ref{exl-15}) and
(\ref{exl-14}),
\beq\label{exl-16}
|\rho_A(\text{bott}_1(u,v))(\tau)-{1\over{2\pi
i}}\tau(\log(vuv^*u))|<3\ep
\eneq
for all $\tau\in T(A)$ and for all $\ep.$ Let $\ep\to 0,$ we
obtain
\beq\label{exl-17}
\rho_A(\text{bott}_1(u,v)={1\over{2\pi i}}\tau(\log(vuv^*u))
\eneq
for all $\tau\in T(A).$

\end{proof}

\section{Asymptotic unitary equivalence}

\begin{lem}\label{NecL}
Let $A$ be a separable \CA\,  and let  $B$ be  a  unital \CA.
Suppose that $\phi_1, \phi_2: A\to B$ are two unital \hm s such that
there is a continuous path of unitaries $\{u(t): t\in
[0,\infty)\}\subset B$ such that \beq\label{NecL1}
\lim_{t\to\infty}{\rm ad}\, u(t)\circ \phi_1(a)=\phi_2(a)\rforal
a\in A. \eneq Then there is a continuous piecewisely  smooth path of
unitaries $\{w(t): t\in [0,\infty)\}\subset B$ such that
\beq\label{NecL2} \lim_{t\to\infty}{\rm ad}\, v(t)\circ
\phi_1(a)=\phi_2(a)\rforal a\in A. \eneq
\end{lem}

\begin{proof}
For each integer $n\ge 1,$ since $v(t)$ is continuous in $[n-1,
n],$ there are points $\{t_{i,n}: i=0,1,...,l(n)$ such that
$n-1=t_{0,n}<t_{1,n}<\cdots t_{l(n),n}=n+1$ and
$$
\|u(t_{i,n})-u(t_{i-1, n})\|<{1\over{2^{n+5}}},\,\,\,i=1,2,...,l(n).
$$
Therefore there is, for each $i,$ a selfadjoint element
$h_{i,n}\in B$ with $\|h_{i,n}\|\le {1\over{2^{n+3}}}$ such that
$$
u(t_i)=u(t_{i-1})\exp(\sqrt{-1} h_{i,n}),\,\,\, i=1,2,...,l(n).
$$
Define
$$
v(t)=u(t_{i-1})\exp(\sqrt{-1}
({t-t_{i-1,n}\over{t_{i,n}-t_{i-1,n}}})h_{i,n})\tforal t\in
[t_{i-1,n}, t_{i, n}),
$$
$i=1,2,...,l(n)$ and $n=1,2,....$ Note that, for any $c\in C,$
\beq
\hspace{-0.8in}\|v(t)^*\phi_1(c)v(t)-u(t_{i,n})^*\phi_1(c)u(t_{i,n})\|\le
\|(v(t)^*-u(t_{i,n})^*)\phi_1(c)v(t)\|+\\\|u(t_{i,n})^*\phi_1(c)(v(t)-u(t_{i,n}))\|<
({1\over{2^{n+1}}}+{1\over{2^{n+1}}})\|\phi_1(c)\|={1\over{2^n}}\|\phi_1(c)\|
\eneq
for all $t\in [t_{i-1,n}, t_{i,n}),$ $i=1,2,...,l(n)$ and
$n=1,2,....$ It follows that
$$
\lim_{t\to\infty}{\rm ad}\, v(t)\circ \phi_1(c)=\phi_2(c)\tforal
c\in C.
$$
Note that $v(t)$ is continuous and piecewisely  smooth.

\end{proof}

\begin{thm}\label{NecT}
Let $A\in {\cal N}$ be a unital \CA\, and let $B$ be a unital
separable  \CA. Suppose that $\phi_1, \phi_2: A\to B$ are unital
monomorphisms such that \beq\label{NecM1} \lim_{t\to\infty}{\rm
ad}\, u(t)\circ \phi_1(a)=\phi_2(a)\tforal a\in A \eneq for some
continuous and piecewisely smooth path of unitaries $\{u(t): t\in
[0, \infty)\}\subset B.$  Then \beq\label{NecM2}
[\phi_1]=[\phi_2],\,\,\,\phi_1^{\ddag}=\phi_2^{\ddag},\,\,\,
(\phi_1)_T=(\phi_2)_T\andeqn
\overline{R}_{\phi_1, \phi_2}=0. \eneq
\end{thm}

\begin{proof}
It is clear that $\phi_1^{\dag}=\phi_2^{\dag},$ if $\phi_1$ and
$\phi_2$ are asymptotically unitarily equivalent. Thus this
theorem follows from Theorem 4.3 of \cite{Lnasym}.

\end{proof}

\begin{cor}\label{NecC}
Let $A\in {\cal N}$ and let $B$ be a unital \CA. Suppose that
$\phi_1, \phi_2: A\to B$ are two unital monomorphisms which are
asymptotically unitarily equivalent. Then
\beq\label{Necc}
[\phi_1]=[\phi_2]\,\,\,\text{in}\,\,\,KK(A,B), \phi_1^{\ddag}=\phi_2^{\ddag},\,
(\phi_1)_T=(\phi_2)_T\andeqn \overline{R}_{\phi_1, \phi_2}=\{0\}.
\eneq
\end{cor}

\begin{proof}
This follows from \ref{NecT} and \ref{NecL} immediately.

\end{proof}

\section{An existence theorem}

%

\begin{lem}\label{meas}
Let $C$ be a unital separable exact simple \CA\, with real rank zero, stable rank one and weakly unperforated $K_0(C).$
 Let $A$ be a unital simple \CA\, of stable rank one.
Suppose that there is a \hm\, $\kappa: K_0(C)\to K_0(A)$ for which
$\kappa([1_C])=[1_A]$ and $\kappa(K_0(C)_+\setminus\{0\})\subset
K_0(A)_+\setminus\{0\}.$ Then $\kappa$ induces a positive linear
map $\Lambda: C_{s.a}\to \text{Aff}(T(A))$ such that
$$
\Lambda(p)(\tau)=\tau(\kappa([p]))
$$
for any projection $p\in C.$ Moreover, suppose that there is a
sequence of unital completely positive linear maps $L_n: C\to A$
such that
\beq
\lim_{n\to\infty}\|L_n(a)L_n(b)-L_n(ab)\|=0\tand 
\lim_{n\to\infty}\sup_{\tau\in T(A)}|\rho_A( [L_n](x)-\kappa(x))(\tau)|=0
\eneq
for all $a, b\in C$ and for all $x\in K_0(C).$  Then
$$
\lim_{n\to\infty}\sup\{|\tau\circ L_n(f)-\Lambda(f)(\tau)|:\tau\in
T(A)\}=0\tforal f\in C.
$$
Moreover,  if $f\in C_+\setminus \{0\},$ $\Lambda(f)(\tau)>0$ for
all $\tau\in T(A).$

\end{lem}

\begin{proof}
It follows from \cite{BR} that there is
a quasi-trace $\Lambda(\tau)$ on $C$ such that
$$
\Lambda(\tau)(p)=\kappa([p])(\tau)
$$
for each projection $p \in C$ and for each 
$\tau\in T(A).$ It follows from the  Haagerup theorem
that $\Lambda(\tau)$ is a tracial state on $C.$
  Thus $\kappa$
gives a positive linear map $\Lambda: C_{s.a}\to {\text{Aff}}(T(A))$
such that
$$
\Lambda(p)(\tau)=\tau(\kappa([p]))
$$
for all projections $p\in A.$ Moreover, for each $f\in C_{s.a},$
there is  a sequence of sets of  finitely many mutually orthogonal
projections $p_{1,n}, p_{2, n},...,p_{k(n),n}\subset  C$ and real
numbers $\lambda_{1, n}, \lambda_{2,n},...,\lambda_{k,n}$ such that
\beq\label{meas2} \lim_{n\to\infty}\|f-\sum_{i=1}^{k(n)}
\lambda_{i,n}p_{i,n}\|=0. \eneq Note that
$$
\lim_{m\to\infty}\sup\{|\tau\circ L_m(\sum_{i=1}^{k(n)}
\lambda_{i,n}p_{i,n})-\sum_{i=1}^{k(n)}\lambda_{i,n}\tau(\kappa([p_{i,n}])|:
\tau\in T(A)\}=0.
$$
 It follows from (\ref{meas2}) that
$$
\lim_{n\to\infty}\sup\{|\tau(L_n(f))-\Lambda(f)(\tau)|:\tau\in
T(A)\}=0\tforal f\in C_{s.a.}.
$$
 For any $f\in C_+$ with $\|f\|=4/3,$  since $C$ has
real rank zero, there is a projection $p\in C$ such that
$
f\ge p.
$
Thus $\Lambda(f)(\tau)\ge \tau\circ \kappa([p])>0$ for all
$\tau\in T(A).$

\end{proof}

Let $A$ and let $C$ be unital \CA s. Let $T=N\times R:
C_+\setminus \{0\}\to \N\times \R$ be a map and let ${\cal
H}\subset C_+\setminus \{0\}$ be a subset.  Recall that a map $L:
C\to A$ is said to be ${\cal H}\times T$-full, if, for any $a\in
{\cal H},$ there are $N(a)$ elements $x_1,x_2,...,x_{N(a)}$ with
$\|x_i\|\le R(a)$ such that
$$
\sum_{j=1}^{N(a)} x_i^*L(a)x_i=1_A.
$$
\vspace{0.1in}

\begin{lem}\label{2meas}
Let $C$ be a unital separable exact simple \CA\, with real rank
zero, stable rank one and weakly unperforated $K_0(C).$
 Let $A$ be a unital simple \CA\, with $TR(A)\le 1. $
Suppose that there is a \hm\, $\kappa: K_0(C)\to K_0(A)$ for which
$\kappa([1_C])=[1_A]$ and $\kappa(K_0(C)_+\setminus\{0\})\subset
K_0(A)_+\setminus\{0\}$ and suppose that there is a sequence of
unital completely positive linear maps $L_n: C\to A$ such that
\beq\label{2meas1}
\lim_{n\to\infty}\|L_n(a)L_n(b)-L_n(ab)\|=0\tforal a,\, b\in C
\tand\\\label{2meas2}
\lim_{n\to\infty}\sup\{|\tau([L_n](x)-\kappa(x))|: \tau\in
T(A)\}=0\tforal x\in K_0(C).
\eneq

Then there is $T=N\times R: C_+\setminus\{0\}\to \N\times \R$
such that, for any finite subset ${\cal H}\subset
C_+\setminus\{0\},$ there exists an integer $n_0\ge 1$ such that
$L_n$ is ${\cal H}$-$T$-full for all $n\ge n_0.$

\end{lem}

\begin{proof}
Let $\Lambda$ be given by \ref{meas}. For any $a\in
D_+\setminus\{0\},$ let $a_1=a/\|a\|.$ Let $a_0=f_0(a_1)$ where
$f_0\in C([0,1]),$ $f_0(t)=0$ if $t\in [0,3/4],$ $f_0(1)=1$ and
$0\le f_0(t)\le 1.$ Define $f_1\in C_0((0,1])$ such that $0\le
f_1(t)\le 1,$ $f(t)=0$ for $t\in [0,1/2],$ $f(t)=1$ for $t\in
[3/4,1].$ Since $C$ has real rank zero, there is non-zero
projection $p_0\in \overline{a_0Ca_0}.$
Put
$$
d(a)=\inf\{\Lambda(p_0)(\tau): \tau\in T(A)\}.
$$
Then $d(a)>0.$ Since $A$ has stable rank one, there is a
projection $p'\in A$ such that $[p']=\kappa([p_0])$ (since
$\kappa([p_0])\le [1_A]$).  We have that
\beq\label{2meas-1}
\tau(p')\ge (3/4)d(a)\tforal \tau\in T(A).
\eneq
Let
\beq\label{2meas-2}
N(a)=[{1\over{(3/4)d(a)}}]+1\andeqn R(a)=\max\{1/\|a\|,1\}.
\eneq
Note that, since $TR(A)\le 1,$ for any projection $p'\in A$ with
$\tau(p')\ge (3/4)d(a)\tforal \tau\in T(A),$ there are partial
isometries $v_1,v_2,...,v_m\in M_m(A)$ with $m\le N(a)$ such that
$
\sum_{i=1}^m v_i^*p'v_i\ge 1_A.
$
Let $w\in M_m(A)$ be a partial isometry such that
$$
w^*w=1\andeqn ww^*=q\le \sum_{i=1}^mv_i^*p'v_i.
$$
It follows that
\beq
\sum_{i=1}^m w^*v_i^*p'v_iw&=& w^*qw=1_A.
\eneq
Let $x_i=w^*v_i,$ $i=1,2,...,m.$ Note that $x_i\in A$ and
$\|x_i\|\le 1.$ Define $T(a)=(N(a), 2\sqrt{2}R(a)^{1/2})$ for all
$a\in D_+\setminus\{0\}.$ Fix a finite subset ${\cal H}\subset
D_+\setminus\{0\}.$
Choose $n_0'\ge 1$ such that
$$
L_n(p_0)\ge (15/16)\Lambda(p)(\tau)\tforal \tau\in T(A)
$$
for all $a\in {\cal H}$ and for all $n\ge n_0'.$ Put
$b=L_n(f_1(a_1)).$ By (\ref{2meas1}) and Section 2.5 of
\cite{Lnbk} , there is a projection $p\in A$ such that, with $n\ge
n_0\ge n_0',$
\beq\label{2meas-3}
\|L_n(p_0)-p\|<1/32\andeqn \|bp-p\|<1/16.
\eneq
From what we have shown above, there are $z_1,z_2,...,z_{N(a)}\in
A$ with $\|z_i\|\le 1$ such that
$$
\sum_{i=1}^{N(a)}z_i^*pz_i=1_A.
$$
It follows from (\ref{2meas-3}) (and Section 2.5 of \cite{Lnbk},
for example) that there is $0\le b_1\le 2$ such that
$
b_1bb_1=p.
$
Hence
$$
\sum_{i=1}^{N(a)}z_i^*b_1bb_1z_i=1_A.
$$
Since $2R(a)a\ge f_1(a_1),$ we have
$$
\sum_{i=1}^{N(a)} z_i^*2R(a)L_n(a)z_i\ge 1_A.
$$
It follows that there are $y_1,y_2,...,y_{N(a)}$ with $\|y_i\|\le
2\sqrt{2}R(a)^{1/2}$ such that
$$
\sum_{i=1}^{N(a)}y_i^*L_n(a)y_i=1_A.
$$
It follows that $L_n$ is ${\cal H}$-$T$-full for all $n\ge n_0.$

\end{proof}


%
%
%
%

\begin{df}\label{Dc0}
{\rm  Denote by ${\cal C}_{00}$ the class of \CA s with the form
$PM_l(C(X))P,$ where $X=\T\vee \T\vee\cdots \vee Y,$ where $Y$ is
a connected finite simplicial complex with dimension no more than
three and with torsion $K_1(C(Y)),$ $P\in M_l(C(X))$ is a
projection.  In particular, $X\in {\bf X}$ as defined in 8.2 of
\cite{Lnn1}. Denote by ${\cal C}_0$ the class of \CA s in ${\cal
C}_{00}$ so that the projection $P$ has rank at least 6, if
$X\not=\T\vee \T\vee\cdots \vee \T.$ In particular,
$K_1(C)=U(C)/U_0(C)$ for $C\in{\cal  C}_0.$ Denote by ${\cal C}$ the
class of AH-algebras which have the form $C=\lim_{n\to\infty}(C_n,
\imath_n),$ where each $C_n$ is a finite direct sum of \CA s in
${\cal C}_0$ and $\imath_n$ is a unital monomorphism.
If $C\in {\cal C},$ then  one has
splitting short exact sequence
\beq\label{UCad}
U_0(C)/CU(C)\to U(C)/CU(C)\to K_1(C)\to 0
\eneq
and  $\Delta_C$ is an isometric and isomorphism (see
\ref{DDCU}).


}

\end{df}

\begin{thm}\label{hitK1}
Let $C$ be a unital separable simple amenable \CA\, with $TR(C)=0$
which satisfies the UCT and let $A$  be a unital separable simple
\CA\, with $TR(A)\le 1.$ Suppose that $\kappa\in KL(C,A)^{++}$
with $\kappa([1_C])=[1_A].$ Then there is a unital \hm\, $h: C\to
A$ such that
$$
[h]=\kappa\,\,\, {\rm in}\,\,\, KL(C,A).
$$
\end{thm}


\begin{proof}
It follows from \cite{Lnduke} that $C$ is a unital simple
AH-algebra. Moreover, it can be written as $C=\lim_{n\to\infty}
(C_n, \phi_{n, n+1}),$ where each $C_n$ is a finite direct sum of
\SCA s in ${\cal C}_0$ (this follows from a theorem of Villadsen
and Theorem 10.9 and Theorem 10.10 of \cite{Lnctr1}, see also
\cite{EGL}). Furthermore, by \cite{EGLb}, we may assume that
$\phi_{n, n+1}$ are injective.
Let $\{{\cal F}_n\}$ be an increasing sequence of finite subsets
of $C$ whose union is dense in $C.$ We may assume that there is a
finite subset ${\cal F}_n'\subset C_n$ such that $\phi_{n,
\infty}({\cal F}_n')={\cal F}_n,$ $ n=1,2,....$ Let $\{{\cal
P}_n\}$ be an increasing sequence of finite subsets of
$\underline{K}(C).$ By 9.10 of \cite{Lnctr1}, there exists a
sequence of unital completely positive linear maps $L_n: C\to A$
such that
\beq\label{hitK-1}
\lim_{n\to\infty} \|L_n(a)L_n(b)-L_n(ab)\|=0 \tforal a,\, b\in
C\andeqn [L_n]|_{{\cal P}_n}=\kappa|_{{\cal P}_n},
\eneq
$n=1,2,....$ Let $\Lambda$ be in \ref{meas} associated with
$\kappa$ and $T=N\times R: C_+\setminus\{0\}\to \N\times \R$ be
given by $\kappa$ as in \ref{2meas}.

Let $\dt_n$ (in place of $\dt$), ${\cal G}_n\subset C$ (in place
of ${\cal G}$), ${\cal P}_n'\subset \underline{K}(C)$ (in place of
${\cal P}$) and ${\cal U}_n\subset U(M_{\infty}(C))$ be as
required by Theorem 11.5 of \cite{Lnn1} for $\ep_n=1/2^{n+1},$
${\cal F}_n$ and $T.$  Since $C$ is of stable rank one, we only
need to consider the case that  ${\cal U}_n\subset U(C).$ To
simplify the notation, we may assume that ${\cal U}_n\subset
U(C).$ By passing  to a subsequence of $\{{\cal P}_n\},$ we may
assume, without loss of generality, that ${\cal P}_n={\cal P}_n',$
$n=1,2,....$ Note as in (\ref{UCad}) we may write
$$
U(C)/CU(C)=U_0(C)/CU(C)\oplus K_1(C).
$$
Let $\pi_1: U(C)/CU(C)\to U_0(C)/CU(C)$ and $\pi_2: U(C)/CU(C)\to
K_1(C)$ be the projection. Put $\overline{\cal
U}_n^{(1)}=\pi_1(\overline{\cal U}_n)$ and $\overline{\cal
U}_n^{(2)}=\pi_2(\overline{\cal U}_n),$ $n=1,2,....$ To simplify
notation, without loss of generality, we may assume that
$\overline{\cal U}_n=\overline{\cal U}_n^{(1)}\cup\overline{\cal
U}_n^{(2)},$ $n=1,2,....$ Let ${\cal S}_n\subset C_{s.a.}$ be a
finite subset such that $\Delta_C(\overline{\cal
U}_n^{(1)})\subset \overline{\widehat{{\cal J}_n}},$ where $\widehat{{\cal
J}_n}$ is the image of ${\cal J}_n$ in $\text{Aff}(T(C))$ and $\overline{\widehat{{\cal J}_n}}$ is the image 
in $\text{Aff}(C)/\overline{\rho_C(C)},$
$n=1,2,....$ Put ${\cal G}_n'={\cal G}_n\cup {\cal J}_n,$ $n=1,2,....$

Without loss of generality, we may also assume that there is a
finite subset ${\cal Q}_n\in \underline{K}(C_n)$ such that
$[\phi_{n, \infty}]({\cal Q}_n)={\cal P}_n,$ a finite subset
${\cal G}_n^{(0)}\subset C_n$ such that $\phi_{n, \infty}({\cal
G}_n^{(0)})={\cal G}_n'$ and ${\cal U}_n\subset \phi_{n,
\infty}(U(C_n)),$ $n=1,2,....$ Let ${\cal V}_n\subset U(C_n)$ be a
finite subset such that $\overline{\phi_n({\cal
V}_n)}=\overline{{\cal U}_n}^{(2)}$ and $\kappa_1^{(n)}:
G(\overline{{\cal V}_n})\to K_1(C_n)$ is injective, where $G({\cal
V}_n)$ is the subgroup of $U(C_n)/CU(C_n)$ generated
${\overline{{\cal V}_n}}$ and where $\kappa_1^{(n)}:
U(C_n)/CU(C_n)\to K_1(C_n)$ is the quotient map, $n=1,2,....$

We may assume, by passing to a subsequence, that
\beq\label{NHIT+}
\sup\{|\tau(L_n(a))-\Lambda(a))(\tau)|: \tau\in T(A)\}<\dt_n/32\pi
\eneq
for $a=(1/2)(c^*+c)$ and $a=(1/2i)(c-c^*)$ for all $c\in {\cal
G}_k'.$

We claim, that,  there is a subsequence $\{n(k)\}$
such that
\beq\label{nhik1}
L_{n(k+1)}'&\approx_{1/2^{n+1}}& L_{n(k)}'\,\,\,{\rm on}\,\,\,
{\cal F}_k,\\\label{nhit2}
 [L_{n(k)}']|_{{\cal
P}_k}&=&[L_{n(k)}]|_{{\cal P}_k}\andeqn\\\label{nhit2+}
\sup\{|\tau(L_{n(k)}'(a))-\tau(L_{n(k)}(a))|:\tau\in
T(A)\}&<&\dt_k/8\pi
\eneq
for all $a\in {\cal G}_k',$ $k=1,2,...,$ where  
$L_{n(k)}': C\to A$ is a $\dt_{k}$-${\cal
G}_{k}'$-multiplicative \morp. We also assume that  $[L_{n(k)}']|_{{\cal
P}_k}$ is well defined and $L_{n(k)}'$ is ${\cal H}_k$-$T$-full,
where ${\cal H}_k={\cal G}_k'\cap C_+\setminus\{0\},$ 
$L_{n(k)}'\circ \phi_{k, \infty}$ induces  $\af_k':
K_1(C_k)\to U(A)/CU(A),$ where we write $U(C_k)/CU(C_k)=U_0(C_k)/CU(C_k)\oplus K_1(C_k)$ and define 
$\af_k: U(C_k)/CU(C_k)\to U(A)/CU(A)$ so that $\af_k|_{K_1(C_k)}=\af_k'$ and $\af_k|_{U_0(C_k)/CU(C_k)}=0.$
Let
${\tilde \dt}_k$ (in place of $\dt$), ${\tilde G}_k$ (in place of
${\cal G}$) and $\eta_k$ (in place of $\eta$) be required by Lemma
7.4 of \cite{Lnctr1} for $C_k$ ( in place of $C$), $\dt_k/2\pi$
(in place of $\ep$) and for $\phi_{k-1, k}({\cal V}_{k-1})$ (in
place of ${\cal U}$). We may assume that ${\tilde \dt}_k<\dt_k/2,$
$k=1,2,....$

We use induction. Suppose that we have constructed
$n(1)<n(2)<...,n(k)$
so that (\ref{nhik1}), (\ref{nhit2}) and (\ref{nhit2+}) hold.

  Since
 $TR(C)=0,$
  we may assume that
 there is a projection $p_k\in C$  and a unital \hm\, $\psi_k: C_k\to B$ for some finite
  dimensional \SCA\, $B$ of $C$ with
 $1_C=p_k$ such that
 \beq\label{nhitK-12}
\|\phi_{k,\infty}(f)-\psi_k\oplus (1-p_k)\phi_{k,
\infty}(f)(1-p_k)\|&<&\min\{\dt_{k+1}/8\pi, \eta_{k+1}/4, 1/2^{n+2}\}\\
\andeqn
\|p_k\phi_{k,\infty}(f)-\phi_{k,\infty}(f)p_k\|&<&\min\{\dt_{k+1}/8\pi,
\eta_{k+1}/4, 1/2^{n+2}\} \eneq for all $f\in {\cal G}_k'\cup
\tilde{ G}_k$ and \beq\label{nhit3} \tau(p_k)<\min\{{\tilde
\dt}_{k+1}/4, \dt_{k+1}/4\}\tforal \tau\in T(C). \eneq

To simplify notation, we may assume that $L_n(p_k)$ is a
projection for all $n\ge n(k)+1,$ without loss of generality.

By \ref{2meas}, we may assume, without loss of generality, that $L_n$ are ${\cal
H}_{k+1}$-$T$-full for all $n\ge N(k)$ for some $N(k)\ge 1,$ where
${\cal H}_{k+1}={\cal G}_{k+1}'\cap C_+\setminus\{0\}.$  We may
also assume that $L_n$ is $\min\{\dt_{k+1}, \eta_k\}$-${\cal
S}_{k+1}$-multiplicative, where ${\cal S}_{k+1}\supset {\cal
G}_{k+1}'\cup (\psi_k+(1-p_k)\phi_{k, \infty})({\cal
G}_k^{(0)}\cup {\tilde G}_k).$ We may also assume that
$N(k)>n(k).$  Without loss of generality, we may further assume
that $L_{N(k)}(p_k)$ is a projection.

It follows from Lemma 7.4 of \cite{Lnctr1} that there is a unital
\hm\, $\Psi_{k+1}: C_{k+1}\to L_{N(k)}(p_k)AL_{N(k)}(p_k)$ such
that $\Psi_{k+1}$ is homotopically trivial,
$(\Psi_{k+1}){*0}=[L_{N(k)}\circ \psi_{k}]_{*0}$ and
\beq\label{nhit4}
\alpha_k({\bar w})^{-1}(\Psi_{k+1}\oplus
(1-q_k)L_{N(k)}(1-q_k))^{\ddag}({\bar w})=\overline{g_w}
\eneq
for all $w\in \phi_{k, k+1}({\cal V}_k),$ where $g_w\in U_0(A),$
${\rm cel}(g_w)<\dt_k/2\pi$ and where $q_k=L_{N(k)}(p_k).$ Let
$n(k+1)=N(k).$ Define $L_{n(k+1)}''=\Psi_k\oplus
(1-q_k)L_{N(k)}(1-q_k).$ We have
\beq\label{nhit6}
[L_{n(k+1)}'']|_{{\cal P}_k}=[L_{n(k+1)}]|_{{\cal P}_k}.
\eneq
 Note, by \ref{2meas},  we may assume that $L''_{n(k+1)}$ is also
 ${\cal H}_{k+1}$-$T$-full.
 By, (\ref{nhit3}),
 $$
 |\tau\circ L_{n(k+1)}''(a)-\tau\circ
 L_{n(k+1)}(a)|<\dt_{k+1}/4\pi \tforal \tau\in T(A)
 $$
 and $a\in {\cal G}_{k+1}'.$  
 In particular,
by (\ref{NHIT+}),
$$
|\tau\circ L_{n(k+1)}''(a)-\tau\circ L_{n(k)}'(a)|<\dt_k/2\pi \tforal
\tau\in T(A)
$$
and $a\in {\cal G}_k'.$ Since ${\cal J}_k\subset {\cal G}_k'$ (see also \ref{DDCU}),
one also has 
\beq\label{nhit5-}
{\rm dist}((L_{n(k+1)}'')^{\ddag}({\bar u}), (L_{n(k)}')^{\ddag}({\bar u}))<\dt_k
\eneq
for all ${\bar u}\in \overline{{\cal U}_k^{(1)}}.$ By combining with (\ref{nhit4}), one sees that 
(\ref{nhit5-}) holds for all $u\in {\cal U}_k.$

By applying Theorem 11.5 of \cite{Lnn1}, we obtain a unitary
$u_k\in A$ such that
\beq\label{nhi5}
{\rm ad}\, u_k\circ L_{n(k+1)}''\approx_{1/2^{k+1}}
L_{n(k)}'\,\,\, {\rm on}\,\,\, {\cal F}_k.
\eneq

Moreover, it is clear that, by choosing larger $N(k),$ if
necessarily,  we assume that
 $L_{n(k+1)}''\circ \phi_{k+1, \infty}$ induces
a \hm\,  $\af_{k+1}':K_1(C_{k+1})\to U(A)/CU(A).$ We then define $\af_{k+1}$ accordingly. 
 Define
$L_{n(k+1)}'={\rm ad}\, u_k\circ L_{n(k+1)}''.$  This proves the
claim.

It follows that $\{L'_{n(k)}\}$ is a Cauchy sequence. One obtains
a unital linear map $h: C\to A$ such that
$$
\lim_{k\to\infty}L'_{n(k)}(c)=h(c)\tforal c\in C.
$$
Therefore $h$ must be a unital \hm. It follows from (\ref{nhit6})
that
$$
[h]=\kappa.
$$

\end{proof}

\section{Almost commuting with a unitary}

\begin{lem}\label{Ql}
Let $C$ be a unital AH-algebra and let $A$ be a unital separable
simple \CA\, with $TR(A)=0.$ Suppose that there is a unital
monomorphism $h: C\to A.$ Then, for any $\ep>0,$ any finite subset
${\cal F}$ and any finite subset ${\cal P}\subset
\underline{K}(C),$ there exists  a \SCA\, $C_0\subset C$ with
${\cal P}\subset [\imath](\underline{K}(C_0)),$ where $\imath:
C_0\to C$ is the embedding, and a finite subset ${\cal Q}\subset
K_1(C_0)$ and $\dt>0$ satisfying the following: Suppose that
$\kappa\in {\rm Hom}_{\Lambda}(\underline{K}(C_0\otimes C(\T)),
\underline{K}(A))$ with
\beq\label{Ql1}
|\rho_A\circ \kappa(\boldsymbol{\bt}(x))(\tau)|<\dt\tforal x\in
{\cal Q}\andeqn\tforal \tau\in T(A).
\eneq
Then there exists a unitary $u\in U(A)$ such that
$$
\|[h(c), \,u]\|<\ep\tforal c\in {\cal F}\andeqn {\rm Bott}(h\circ
\imath, u)=\kappa\circ [\imath]\circ \boldsymbol{\bt}.
$$

Moreover, there is a sequence of \CA s $C_n$ with the form
$C_n=P_nM_{r(n)}(C(X_n))P_n,$ where $X_n$ is a finite CW complex
and $P_n\in M_{r(n)}(C(X_n))$ is a projection, such that
$C=\lim_{n\to\infty}(C_n, \phi_n)$ for a sequence of unital
monomorphisms $\phi_n: C_n\to C_{n+1}$ and one may choose
$C_0=\phi_{n,\infty}(C_n)$ for some integer $n\ge 1.$

\end{lem}

\begin{proof}
Suppose that $\kappa\in {\rm Hom}_{\Lambda}(\underline{K}(C\otimes
C(\T)), \underline{K}(A)).$ Note (see 2.10 of \cite{Lnhomp}) that
$$
\underline{K}(C\otimes C(\T))=\underline{K}(C)\oplus
{\boldsymbol{\bt}}( \underline{K}(C)).
$$
Define $\kappa'\in {\rm Hom}_{\Lambda}(\underline{K}(C\otimes C(\T)),
\underline{K}(A))$ as follows.
\beq
\kappa'|_{\underline{K}(C)}=[h]\andeqn
\kappa'|_{{\boldsymbol{\bt}}(
\underline{K}(C))}=\kappa|_{{\boldsymbol{\bt}}(
\underline{K}(C))}.
\eneq
We see then the lemma follows immediately from Lemma 7.5 of
\cite{Lnasym} by considering $\kappa'$ instead of $\kappa.$

\end{proof}

\begin{lem}\label{small}
Let $X$ be a finite CW complex and let  $C=PM_r(C(X))P,$ where
$P\in M_r(C(X))$ is a projection and $r\ge 1.$ Let $A$ be a unital
separable simple \CA\, with $TR(A)=0$ and let $\kappa\in KK_e(C,
A)^{++}.$ Suppose that $p\in A$ is a non-zero projection. Then
there is a non-zero projection $p_0\le p,$ a unital monomorphism
$h_0: C\to p_0Ap_0$ and a unital \hm\, $h_1: C\to F$ such that
$$
[h_0+h_1]=\kappa,
$$
where  $F$ is a finite dimensional \SCA\, of $A$ with $1_F=1-p_0.$

\end{lem}

\begin{proof}
First we assume that $C=C(X).$ It is then clear that we may assume
that $X$ is a connected. In this case  the lemma follows from
Lemma 4.2 of \cite{Lnhomp}. For general case, there is $k\ge 1$
such that there is a projection $Q\in M_k(C)$ such that
$QM_k(C)Q\cong M_R(C(X))$ for some integer $R\ge 1.$  Moreover,
there is an integer $R'\ge 1$ and a projection $Q_1\in
M_{R'}(QM_k(C)Q)$ such that $Q_1M_{R'}(QM_k(C)Q)Q_1=M_{R_1}(C)$ for
some integer $R_1\ge 1.$ Then there is a projection $q\in M_k(A)$
such that there are mutually orthogonal and mutually equivalent
projections $q_1,q_2,..,q_R\in M_k(A)$ such that
$\kappa([Q])=[q_1].$ Let $C_0=e_1 M_R(C(X))e_1\cong C(X),$ where
$e_1$ is a rank one projection. Let $e_0\in M_k(A)$ such that
$\kappa([e_1])=[e_0].$ Let $q_0\in e_0Ae_0$ be a non-zero
projection such that
$$
k[q_0]\le [p].
$$

It follows from the special case that there is a unital
monomorphism $h_0: C_0\to h_0(e_1)M_k(A)h_0(e_1)$ such that
$$
[h_0']=\kappa
$$
and $h_0'=h_{0,0}+h_{0,1},$ where $h_{0,0}$ is a unital
monomorphism from $C_0$ to $q_0M_k(A)q_0$ and $h_{0,1}$ maps $C_0$
into a finite dimensional \SCA\, in $(e_0-q_0)M_k(A)(e_0-q_0).$
Let $h_1'=h_0'\otimes {\rm id}_{M_{R'} }$ and
$h_2=(h_1')|_{Q_1M_{R'}(QM_k(C)Q)Q_1}.$ Let
$h'=h_2|_{e_1'M_{R_1}(C)e_1'},$ where $e_1'\in M_{R_1}(C)$ is a
rank one projection. Then we obtain a unital monomorphism $h: C\to
A$ by choosing $h={\rm ad}\, u\circ h'$ for some unitary $u\in
h_2(M_{R_1}(C)).$  By choosing a right unitary,  we may write
$h=h_0+h_1,$ where $h_0$ maps $C$ into $p_0Ap_0$ with $p_0\le p$
and $h_1: C\to (1-p_0)A(1-p_0)$ with finite dimensional range.

\end{proof}

\begin{thm}\label{HITK}
Let $C$ be a unital \CA\, with the form $C=\lim_{n\to\infty}(C_n,
\imath_n),$ where $C_n$ is a finite direct sum of \SCA s in ${\cal
C}_{00}$ and $\imath_n: C_n\to C_{n+1}$ is unital and injective,
and let $A$ be a unital separable simple \CA\, $TR(A)\le 1.$
Suppose that $h: C\to A$ is a monomorphism. Then, for any $\ep>0,$
any finite subset ${\cal F}$ and any finite subset ${\cal
P}\subset \underline{K}(C),$ there exists  $n\ge 1$ for which
${\cal P}\subset [\imath_n](\underline{K}(C_n)),$  a finite subset
${\cal Q}\subset K_1(C_n)$ and $\dt>0$ satisfying the following:
Suppose that $\kappa\in {\rm Hom}_{\Lambda}(\underline{K}(C_n\otimes
C(\T)), \underline{K}(A))$ with
\beq\label{Ql1+}
|\rho_A\circ \kappa(\boldsymbol{\bt}(x)(\tau))|<\dt\tforal x\in
{\cal Q}\tand \tforal \tau\in T(A).
\eneq
Then there exists a unitary $u\in U(A)$ such that
$$
\|[h(c), \,u]\|<\ep\tforal c\in {\cal F}\andeqn {\rm Bott}(h\circ
\imath_n, u)=\kappa\circ \boldsymbol{\bt}.
$$

\end{thm}

\begin{proof}
It is clear that it suffices to prove the case that $C$ is a finite direct sum of \CA s in ${\cal C}_{00}.$ 
Therefore, to simplify notation, without loss of generality, we may assume
that $C\in {\cal C}_{00}$ and write $C=PC_r(C(X))P.$ Let $B$ be a
unital separable simple \CA\, with $TR(B)=0$ which satisfies the
UCT and
$$
(K_0(B), K_0(B)_+, [1_B], K_1(B))=(K_0(A), K_0(A)_+, [1_A],
K_1(A)).
$$
It follows from \ref{hitK1} that we may assume that there is an
embedding $\imath: B\to A$ such that $[\imath]$ (in $KL(B,A)$)
induces an identification of the above. To simplify notation, we
may further assume that $B$ is a unital \SCA\, of $A.$

Let $\ep_1>0$ with $\ep_1<\ep$ and let ${\cal F}_1\supset {\cal
F}$ be a finite subset such that, for any unital \hm\, $H: C\to A$
and unitary $u'\in A$ for which
$$
\|[H(c),\, u']\|<\ep_1\tforal c\in {\cal F}_1,
$$
 ${\rm Bott}(H,u')$ is well-defined. Moreover, if $H': C\to A$
is another unital monomorphism such that
$$
\|H(c)-H'(c)\|<\ep_1\tforal c\in {\cal F}_1,
$$
then
$$
{\rm Bott}(H, u')={\rm Bott}(H', u').
$$


Let $\dt_1>0$ (in place of $\dt$), ${\cal G}_1\subset C$ ( in
place of ${\cal H}$) be a finite subset, ${\cal P}'\subset
\underline{K}(C)$ (in place of ${\cal P}$) be a finite subset and
let ${\cal U}\subset U(M_{\infty}(C))$ be a finite subset required
by Cor. 11.6 of \cite{Lnn1} for $\ep_1/2$ (in place of $\ep$) and
${\cal F}_1$ and $h.$  Without loss of generality, we may assume that
${\cal G}_1$ is in the unit ball of $C.$

For a moment, we assume that rank of $P$ at each point of $X$ has
rank at least 6. In this case $K_1(C)=U(C)/U_0(C).$ Therefore, we
may assume that ${\cal U}\subset U(C).$
Write $U(C)/CU(C)=U_0(C)/CU(C)\oplus K_1(C)$ and $\pi_1: U(C)/CU(C)\to U_0(C)/CU(C)$ and 
$\pi_2: U(C)/CU(C)\to K_1(C)$ are projection maps, respectively. 
Let $\overline{\cal U}_1=\pi_1(\overline{\cal U})$ and $\overline{\cal U}_2=\pi_2(\overline{\cal U}).$
To simplify the notation, without loss of generality, we may assume that
$\overline{\cal U}=\overline{\cal U}_1\cup \overline{\cal U}_2.$ Let ${\cal H}_0\subset C_{s.a.}$ be a finite subset such that
$\overline{\widehat{{\cal H}_0}}\supset \Delta_C(\overline{{\cal U}_1}).$ Let ${\cal G}_2={\cal G}_1\cup {\cal H}_0.$

Let $\dt_2>0$ ( in place of $\dt$)  be as required by Lemma 7.4 of
\cite{Lnctr1} for $\dt_1/2$ (in place of $\ep$), $\overline{\cal U}_2$  and
$\alpha=h^{\ddag}.$

It follows from \ref{small} that there is a unital monomorphism
$h_0: C\to p_0Bp_0$ with $\tau(p_0)<\min\{\dt_1/8, \dt_2/4\}$ for
all $\tau\in T(B)$ and a unital \hm\, $h_1': C\to F,$ where $F$ is
a finite dimensional \SCA\, of $B$ with $1_F=(1-p_0)$ such that
\beq\label{HITK2}
[h_0+h_1']=[h]\,\,\,{\rm in}\,\,\, KK(C, A).
\eneq

Note that we have assumed that $C\in {\cal C}_0.$  Let $1>\dt_3>0$
(in place of $\dt$) and ${\cal Q}\subset K_1(C)$ be required by
Lemma \ref{Ql} for ${\cal F},$  ${\cal P}$ and $p_0Bp_0$ ( in
place of $A$).

Let $\dt=\min\{\dt_3\cdot(\dt_1/16\pi), \dt_3\cdot \dt_2/4 \}.$ Now
suppose that
$$
|\rho_A\circ \kappa\circ\boldsymbol{\bt}(x)(\tau)|<\dt\tforal x\in
{\cal Q}\andeqn \tforal \tau\in T(A).
$$
By \ref{Ql}, there is a unitary $u_0\in p_0Bp_0$ such that such
that
\beq\label{HITK3}
\|[h_0(c), \, u_0]\|<\ep_1/2\tforal c\in {\cal F}\andeqn {\rm
Bott}(h_0, u_0)= \kappa\circ \boldsymbol{\bt}.
\eneq
Put $u=u_0+(1-p_0).$ It is easy to see that there is a non-zero
projection $q_0\in (1-p_0)A(1-p_0)$ such that
\beq\label{HITK4}
q_0f=fq_0\tforal f\in F\andeqn \tau(q_0)<\dt\tforal\tau\in T(A).
\eneq
Define $\phi_0(c)=q_0h_1'(c)$ and $\phi_0'(c)=(1-p_0-q_0)h_1'(c)$
for all $c\in C.$

By Lemma 9.5 of \cite{Lnctr1}, there is \SCA\, $B_0\in
(1-p_0-q_0)A(1-p_0-q_0)$ with $B_0\in {\cal I}$ and a unital
\hm\,$h_1: C\to B_0$ such that
\beq\label{HITK5-}
(h_1)_{*0}&=&(\phi_0')_{*0}\andeqn\\
\|\tau\circ h_1(f)-(\tau(1-p_0-q_0))\tau\circ
h(f)\|&<&\dt_1/8\pi\tforal f\in {\cal G}_2
\eneq
and for all $\tau\in T(A).$

Define $\phi_1=h_0+h_1.$ By applying Lemma 7.4 of \cite{Lnctr1},
we obtain a unital \hm\, $\Phi: C\to q_0Aq_0$ with
$\Phi_{*0}=(\phi_0)_{*0},$ $\Phi$ is homotopically trivial, and,
for all $w\in {\cal U},$
\beq\label{HITK5}
h^{\ddag}({\bar w})^{-1}(\Phi\oplus \phi_1)^{\ddag}({\bar
w})=\overline{g_w},
\eneq
where $g_w\in U_0(A)$ and ${\rm cel}(g_w)<\dt_1/2.$ Define
$h_2=\Phi\oplus \phi_1.$ Then
\beq\label{HITK5+1}
 [h_2]=[h]\,\,\,{\rm in}\,\,\,KK(C, A).
 \eneq
Moreover, we compute that
\beq\label{HITK5+2}
|\tau\circ h(f)-\tau(h_2(f))|&<& \dt_1/4 +|\tau\circ h(f)-\tau(h_1(f))|\\
&\le &\dt_1/2+\dt_1/4 +|\tau(1-p_0-q_0)\tau(h(f))-\tau(h_1(f))|\\
&<& \dt_1/2\pi\,\,\,\,\,\tforal f\in {\cal G}_2
\eneq
and for all $\tau\in T(A).$ Since $\overline{\widehat{{\cal H}_0}}\supset \Delta_C(\overline{{\cal U}_1}),$ by combining the above inequality with (\ref{HITK5}), 
\beq\label{HITK6-1}
{\rm dist}(h_2^{\ddag}({\bar w}), h^{\ddag}({\bar
w}))&<&\dt_1\tforal w\in {\cal U}.
\eneq
By applying Corollary 11.6 of \cite{Lnn1} that  there is a unitary
$U\in A$ such that
\beq\label{HITK6}
{\rm ad}\, U\circ h_2\approx_{\ep_1/2} h\,\,\, {\rm on}\,\,\,{\cal
F}.
\eneq
Choose $u=U^*(u_0+(1-p_0))U.$ Then
\beq\label{HITK7}
\|[h(c),\, u]\|<\ep_1\tforal c\in {\cal F}.
\eneq
Moreover, by the choice of $\ep_1,$
\beq
{\rm Bott}(h, u)&=&{\rm Bott}(h_2,u_0+(1-p_0))\\
&=&{\rm Bott}(h_0, u_0)=\kappa\circ {\boldsymbol{\bt}}.
\eneq

Now consider the general case. We consider $h\otimes {\rm
id}_{M_6}: M_6(C)\to M_6(A).$ Suppose that $1/2>\ep>0.$ For
$\ep_2={\ep_1\over{4\cdot 36}}$ and ${\cal F}_2=\{(a_{i,j})\in
M_6(A): a_{i,j}\in {\cal F}, {\rm or}\,\,\,\,a_{i,j}=0, 1_A\},$
applying what we have proved to $h\otimes{\rm id}_{M_6},$ we
obtain a unitary $W\in U(M_6(A))$ such that
\beq\label{HITK8}
\|[h\otimes{\rm id}_{M_6}(c),\, W]\|<\ep_2\tforal c\in {\cal F}_2\andeqn\\
 {\rm Bott}(h\otimes {\rm id}_{M_6}, W)=\kappa\circ {\boldsymbol{\bt}}.
 \eneq
It is easy to compute that there is a unitary $u\in U(A)$ such
that
\beq\label{HIKT9}
\|U-W\|<\ep_1,
\eneq
where $U={\rm diag}(\overbrace{u,u,...,u}^6).$ It follows that
\beq
\|[h(c),\, u]\|<\ep\tforal c\in {\cal F}\\
\andeqn {\rm Bott}(h,u)={\rm Bott}(h\otimes{\rm id}_{M_6},
U)=\kappa\circ {\boldsymbol{\bt}}.
\eneq

\end{proof}

\begin{rem}\label{CHITK}
{\rm Let $A$ be a unital simple separable \CA\, with $TR(A)\le 1$
and let $C$ be a unital AH-algebra as in Theorem \ref{HITK}. Note
that, given \hm s $b_i: (\imath_n)_{*i}(K_i(C_n))\to K_i(A)$
($i=0,1$), there is an element $\kappa\in
{\rm Hom}_{\Lambda}(\underline{K}(C_n\otimes C(\T)), \underline{K}(A))$
such that $\kappa|_{K_i(C_n)}=b_i\circ (\imath_n)_{*i},$ $i=0,1.$
Thus Theorem \ref{HITK} says that $A$ has property (B1) and (B2)
associated with $C$ as defined in 3.6 of \cite{LN}.

}
\end{rem}

\section{Asymptotic unitary equivalence in simple \CA s of tracial
rank one}

\begin{lem}\label{LTRBOT}
Let $C$ be a unital AH-algebra in ${\cal C}$
 and let $A$ be a unital separable  simple \CA\,
with $TR(A)\le 1.$ Suppose that $\phi_1, \phi_2: C\to A$ are two
unital monomorphisms. Suppose that
\beq\label{ltrbot-1}
[\phi_1]&=&[\phi_2]\,\,\,\text{in}\,\,\,KL(C,A),\\
\phi_1^{\ddag}=\phi_2^{\ddag},\,\,\,\,\,\,
(\phi_1)_T&=&(\phi_2)_T\tand\\\label{ltrbot-1+}
R_{\phi_1,\phi_2}(K_1(M_{\phi_1,\phi_2}))&=&\rho_A(K_0(A)).
\eneq
Then, for any increasing sequence of finite subsets $\{{\cal
F}_n\}$ of $C$ whose union  is dense in $C,$ any increasing
sequence of finite subsets $\{{\cal P}_n\}$ of $K_1(C)$ with
$\cup_{n=1}^{\infty}{\cal P}_n= K_1(C)$ and any decreasing
sequence of positive number $\{\dt_n\}$ with
$\sum_{n=1}^{\infty}\dt_n<\infty,$ there exists a sequence of
unitaries $\{u_n\}$ in $U(A)$ such that
\beq\label{ltrbot-2}
{\rm ad}\, u_n\circ \phi_1\approx_{\dt_n}
\phi_2\,\,\,\text{on}\,\,\, {\cal F}_n
\eneq
and
\beq\label{ltrbot-3}
\rho_A({\rm bott}_1(\phi_2, u_n^*u_{n+1})(x))=0\tforal x\in {\cal
P}_n
\eneq
for all sufficiently large $n.$
\end{lem}

\begin{proof}
It follows  from Corollary 11.7 of \cite{Lnn1}  that there exists a
sequence of unitaries $\{u_n\}\subset A$ such that
\beq\label{LTR1}
\lim_{n\to\infty}{\rm ad}\, v_n\circ \phi_1(a)=\phi_2(a)\tforal
a\in C.
\eneq

Without loss of generality, we may assume that ${\cal
P}_n=\{z_1,z_2,...,z_n\}$ and we may also assume that ${\cal F}_n$
are in the unit ball of $C$ and $\cup_{n=1}^{\infty}{\cal F}_n$ is
dense in the unit ball of $C$ and we write that
$C=\lim_{n\to\infty}(C_n, \psi_n),$ where
$C_n=P_nM_{R(n)}(C(X_n))P_n$ and $C_n$ is a finite direct sum of
\CA s in ${\cal C}_0,$ and where each $\psi_n$ is a unital
monomorphism. We may assume that ${\cal F}_n\subset \phi_{n,
\infty}(C_n)$ and ${\cal P}_n\subset (\phi_{n,
\infty})_{*1}(K_1(C_n)),$ $n=1,2,....$

Let $1/2>\ep_n'>0$ so that ${\text{Bott}}(h', u')|_{{\cal P}_n} $
is well defined for any unital monomorphism $h': C\to A$ and any
unitary $u'\in A$ for which
\beq\label{LTR2}
\|[h'(a), u']\|<\ep_n'\tforal a\in {\cal G}_{n}
\eneq
for some finite subset ${\cal G}_n\subset \phi_{n, \infty}(C_n)$
which contains ${\cal F}_n.$ Moreover, we may assume that\\
${\text{Bott}}(h'\circ \psi_{n,\infty},u')$ is well defined
whenever (\ref{LTR2}) holds.

Put $\ep_n''=\min\{\ep_n'/2, {1/2^{n+1}}, \dt_n/2\}.$

Let $\dt'_n>0$ ( in place of $\dt$), $k(n)\ge 1$ (in place of $n$)
 and ${\cal Q}_n\subset K_1(C_n)$ (in place
of ${\cal Q}$) be as required by \ref{HITK} for
 $\ep_n''/2$ (in place of
$\ep$), ${\cal F}_n$ (in place of ${\cal F}$) and ${\cal
P}_n=\{z_1,z_2,..., z_n\}$ (in place of ${\cal P}$).

Note here we assume that $\{z_1,z_2,...,z_n\}\subset
(\psi_{k(n),\infty})_{*1}(K_1(C_{k(n)})).$ So there are $y_1,
y_2,...,y_n\in K_1(C_{k(n)})$ such that
$(\psi_{k(n),\infty})_{*1}(y_i)=z_i,$ $i=1,2,...,n.$

Suppose that $\{y_1,y_2,...,y_{m(n)}\}$ ($m(n)\ge n$) forms a set of
generators of $K_1(C_{k(n)}).$  We may assume, without loss of
generality, that ${\cal Q}_n=\{y_1, y_2,...,y_{m(n)}\}.$  Put
$\ep_n=\min\{\ep_n''/2, \dt_n'/2\}.$

We may assume that
\beq\label{LTR4}
{\rm ad}\, v_n\circ \phi_1\approx_{\ep_n}
\phi_2\,\,\,\text{on}\,\,\, {\cal G}_{k(n)}
\eneq
Moreover we may assume that
\beq\label{LTR5}
\|[\phi_1(a),\, v_nv_{n+1}^*]\|<\ep_n\tforal a\in {\cal G}_{k(n)}.
\eneq
Thus $\text{bott}_1(\phi_1\circ \psi_{k(n),\infty}, v_nv_{n+1}^*)$
is well defined.






Let $w_1,w_2,...,w_n,...,w_{m(n)}\in M_l(C_{k(n)})$ for some $l\ge
1$ be unitaries such that $[w_i]=(\psi_{k(n),\infty})_{*1}(y_i),$
$i=1,2,...,m(n).$
 Put
$$
{\bar v}_n={\rm diag}(\overbrace{v_n,v_n,\dots, v_n}^{l}).
$$
We will continue to use $\phi_i$ for $\phi_i\otimes{\rm
id}_{M_l},$ $i=1,2.$

By choosing even larger ${\cal G}_{k(n)},$ we may also assume that
\beq\label{LTR6}
\|\phi_2(w_j){\rm ad}\, {\bar
v}_n(\phi_1((w_j)^*))-1\|<(1/4)\sin(2\pi \ep_n),\,\,\,n=1,2,....
\eneq
Put
$$
h_{j,n}=\log({1\over{2\pi i}}\phi_2(w_j){\rm ad}\, {\bar
v}_n(\phi_1((w_j)^*))),\,\,\,j=1,2,...,m(n), n=1,2,....
$$
Then for any $\tau\in T(A),$
\beq\label{LTR7}
\tau(h_{j,n})<\ep_n<\dt_n',\,\,\,j=1,2,...,m(n).
\eneq
By the assumption that
$R_{\phi_1,\phi_2}(M_{\phi_1,\phi_2})=\rho_A(K_0(A))$ and by
Lemma \ref{Exel}, we conclude that
\beq\label{LTR8}
\widehat{h_{j,n}}(\tau)=\tau(h_{j,n})\in \rho_A(K_0(A)).
\eneq

By applying 6.1, 6.2 and 6.3  of \cite{Lnemb2}, and  choosing
larger ${\cal G}_{k(n)},$
we obtain a \hm\, $\af_n': K_1(C_n)\to \rho_A(K_0(A))$ such that
\beq\label{LTR-9}
\af_n'(y_j)(\tau)=\widehat{h_{j,n}}(\tau)=\tau(h_{j,n}),
j=1,2,...,m(n).
\eneq

Since $\af_n'(K_1(C_n))$ is free, it follows from (\ref{LTR8})
that there is a \hm\, $\af_n^{(1)}: K_1(C_n)\to K_0(A)$ such that
\beq\label{LTR-10}
\rho_A\circ \af_n^{(1)}(y_j)(\tau)=\tau(h_{j,n}),\,\,\,
j=1,2,...,m(n).
\eneq
Define $\af_n^{(0)}: K_0(C_n)\to K_1(A)$ by $\af_n^{(0)}=0.$ By
the UCT, there is $\kappa_n\in {\rm Hom}_{\Lambda}(\underline{K}(C_n),
\underline{K}(A))$ such that $\kappa_n|_{K_1(C_n)}=\af_n^{(1)}.$

It follows from \ref{HITK} that there exists a unitary $U_n\in
U_0(A)$ such that
\beq\label{LTR-11}
&&\|[\phi_2(a),\, U_n]\|<\ep_n''/2\rforal a\in {\cal
F}_n\\\label{LTR-11+} &&\hspace{-0.3in}\andeqn\,\,
\rho_A(\text{bott}_1 (\phi_2, U_n))(z_j)=-\rho_A\circ \af_n(z_j),
\eneq
$j=1,2,...,n.$

 By the Exel trace formula \ref{Exel} and (\ref{LTR-9}), we have
\beq\label{LTR-12}
\tau(h_{j,n})&=&-\rho_A(\text{bott}_1(\phi_2,
U_n))(z_j)(\tau)\\\label{LTR-12+} &=&-\tau({1\over{2\pi i}}\log(
{\bar U_n}\phi_2(w_j){\bar U_n}^* \phi_2(w_j^*)))
\eneq
for all $\tau\in T(A),$ $j=1,2,...,m,$ where
\beq\label{LTR-13}
{\bar U}_n={\rm diag}(\overbrace{U_n,U_n\cdots, U_n}^l)
\eneq
 Define $u_n=v_nU_n,$ $n=1,2,....$
Put
$$
{\bar u_n}={\rm diag}(\overbrace{u_n,u_n,\dots, u_n}^l).
$$

By 6.1 of \cite{Lnemb2}), and by (\ref{LTR-11+}) and
(\ref{LTR-12+}), we compute that
\beq\label{LTR-14}
&&\hspace{-0.3in}\tau(\log({1\over{2\pi i}}(\phi_2(w_j){\rm ad}\,
{\bar
u_n}(\phi_1(w_j^*)))))\\
&=&\tau(\log({1\over{2\pi i}}({\bar U_n}\phi_2(w_j^*){\bar
U_n}^*{\bar v}_n^*
(\phi_1(w_j)){\bar v_n})))\\
&=&\tau(\log({1\over{2\pi i}}({\bar U_n}\phi_2(w_j){\bar
U_n}^*\phi_2(w_j^*)\phi_2(w_j){\bar v}_n^*
(\phi_1(w_j^*)){\bar v_n})))\\
&=&\tau(\log({1\over{2\pi i}}({\bar U_n}\phi_2(w_j){\bar
U_n}^*\phi_2(w_j^*))))\\
&&\hspace{0.6in}+\tau({1\over{2\pi i}}\log(\phi_2(w_j){\bar v}_n^*
(\phi_1(w_j^*)){\bar v_n})))\\\label{LTR-15}
&=&\rho_A(\text{bott}_1(\phi_2, U_n))(z_j)(\tau)+\tau(h_{j,n})=0
\eneq
for all $\tau\in T(A).$


Let
\beq\label{LTR-16}
b_{j,n}&=&\log({1\over{2\pi i}}{\bar u_n}\phi_2(w_j){\bar
u_n}^*\phi_1(w_j^*))\andeqn\\
 b_{j,n}'&=&\log({1\over{2\pi
i}}\phi_2(w_j){\bar u_n}^*{\bar u_{n+1}}\phi_2(w_j^*){\bar
u_{n+1}}^*{\bar u_n})),
\eneq
$j=1,2,...,n$ and $n=1,2,....$
 We have, by (\ref{LTR-15}),
\beq\label{LTR-17}
\tau(b_{j,n})&=&\tau(\log({1\over{2\pi i}}{\bar
u_n}\phi_2(w_j){\bar
u_n}^*\phi_1(w_j^*)))\\
&=& \tau(\log{1\over{2\pi i}}{\bar u_n}^*{\bar
u_n}\phi_2(w_j){\bar u_n}^*\phi_1(w_j^*){\bar u_n})\\
&=&\tau(\log{1\over{2\pi i}}\phi_2(w_j){\bar u_n}^*\phi_1(w_j^*)
{\bar u_n})=0
\eneq
for all $\tau\in T(A),$ $j=1,2,...,n$ and $n=1,2,....$  Note also
$\tau(b_{j,n+1})=0$ for all $\tau\in T(A)$ and for $j=1,2,...,n.$
Note that
\beq\label{LTR-18}
{\bar u_n}e^{2\pi i b_{j,n}'}{\bar u_n}^*=e^{2\pi i
b_{j,n}}e^{-2\pi i b_{j,n+1}}.
\eneq
Thus, by 6.1 of \cite{Lnemb2} and by (\ref{LTR-18}), we compute
that
\beq\label{LTR-19}
\tau(b_{j,n}')=\tau(b_{j,n})-\tau(b_{j,n+1})=0\rforal \tau\in
T(A).
\eneq
It follows the Exel trace formula (\ref{Exel}) and (\ref{LTR-19})
that
\beq\label{LTR-20}
&&\hspace{-0.3in}\rho_A(\text{bott}_1(\phi_2,u_n^*u_{n+1}))(z_j^*)(\tau)\\
&=&\tau(\log({\bar u_n}^*{\bar
u_{n+1}}\phi_2(w_j^*){\bar u_{n+1}}^*{\bar u_n}\phi_2(w_j))\\
&=&\tau(\log(\phi_2(w_j){\bar u_n}^*{\bar u_{n+1}}\phi_2(w_j^*){\bar
u_{n+1}}^*{\bar u_n}))=0
\eneq
for all $\tau\in T(A).$ It follows that (\ref{ltrbot-3}) holds. By
(\ref{LTR4}) and (\ref{LTR-11}), one concludes that (\ref{ltrbot-2})
also holds.

\end{proof}

\begin{thm}\label{TM}
Let $C$  be a unital  simple \CA\, in ${\cal N}$ with $TR(C)\le 1$
and let $A$ be a unital separable simple \CA\, with $TR(A)\le 1.$
Suppose that $\phi_1, \phi_2: C\to A$ are two unital
monomorphisms. Then there exists a continuous path of unitaries
$\{u(t):t\in [0,\infty)\}\subset A$ such that
\beq\label{TM-2}
\lim_{t\to\infty}{\rm ad}\, U(t)\circ \phi_1(c)=\phi_2(c)\tforal
c\in C
\eneq
if and only if
\beq\label{TM-1}
[\phi_1]&=&[\phi_2]\,\,\,\text{in}\,\,\,KK(C,A),\\
\phi_1^{\ddag}=\phi_2^{\ddag},\,\,\,\,\,\, (\phi_1)_T&=&(\phi_2)_T\tand\\
\overline{R}_{\phi_1,\phi_2}&=&\{0\}.
\eneq
\end{thm}

\begin{proof}
The ``only if " part follows from  \ref{NecC}. We need to show
that ``if" part of the theorem.

 Let $C=\lim_{n\to\infty}(C_n, \psi_n),$ where $C_n$ is a finite direct sum of \CA s in ${\cal C}_0$
 and $\psi_n: C_n\to C_{n+1}$ is a unital monomorphism.
Let $\{{\cal F}_n\}$ be an increasing sequence of finite subsets
of $C$ such that $\cup_{n=1}^{\infty}{\cal F}_n$ is dense in $C.$

Put
$$
M_{\phi_1, \phi_2}=\{f\in C([0,1],A): f(0)=\phi_1(a)\andeqn
f(1)=\phi_2(a)\,\,\,{\rm for\,\,\,some}\,\,\,a \in C\}.
$$

Since $C$ satisfies the Universal Coefficient Theorem, the
assumption that   $[\phi_1]=[\phi_2]$ in $KK(C,A)$ implies the
following
 exact sequence splits:
\beq\label{nFM}
\begin{array}{ccccccc}
0 \to  & \underline{K}(SA)  &\to &
\underline{K}(M_{\phi_1, \phi_2})
&{\stackrel{\pi_0}{\rightleftarrows}}_{\theta}& \underline{K}(C)
&\to 0
\end{array}
\eneq
For some $\theta\in {\rm Hom}_{\Lambda}(\underline{K}(C),
\underline{K}(A)).$

Furthermore, since $\tau\circ \phi_1=\tau\circ \phi_2$ for all
$\tau\in T(A)$ and $\overline{R}_{\phi_1,\phi_2}=\{0\},$ we may
also assume that
\beq\label{FM-1-1}
{\rm R}_{\phi_1,\phi_2}(\theta(x))=0\rforal x\in K_1(C).
\eneq

By \cite{DL}, one has
\beq\label{FM-1}
\lim_{n\to\infty}(\underline{K}(C_n), [\psi_n])=\underline{K}(C).
\eneq

Since each $K_i(C_n)$ is finitely generated, there is an integer
$K(n)\ge 1$ such that
\beq\label{NNFM-00}
{\rm Hom}_{\Lambda}(F_{K(n)}\underline{K}(C_n),
F_{K(n)}\underline{K}(A))={\rm Hom}_{\Lambda}(\underline{K}(C_n),
\underline{K}(A)).
\eneq

Let $\dt_n'>0$ (in place of $\dt$), ${\cal G}_n'\subset C$  (in
place of ${\cal G}$) and ${\cal P}_n'\subset \underline{K}(C)$ (in
place of ${\cal P}$) be finite subsets corresponding to
$1/2^{n+2}$ and ${\cal F}_n$ required by Theorem 7.4 of
\cite{Lnnhomp}. Without loss of generality, we may assume that
${\cal G}_n'\subset \psi_{n, \infty}({\cal G}_{n})$ and ${\cal
P}_n'= [\psi_{n, \infty}]({\cal P}_n)$ for some finite subset
${\cal G}_n$ of $C_n$ and for some finite subset ${\cal
P}_n\subset \underline{K}(C_{n}).$  We may assume that ${\cal
P}_n$ contains a set of generators of
$F_{K(n)}\underline{K}(C_{n}),$ $\dt_n'<1/2^{n+3}$ and ${\cal
F}_n\subset {\cal G}_n'.$ We also assume that $\text{Bott}(h',
u')|_{{\cal P}_n}$ is well defined whenever $\|h'(a),\,
u']\|<\dt_n'$ for all $a\in {\cal G}_n'$ and for any unital \hm\,
$h'$ and unitary $u'.$ Note that $\text{Bott}(h',u')|_{{\cal
P}_n}$ defines $\text{Bott}(h'|_{C_n},u').$

We further assume that
\beq\label{NNF0}
\text{Bott}(h,u)|_{{\cal P}_n}=\text{Bott}(h',u)|_{{\cal P}_n}
\eneq
provided that $h\approx_{\dt_n'}h'$ on ${\cal G}_n'.$


We may also assume that $\dt_n'$ is smaller than $\dt/3$ for that
$\dt$ in \ref{ddbot} for $C_n$ and ${\cal P}_n.$

Let $k(n)\ge n$ (in place of $n$) and $\eta_n>0$ (in place of
$\dt$) be required by \ref{HITK} for $\psi_{n,\infty}({\cal G}_n)$
(in place of ${\cal F}$), ${\cal P}_n$ (in place of ${\cal P}$)
and $\dt_n'/4$ (in place of $\ep$). For $C_n,$ since $K_i(C_n)$
($i=0,1$) is finitely generated, by choosing larger $k(n),$ we may
assume that $(\phi_{k(n),\infty})_{*i}$ is injective on $(\phi_{n,
k(n)})_{*i}(K_i(C_n)),$ $i=0,1.$
Since $K_i(C_n)$ is finitely generated, by (\ref{NNFM-00}), we may
further assume that $[\phi_{k(n), \infty}]$ is injective on
$[\phi_{n, k(n)}](\underline{K}(C_n)),$ $n=1,2,....$

By passing to a subsequence, to simplify notation, we may assume
that $k(n)=n+1.$

Let $\dt_n=\min\{\eta_n, \dt_n/2'\}.$

By \ref{LTRBOT},  there are unitaries $v_n\in U(A)$ such that
\beq\label{NFM2}
&&{\rm ad}\, v_n\circ \phi_1\approx_{\dt_{n+1}/4}
\phi_2\,\,\,\text{on}\,\,\,\psi_{n+1,\infty}({\cal
G}_{n+1}),\\\label{NFM3} &&\rho_A(\text{bott}_1(\phi_2,\,
v_n^*v_{n+1}))(x)=0
\rforal x\in \psi_{n+1,\infty}(K_1(C_{n+1}))\andeqn\\
&&\|[\phi_2(a),\,v_n^*v_{n+1}]\|<\dt_{n+1}/2\rforal a\in
\psi_{n+1,\infty}({\cal G}_{n+1})
\eneq
(Note that $K_1(C_{n+1})$ is finitely generated).

Note that, by (\ref{NNF0}), we may also assume that
\beq\label{NFM3+}
\text{Bott}(\phi_1, v_{n+1}v_n^*)|_{{\cal
P}_n}=\text{Bott}(v_n^*\phi_1v_n,v_n^*v_{n+1})|_{{\cal
P}_n}=\text{Bott}(\phi_2, v_n^*v_{n+1})|_{{\cal P}_n}.
\eneq
In particular,
\beq\label{NFM3+1}
\text{bott}_1(v_n^*\phi_1v_n,v_n^*v_{n+1})(x)=\text{bott}_1(\phi_2,
v_n^*v_{n+1})(x)
\eneq
for all $x\in \psi_{n+1,\infty}(K_1(C_{n+1})).$

 By applying 10.4 and 10.5 of \cite{Lnasym}, without loss of generality, we may assume that
$\phi_1$ and $v_n$ define $\gamma_n\in
{\rm Hom}_{\Lambda}(\underline{K}(C_{n+1}), \underline{K}(M_{\phi_1,
\phi_2}))$ and $[\pi_0]\circ \gamma_n=[{\rm id}_{C_{n+1}}].$
Furthermore, by 10.4  and 10.5 of \cite{Lnasym}, without loss of
generality, we may assume that
\beq\label{0TM-10-1}
\tau(\log(\phi_2\circ\psi_{n+1,\infty}(z_j)^*{\widetilde{v_n}}^*\phi_1\circ
\psi_{n+1,
\infty}(z_j){\widetilde{v_n}}))<\dt_{n+1},\,\,\,j=1,2,...,r(n),
\eneq
where $\{z_1,z_2,...,z_{r(n)}\}\subset U(M_k(C_{n+1}))$ which
forms a set of generators of $K_1(C_{n+1})$ and where
$\widetilde{v_n}={\rm diag}(\overbrace{v_n,v_n,...,v_n}^k).$

Let $H_n=[\phi_{n+1}](\underline{K}(C_{n+1})).$ Since
$\cup_{n=1}[\phi_{n+1,
\infty}](\underline{K}(C_n))=\underline{K}(C)$ and $[\pi_0]\circ
\gamma_n=[{\rm id}_{C_{n+1}}],$ we conclude that
\beq\label{0TM-10-}
\underline{K}(M_{\phi_1,\phi_2})=\underline{K}(SA)+\cup_{n=1}^{\infty}\gamma_n(H_n).
\eneq
Thus, by passing to a subsequence, we may further assume that
\beq\label{0TM-10--}
\gamma_{n+1}(H_n)\subset
\underline{K}(SA)+\gamma_{n+2}(H_{n+1}),\,\,\,n=1,2,....
\eneq

By identifying $H_n$ with $\gamma_{n+1}(H_n),$ we may write $j_n:
\underline{K}(SA)\oplus H_n\to \underline{K}(SA)\oplus H_{n+1}.$
By (\ref{0TM-10-}), the inductive limit is
$\underline{K}(M_{\phi_1,\phi_2}).$

From the definition of $\gamma_n,$ we note that,
$\gamma_{n}-\gamma_{n+1}\circ [\psi_{n+1}]$ maps
$\underline{K}(C_{n+1})$ into $\underline{K}(SA).$

By 10.6 of \cite{Lnasym},
$
\Gamma(\text{Bott}(\phi_1,
v_{n}v_{n+1}^*))|_{H_{n}}=(\gamma_{n+1}-\gamma_{n+2}\circ
[\psi_{n+2}])|_{H_n}
$
gives a \hm\, $\xi_n: H_n\to \underline{K}(SA).$
Put $\zeta_n=\gamma_{n+1}|_{H_n}.$ Then
\beq\label{0TM-19}
j_n(x,y)=(x+\xi_n(y),[\psi_{n+2}](y))\rforal (x, y)\in
\underline{K}(SA)\oplus H_n.
\eneq

Thus, we obtain the following  diagram:
\beq\label{LDG1}
\begin{array}{ccccccc}
0 \to  & \underline{K}(SA)  &\to & \underline{K}(SA)\oplus H_n
&\to
& H_n &\to 0\\\nonumber
 &\| & &\hspace{0.4in}\| \hspace{0.15in}\swarrow_{\xi_n} \hspace{0.05in}\downarrow_{[\psi_{n+2}]} &&
 \hspace{0.2in}\downarrow_{[\psi_{n+2}]} &\\\label{LDG2}
 0 \to  & \underline{K}(SA)  &\to & \underline{K}(SA)\oplus H_{n+1} &\to & H_{n+1} &\to 0\\\nonumber
 &\| & &\hspace{0.4in}\| \hspace{0.1in}\swarrow_{\xi_{n+1}}\downarrow_{[\psi_{n+3}]} &&
\hspace{0.2in} \downarrow_{[\psi_{n+3}]} &\\
 0 \to  & \underline{K}(SA)  &\to & \underline{K}(SA)\oplus H_{n+2} &\to & H_{n+2} &\to 0\\
 \end{array}
\eneq
By the assumption that $\overline{R}_{\phi_1,\phi_2}=0,$ $\theta$
also gives the following
$$
{\rm ker}R_{\phi_1, \phi_2}={\rm ker}\rho_A\oplus K_1(C).
$$

Define $\theta_n=\theta\circ [\psi_{n+2, \infty}]$ and
$\kappa_n=\zeta_n-\theta_n.$ Note that
\beq\label{NFM11-1}
\theta_n=\theta_{n+1}\circ [\psi_{n+2}].
\eneq
We also have that
\beq\label{NFM-11-2}
\zeta_n-\zeta_{n+1}\circ [\psi_{n+2}]=\xi_n.
\eneq

 Since $[\pi_0]\circ (\zeta_n-\theta_n)=0,$ $\kappa_n$ maps $H_n$
into $\underline{K}(SA).$ It follows that
\beq\nonumber
\kappa_n-\kappa_{n+1}\circ [\psi_{n+2}] &=&
\zeta_n-\theta_n-\zeta_{n+1}\circ [ \psi_{n+2}]+\theta_{n+1}\circ
[\psi_{n+2}]\\\label{NFM11}
&=&\zeta_n-\zeta_{n+1}\circ[\psi_{n+2}]=\xi_n
\eneq

It follows from 10.3 of \cite{Lnasym}  that  there are integers
$N_1\ge 1,$  a ${\dt_{n+1}\over{4}}$-$\psi_{n+1}({\cal
G}_{n+1})$-multiplicative \morp\, $L_n: \psi_{n,
\infty}(C_{n+1})\to M_{1+N_1}(M_{\phi_1,\phi_2}),$ a unital \hm\,
$h_0: \psi_{n+1, \infty}(C_{n+1})\to M_{N_1}(\C),$
and a continuous path of unitaries $\{V_n(t): t\in [0,3/4]\}$ of
$M_{1+N_1}(A)$ such that $[L_n]|_{{\cal P}_{n+1}'}$ is well
defined, $V_n(0)=1_{M_{1+N_1}(A)},$
\beq\label{NVuV1}
[L_n\circ \psi]|_{{\cal P}_n}=(\theta\circ
[\psi_{n+1,\infty}]+[h_0\circ \psi_{n+1,\infty}])|_{{\cal P}_n},
\eneq
\beq\label{NVuV2}
\hspace{-0.4in}\pi_t\circ L_n\circ
\psi_{n+1,\infty}\approx_{\dt_{n+1}/4} {\rm ad}\, V_n(t)\circ
((\phi_1\circ \psi_{n+1,\infty})\oplus (h_0\circ
\psi_{n+1,\infty}))
\eneq
on $\psi_{n+1,\infty}({\cal G}_{n+1})$ for all $t\in (0,3/4],$
\beq\label{NVuV3}
\hspace{-0.4in}\pi_t\circ L_n\circ
\psi_{n+1,\infty}\approx_{\dt_{n+1}/4} {\rm ad}\, V_n(3/4)\circ
((\phi_1\circ \psi_{n+1,\infty})\oplus (h_0\circ
\psi_{n+1,\infty}))
\eneq
on $\psi_{n+1,\infty}({\cal G}_{n+1})$ for all $t\in (3/4,1),$ and
\beq\label{NVuV33}
\pi_1\circ L_n\circ
\psi_{n+1,\infty}\approx_{\dt_{n+1}/4}\phi_2\circ
\psi_{n+1,\infty}\oplus h_0\circ \psi_{n+1,\infty}
\eneq
on $\psi_{n+1,\infty}({\cal G}_{n+1}),$ where $\pi_t:
M_{\phi_1,\phi_2}\to A$ is the point-evaluation at $t\in (0,1).$

Note that $R_{\phi_1,\phi_2}(\theta(z))=0$ for all $x\in
\phi_{n+1,\infty}(K_1(C_{n+1})).$  As computed in 10.4 of
\cite{Lnasym},
\beq\label{NVuV4}
\tau(\log((\phi_2(z)\oplus h_0(z)^*V_n(3/4)^*(\phi_1 (z)\oplus
h_0(z))V_n(3/4)))=0
\eneq
for $z=\psi_{n+1, \infty}(y),$ where $y$ is in a set of generators
of $K_1(C_{n+1})$ and for all $\tau\in T(A).$

Define $W_n'={\rm daig}(v_n,1)\in M_{1+N_1}(A).$ Then
$\text{Bott}((\phi_1\oplus h_0)\circ \psi_{n+1, \infty},\,
W_n'(V_n(3/4)^*)$ defines a \hm\, ${\tilde \kappa}_n\in
{\rm Hom}_{\Lambda}(\underline{K}(C_{n+1}),\underline{K}(SA)).$ By
(\ref{0TM-10-1})
\beq\label{NNFM1-}
\tau(\log((\phi_2\oplus h_0)\circ \psi_{n+1,\infty}(z_j)^*{\tilde
V}_n^*(\phi_1\oplus h_0)\circ \psi_{n+1,\infty}(z_j){\tilde
V}_n))<\dt_{n+1},
\eneq
$j=1,2,...,r(n),$ where ${\tilde V}_n={\rm diag}(V_n,1).$ Then, by
(\ref{NVuV4}), using exactly the same argument from
(\ref{LTR-17})--(\ref{LTR-19}), we compute that
\beq\label{NNFM1}
\rho_A({\tilde \kappa}_n(z_j))(\tau) <\dt_{n+1},\,\,\,j=1,2,....
\eneq
It follows from \ref{HITK} that there is a unitary $w_n'\in U(A)$
such that
\beq\label{NNFM2}
\|[\phi_1(a), w_n']\|<\dt_{n+1}'/4\rforal a\in \psi_{n,\infty}({\cal G}_n)\andeqn\\
\text{Bott}(\phi_1\circ \psi_{n,\infty},\, w_n')=-{\tilde
\kappa}_n\circ[\psi_n].
\eneq
By (\ref{NNF0}),
\beq\label{NNFM3}
\text{Bott}(\phi_2\circ \psi_{n,\infty},\,v_n^*w_n'v_n)|_{{\cal
P}_n}=-{\tilde \kappa}_n\circ [\psi_n]|_{{\cal P}_n}.
\eneq
Put $w_n=v_n^*w_n'v_n.$

It follows from 10.6 of \cite{Lnasym} that
\beq\label{NNFM4}
\Gamma(\text{Bott}(\phi_1\circ \psi_{n, \infty}, w_n'))=-\kappa_n
\andeqn \Gamma(\text{Bott}(\phi_1\circ \psi_{n, \infty},
w_{n+1}'))=-\kappa_{n+1}.
\eneq
We also have
\beq\label{NNFM5}
\Gamma(\text{Bott}(\phi_1\circ \psi_{n,\infty},
v_nv_{n+1}^*))|_{H_n}=\zeta_n-\zeta_{n+1}\circ [\psi_{n+2}]=\xi_n.
\eneq
But, by (\ref{NFM11}),
\beq\label{NNFM6}
(-\kappa_n +\xi_n+\kappa_{n+1}\circ [\psi_{n+1}])=0.
\eneq
It follows from (\ref{NNFM4}), (\ref{NNFM5}), (\ref{NNFM6}) and
10. 6 of \cite{Lnasym}that
\beq\label{NNFM7}
\hspace{-0.2in} -\text{Bott}(\phi_1\circ \psi_{n, \infty},\,w_n')
+\text{Bott}(\phi_1\circ \psi_{n, \infty}, \,v_nv_{n+1}^*)
+\text{Bott}(\phi_1\circ\psi_{n,\infty}, w_{n+1}') =0.
\eneq
Define $u_n=v_nw_n^*,$ $n=1,2,....$ Then, by (\ref{NFM2}) and
(\ref{NNFM2}),
\beq\label{NFM16}
{\rm ad}\, u_n\circ \phi_1\approx_{\dt_n'/2} \phi_2\rforal a\in
\psi_{n, \infty}({\cal G}_n).
\eneq

From (\ref{NFM3+}), (\ref{NNF0}) and (\ref{NNFM7}), we compute
that
\beq\label{NNFM8}
&&\hspace{-0.6in}\text{Bott}(\phi_2\circ \psi_{n, \infty},u_n^*u_{n+1})\\
&=& \text{Bott}(\phi_2\circ \psi_{n, \infty}, w_nv_n^*v_{n+1}w_{n+1}^*)\\
&=& \text{Bott}(\phi_2\circ \psi_{n, \infty}, w_n)+\text{Bott}(\phi_2\circ \psi_{n, \infty}, v_n^*v_{n+1})\\
&&\hspace{1.8in}+\text{Bott}(\phi_2\circ \psi_{n, \infty}, w_{n+1}^*)\\
&=&\text{Bott}(\phi_1\circ \psi_{n, \infty},w_n')+\text{Bott}(\phi_1\circ \psi_{n, \infty}, v_{n+1}v_n^*)\\
&&\hspace{1.8in}+\text{Bott}(\phi_1\circ \psi_{n, \infty}, (w_{n+1}')^*)\\
&=&-[-\text{Bott}(\phi_1\circ \psi_{n, \infty}, w_n')+\text{Bott}(\phi_1\circ \psi_{n, \infty}, v_nv_{n+1}^*)\\
&&\hspace{1.8in}+\text{Bott}(\phi_1\circ \psi_{n, \infty}, w_{n+1})]\\
&=&0
\eneq

Therefore, by \ref{LNHOMP}, there exists a piece-wise smooth and
continuous path of unitaries $\{z_n(t): t\in [0,1]\}$ of $A$ such
that
\beq\label{FM-16}
&&z_n(0)=1,\,\,\, z_n(1)=u_n^*u_{n+1}\andeqn\\\label{FM-16+}
&&\|[\phi_2(a),\, z_n(t)]\|<1/2^{n+2}\rforal a\in {\cal F}_n\andeqn
t\in [0,1].
\eneq
Define
$$
u(t+n-1)=u_nz_{n+1}(t)\,\,\,t\in (0,1].
$$
Note that $u(n)=u_{n+1}$ for all integer $n$ and $\{u(t):t\in [0,
\infty)\}$ is a continuous path of unitaries in $A.$ One estimates
that, by (\ref{NFM16}) and (\ref{FM-16+}),
\beq\label{FM-17}
{\rm ad}\, u(t+n-1)\circ \phi_1 &\approx_{\dt_n'}& {\rm
ad}\,z_{n+1}(t)\circ \phi_2
\\
 &\approx_{1/2^{n+2}}& \phi_2
\,\,\,\,\,\,\,\,\,\,\,\,\,\,\,\,\,\,\hspace{0.5in}\text{on}\,\,\,\,{\cal
F}_n
\eneq
 for all $t\in (0,1).$
It then follows that
\beq\label{FM-18}
\lim_{t\to\infty}u^*(t)\phi_1(a) u(t)=\phi_2(a)\rforal a\in C.
\eneq

\end{proof}

\section{The range of  asymptotic unitary equivalence
classes}

\begin{df}
{\rm Let $C$ be a unital \CA\, for which $T(C)\not=\emptyset.$
Denote by  $T_{\mathfrak{f}}(C)$ the set of faithful tracial
states.

Let $A$ be another unital \CA\, with $T(A)\not=\emptyset.$  Let
$\kappa \in KK_e(A,B)^{++},$ $\gamma: T(A)\to
T_{\mathfrak{f}}(C).$ We say that $\gamma$ is compatible with
$\kappa$ if for any projection $p\in M_k(C)$ and $\tau\in T(A),$
$\gamma(\tau)(p)=\tau(\kappa([p])).$ Let $\gamma^*:
\text{Aff}(T(C))\to \text{Aff}(T(A))$ be the continuous affine map
induced by $\gamma,$ i.e.,
$$
\gamma^*(f)(\tau)=f(\gamma(\tau))\tforal f\in \text{Aff}(C(C))
$$
and for all $\tau\in T(A).$ If $\gamma$ is compatible with
$\kappa,$ denote by $\overline{\gamma^*}:
\text{Aff}(T(C))/\overline{\rho_C(K_0(C))}\to
\text{Aff}(T(A))/\overline{\rho_A(K_0(A))}$ the map induced by
$\gamma^*.$ Let $\af: U(C)/CU(C)\to U(A)/CU(A)$ be a \hm. We say
$\af$ is compatible with $\kappa,$ if $\af(U_0(C)/CU(C)\subset
U_0(A)/CU(A)$ and  $q_1\circ \af({\bar
u})=\kappa([u])$  for any unitary $u\in U(A),$ where $q_1: U(A)/CU(A)\to K_1(A)$ is the
quotient map.

Now uppose that $K_1(C)=U(C)/U_0(A)$
and  suppose that $A$ is a unital simple \CA\, with $TR(A)\le 1.$ We say $\af,$
$\gamma$ and $\kappa$ are compatible if $\af$ and $\gamma$ are
compatible with $\kappa$ and
$$
\Delta_A\circ\af\circ \Delta_C^{-1}=\overline{\gamma^*}.
$$

Note that if $u=\prod_{k=1}^m \exp(i h_k),$ where $h_k\in
C_{s.a.},$ then $u\in CU(C)$ if and only if $\sum_{k=1}^m
\tau(h_j)=0$ for all $\tau\in T(C),$ $j=1,2,...,m.$
 }

\end{df}

\begin{NN}\label{Delt}
{\rm Let $X$ be a compact metric space and let $C=PM_k(C(X))P,$
where $P\in M_k(C(X))$ is a projection, and let $A$ be a unital
simple \CA\, with stable rank one. Let $\gamma: T(A)\to
T_{\mathfrak{f}}(C))$ be a continuous affine map. For any $\tau\in
T(A),$ and open subset $O\subset X,$ define
$\mu_{\gamma(\tau)}(O)=\sup\{\gamma(\tau)(f): 0<f\le 1\,\,\,
\text{supp}(f)\subset O\}.$ Since $\gamma(T(A))$ is compact, we
conclude that
$$
\inf_{\tau\in T(A)}\mu_{\gamma(\tau)}(O)>0
$$
for any non-empty open set $O\subset X.$

Fix $a\in (0,1).$ There are finitely many points
$x_1,x_2,...,x_m\in X$ such that $\cup_{i=1}^mB(x_i, a/2)\supset
X.$ Let $O_a$ be an open ball with radius $a.$ Then $O_a\supset
B(x_i, a/2)$ for some $i.$ Define
$$
\Delta_1(a)=\min_{1\le i\le m} \{\inf_{\tau\in
T(A)}\mu_{\gamma(\tau)}(B(x_i, a/2))\}.
$$

}

\end{NN}

\begin{lem}\label{lex}
Let $C\in {\cal C}_0$ and let $A$ be a unital separable simple
\CA\, with $TR(A)\le 1.$
 Suppose that $\kappa\in KK_e(C,A)^{++},$
$\gamma: T(A)\to T_{\mathfrak{f}}(C)$ is a continuous affine map
and $\af: U(C)/CU(C)\to U(A)/CU(C)$ is a \hm\, for which $\gamma,$
$\af$ and $\kappa$ are compatible.

Let $\ep_1>0,$ $\ep_2>0,$ $\eta>0,$ ${\cal H}\subset C_{s.a}$ be a
finite subset and let ${\cal U}\subset U(M_{\infty}(C))$ be a
finite subset. Then there exists a unital monomorphism $h: C\to A$
such that
\beq
[h]&=&\kappa\,\,\,\text{in}\,\,\, KK(C,A),\\
\sup_{\tau\in T(A)}|\tau\circ h(a)-\gamma(\tau)(a)|&<&\ep_1\tforal a\in {\cal H},\,\,\,\text{and},\\
{\rm dist}(h^{\ddag}(\bar{z}),\af(\bar{z}))&<&\ep_2\tforal z\in
{\cal U},
\eneq

Moreover, we may also require that
$$
\mu_{\tau\circ h}(O_a)\ge 3\Delta_1(a)/4\tforal a\ge \eta
$$
for all $\tau\in T(A),$ where $O_a$ is an open ball of $X$ with
radius $a$ and $\Delta_1: (0,1)\to (0,1)$ is given in \ref{Delt}
and where the measure $\mu_{\tau\circ h}$ is defined as in 10.9 of
{\rm \cite{Lnn1}.}

\end{lem}

\begin{proof}
The most part of the proof of this lemma is contained in that of
\ref{hitK1}. For the convenience of the reader, we present the
proof below.

As in the proof of \ref{HITK}, there is an unital embedding
$\imath: B\to A,$ where $B$ is a unital separable amenable simple
\CA\, with $TR(B)=0,$ such that $[\imath]\in KK_e(B,A)^{++}$ is an
invertible element. Therefore there is $\kappa_0: KK_e(C, B)^{++}$
such that
$$
\kappa=[\imath]\times \kappa_0\,\,\, {\rm in}\,\,\, KK(C, A).
$$

We write  $C=PM_r(C(X))P,$ where $X$ is a connected finite CW
complex as described in \ref{Dc0}, $r\ge 6$ and $P\in M_r(C(X))$
is projection.

Let $\ep_1>0,$ $ \ep_2>0,$ $\kappa,$ ${\cal H}$  and ${\cal U}$ be
given. We may assume that $\ep_1<\ep_2/2\pi.$ Denote by
${\overline{\cal U}}$ the image of ${\cal U}$ in $U(C)/CU(C).$
Write $U(C)/CU(C)=U_0(C)/CU(C)\oplus K_1(C).$ Let $\pi_1:
U(C)/CU(C)\to U_0(C)/CU(C)$ and $\pi_2: U(C)/CU(C)\to K_1(C)$ be
two fixed projection maps. Let ${\cal H}_0\subset C_{s.a}$ be a
finite subset so that the image of $\overline{\widehat{{\cal H}_0}}$ in
$\text{Aff}(T(C))/\overline{\rho_C(K_0(C))}$ containing
$\Delta_C(\pi_1(\overline{\cal U})).$ Let ${\cal H}_1={\cal H}\cup
{\cal H}_0.$  To simplify notation, without loss of generality, we
may also assume that $\overline{\cal U}=\pi_1(\overline{\cal
U})\cup\pi_2(\overline{\cal U}).$
Let $\dt_1>0$ (in place of $\dt$) be required by Lemma 7.4 of
\cite{Lnctr1} for $\ep_2/2$ (in place of $\ep$), ${\cal U}$ and
$\af.$

By \ref{small}, there is a projection $p_0\in A,$ a finite
dimensional \SCA\, $F\subset A$ with $1_F=1-p_0$ and unital
monomorphism  $h_0: C\to p_0Ap_0$ and a unital \hm\, $h_1: C\to F$
such that
\beq\label{lex-1}
[h_0+h_1]=\kappa\andeqn \tau(p_0)<\min\{\ep_1/4, \dt_1/4\}
\eneq
for all $\tau\in T(A).$ It is easy to find a projection $e_0\in
(1-p_0)A(1-p_0)$ such that $e_0$ commutes with every element in
$F$ and
\beq\label{lex-2}
\tau(e_0)<\min\{\ep_1/4, \dt_1/4\}\tforal \tau\in T(A).
\eneq
Let $h_{1,0}=e_0h_1$ and $h_{1,1}=(1-p_0-e_0)h_1.$

It follows from Lemma 9.5 of \cite{Lnctr1} that there is a \SCA\,
$B_0\subset (1-p_0-e_0)A(1-p_0-e_0)$ for which $B_0\in {\cal I}$
and there exists a unital \hm\, $h_2: C\to B_0$ such that
\beq\label{lex-3}
(h_2)_{*0}=(h_{1,1})_{*0}\andeqn |\tau\circ
h_2(f)-\tau(1-p_0-e_0)\gamma(\tau)(f)|<\ep_1/4
\eneq
for all $f\in {\cal H}_1$ and for all $\tau\in T(A).$

Now define $\phi_0=h_{1,0}$ and $\phi_1=h_0\oplus h_2.$ Note that
$\phi_0$ is homotopically trivial. By applying Lemma 7.4 of
\cite{Lnctr1}, we obtain a unital \hm\, $\Phi: C\to e_0Ae_0$ such
that
\beq\label{lex-4}
\Phi_{*0}=\phi_{*0},\,\,\,\af({\bar w})^{-1}(\Phi\oplus
\phi_1)^{\ddag}({\bar w})={\overline{g_w}},
\eneq
where $g_w\in U_0(A)$ and ${\rm cel}(g_w)<\ep_2/2$ for all
$\bar{w}\in \pi_2(\overline{{\cal U}})$ and $\Phi$ is
homotopically trivial. Now define $h=\Phi+\phi_1.$ Then clearly
\beq\label{lex-5}
[h]=[h_0+h_1]=\kappa\,\,\,{\rm in}\,\,\, KK(C,A).
\eneq
We also have
\beq\label{lex-6}
 |\tau\circ h(f)-\gamma(\tau)(f)|&<&|\tau\circ
 h(f)-\tau(1-p_0-e_0)\gamma(\tau)(f)|+\ep_1/2\\
 &<& \ep_1/4+\ep_1/2<3\ep_1/4\,\,\,\,\,\,\tforal \tau\in T(A)
 \eneq
 and for all $f\in {\cal H}_1.$ In particular, for all $x\in
 \Delta_C^{-1}(\overline{{\cal H}_0}),$
\beq\label{lex-6+1}
{\rm dist}(h^{\ddag}(x),\af(x))<\ep_2
\eneq
(see \ref{DDCU}).
 Combining this with (\ref{lex-4}), we have, for all $w\in {\cal U},$
\beq\label{lex-6+}
{\rm dist}(h^{\ddag}({\bar w}) , \af({\bar w})) <\ep_2.
\eneq

We note that the last part of the lemma follows by choosing
smaller $\ep_1$ and larger ${\cal H}.$

\end{proof}

\begin{lem}\label{LEx1}
Let $C\in {\cal C}_{0}$ and let $A$ be a unital separable simple
\CA\, with $TR(A)\le 1.$ Suppose that $\kappa\in KK_e(C,A)^{++},$
$\gamma: T(A)\to T_{\mathfrak{f}}(C)$ is a continuous affine map
and $\af: U(C)/CU(C)\to U(A)/CU(C)$ is a \hm\, for which $\gamma,$
$\af$ and $\kappa$ are compatible. Then there exists a unital
monomorphism $h: C\to A$ such that
$$
[h]=\kappa\,\,\,{\rm in}\,\,\, KK(C,A), \tau\circ
h(f)=f(\gamma(\tau))\tforal f\in C_{s,a}\tand\\
h^{\ddag}=\af.
$$

\end{lem}

\begin{proof}
Let $\Delta=\Delta_1/2$ be as in \ref{Delt} associated with
$\gamma.$

Let $\{\ep_n\}$ be a sequence of decreasing  positive numbers with
$\lim_{n\to\infty}\ep_n=0,$  $\{{\cal H}_n\}$ be an increasing
sequence of finite subsets of $C_{s.a}$ such that the union is dense
in $A_{s.a.}$ and $\{{\cal U}_n\}$ be an increasing sequence of
finite subsets of $U(C)$ such that the union is dense in $U(C).$ Let
$\{\eta_n\}$ be another sequence of decreasing positive numbers such
that $\lim_{n\to\infty}\eta_n=0.$

It follows from \ref{lex} that there exists a sequence of unital
monomorphisms $h_n: C\to A$ such that
\beq\label{LEx1-1}
[h_n]&=&\kappa\,\,\,{\rm in}\,\,\, KK(C,A),\\
|\tau\circ h_n(f)-\gamma(\tau)(f)|&<&\ep_n
\eneq
for all $\tau\in T(A)$ and for all $f\in {\cal H}_n,$ and
\beq\label{LEx1-2-}
\mu_{\tau\circ h_n}(O_a)\ge \Delta(a)\tforal \tau\in T(A)
\eneq
and for any open ball with radius $a\ge \eta_n.$ Moreover,
\beq\label{LEx1-2}
 {\rm
dist}(h_n^{\ddag}({\bar w}), \af({\bar w}))&<&\ep_n
\eneq
for all $w\in {\cal U}_n.$

Let $\{{\cal F}_n\}$ be an increasing sequence of finite subsets
of $C$ such that the union is dense in $C.$  It follows from 10.10
of \cite{Lnn1} that there exists a subsequence $\{n(k)\}$ and a
sequence of unitaries $\{u_n\}\subset A$ such that
\beq\label{LEx1-3}
{\rm ad}\, u_k\circ h_{n(k+1)}\approx_{1/2^k} {\rm ad}\,
u_{k-1}\circ h_{n(k)}\,\,\,{\rm on}\,\,\, {\cal F}_k,
\eneq
$n=1,2,....$
 It follows that $\{{\rm ad}\, u_k\circ h_{n(k+1)}\}$ is a Cauchy
 sequence. Therefore it converges. Let $h$ be the limit.
It is then easy to check that $h$ meets all requirements.

\end{proof}

\begin{lem}\label{LEx2}
Let $C\in {\cal C}$ and let $A$ be a unital separable simple \CA\,
with $TR(A)\le 1.$ Suppose that $\kappa\in KL_e(C,A)^{++},$
$\gamma: T(A)\to T(C)$ is a continuous affine map and $\af:
U(C)/CU(C)\to U(A)/CU(A)$ is a \hm\, for which $\gamma,$ $\af$ and
$\kappa$ are compatible. Then there exists a unital \hm\, $h: C\to
A$ such that
$$
[h]=\kappa\,\,\,{\rm in}\,\,\, KL(C,A),
h_T=\gamma(\tau)\tand\\
h^{\ddag}=\af.
$$
\end{lem}

\begin{proof}
We may write that $C=\lim_{n\to\infty}(C_n, \phi_{n,n+1}),$ where
$C_n$ is a finite direct sum of \CA s in ${\cal C}_0$ and each
$\phi_n$ is unital and injective.

Let $\kappa_n=\kappa\circ [\phi_{n, \infty}],$ $\af_n=\af\circ
\phi_n^{\ddag}$ and $\gamma_n=(\phi_{n,\infty})_T\circ \gamma,$
where $[\phi_{n,\infty}]\in KK(C_n,C),$ $\phi_{n, n+1}^{\ddag}:
U(C_n)/CU(C_n)\to U(A)/CU(A)$ and $(\phi_{n,\infty})_T$ are
induced by $\phi_{n, \infty}.$

It follows from \ref{LEx1} that there are unital monomorphisms
$\psi_n: C_n\to A$ such that
\beq\label{LEx2-1}
&&[\psi_n]=\kappa_n\,\,\,{\rm in}\,\,\,KK(C_n,A),\,\,\,
\psi_n^{\ddag}=\af_n\andeqn\\
&& (\psi_n)_T=\gamma_n.
\eneq
In particular,
\beq\label{LEx2-1-}
[\psi_{n+1}\circ
\phi_{n,n+1}]=\kappa_n,\,\,\,\psi_{n+1}^{\ddag}\circ
\phi_{n,n+1}^{\ddag}=\psi_n^{\ddag}\andeqn (\psi_{n+1}\circ
\phi_{n,n+1})_T=(\psi_n)_T.
\eneq

Let $\{{\cal F}_n\}$ be an increasing sequence of finite subsets
of $C$ whose union is dense in $C.$ Without loss of generality, we
may assume that there is a finite subset ${\cal F}_n'\subset C_n$
such that $\phi_{n, \infty}({\cal F}_n')={\cal F}_n,$ $n=1,2,....$
Note that $\phi_{n, n+1}({\cal F}_n')\subset {\cal F}_{n+1}',$
$n=1,2,....$

It follows from Corollary 11.7 of \cite{Lnn1} that there is a
subsequence $\{n(k)\},$ a sequence of unitaries $\{u_k\}\subset A$
such that
\beq\label{LEx2-2}
{\rm ad}\, u_k\circ \psi_{n(k+1)}\circ \phi_{n(k),
n(k+1)}\approx_{1/2^k} {\rm ad}\, u_{n-1}\circ
\psi_{n(k)}\,\,\,{\rm on}\,\,\,{\cal F}_{n(k)}',
\eneq
$k=1,2,....$

Thus one obtains a unital \hm\, $h: C\to A$ so that
\beq\label{LEx2-3}
h(f)=\lim_{k\to\infty} {\rm ad}\, u_k\circ \psi_{n(k),
n(k+1)}\circ \phi_{m, n(k)}(f)
\eneq
for all $f\in C_m,$ $m=1,2,....$ It follows that
\beq\label{LEx2-4}
[h]=\kappa\,\,\, h^{\ddag}=\af\andeqn\\
\tau\circ h(f)=f(\gamma(\tau))
\eneq
for all $\tau\in T(A)$ and all $f\in C_{s.a.}.$

\end{proof}

\begin{thm}\label{TEX}
Let $C\in {\cal C}$ and let $A$ be a unital separable simple \CA\,
with $TR(A)\le 1.$ Suppose that $\kappa\in KK_e(C,A)^{++},$
$\gamma: T(A)\to T(C)$ is a continuous affine map and $\af:
U(C)/CU(C)\to U(A)/CU(A)$ is a \hm\, for which $\gamma,$ $\af$ and
$\kappa$ are compatible. Then there exists a unital \hm\, $h: C\to
A$ such that
$$
[h]=\kappa\,\,\,{\rm in}\,\,\, KK(C,A), \tau\circ
h(f)=f(\gamma(\tau))\tforal f\in C_{s,a}\tand\\
h^{\ddag}=\af.
$$
\end{thm}

\begin{proof}
Denote by $\overline{\kappa}$ the image of $\kappa$ in $KL(C,A).$
It follows from \ref{LEx2} that there is unital monomorphism
$\phi: C\to A$ such that
\beq\label{TEX-0}
[\phi]=\overline{\kappa},\,\,\phi^{\ddag}=\af\andeqn\\
\tau\circ \phi(c)=\gamma(\tau)(c)\tforal \tau\in T(A)
\eneq
and for all $c\in C_{s.a.}.$ Regarding $\phi(C)$ as a unital
\SCA\, of $A.$ It follows from Theorem 5.4 of \cite{Lnctr1} that
$A$ is tracially approximately divisible. Note, by the UCT, that
$\kappa-[\phi]\in Pext(K_*(C), K_{*+1}(A)).$ It follow from
\ref{CHITK} that $A$ has property (B1) and property (B2)
associated with $C.$  By Theorem 3.15 of \cite{LN}, there is a
unital monomorphism $\psi_0$ and a sequence of unitaries
$\{u_n\}\subset A$ such that
\beq\label{TEX-1}
[\psi_0\circ \phi]-[\phi]=\kappa-[\phi]\,\,\,{\rm in}\,\,\,
KK(C,A)\andeqn\\\label{TEX-2}
 \lim_{n\to\infty}{\rm ad}\, u_n\circ
\psi(c)=\psi_0\circ \phi(c)\tforal c\in C.
\eneq
Define $h=\psi_0\circ \phi.$ Then $[h]=\kappa$ in $KK(C,A).$ By
(\ref{TEX-2}, we still have that
$$
h^{\ddag}=\psi^{\ddag}\andeqn \tau\circ h(c)=\gamma(\tau)(c)
$$
for all $c\in C_{s.a.}$ and for all $\tau\in T(A).$

\end{proof}

\section{Rotation maps}

The following follows from a result of K. Thomsen.
\begin{lem}\label{baraff}
Let $A$ be a unital separable simple \CA\, with $TR(A)\le 1.$
Suppose that $u\in CU(A).$ Then, for any piece-wise smooth
continuous path $\{u(t):t\in [0,1]\}$ with $u(0)=u$ and
$u(1)=1_A,$
$$
R_A(\{u(t)\})\in \overline{\rho_A(K_0(A))}\,\,\,{\rm in}\,\,\,
\text{Aff}(T(A)).
$$
\end{lem}

\begin{proof}
It follows from Corollary 3.5 of \cite{Lnnhomp} that the map $j: u\mapsto {\rm diag}(u,1,1,...,1)$
from $U(A)$ into $U(M_n(A))$ induces an isomorphism from $U(A)/CU(A)$ to $U(M_n(A))/CU(M_n(A)).$
Then the lemma follows from Lemma 3.1 of \cite{Th1} (see also Theorem 3.2 of \cite{Th1}).

\end{proof}

\begin{lem}\label{BARAFF}
Let $A$ be a unital simple \CA\, with $TR(A)\le 1.$ Let $C$ be a
unital separable \CA\, with $K_1(C)=U(C)/U_0(C).$   Suppose that
$\phi, \psi: C\to A$ are two unital monomorphisms such that
\beq
[\phi]=[\psi]\,\,\,{\rm in}\,\,\, KK(C,A),\\\label{BARAFF1}
\phi_T=\psi_T\tand \phi^{\ddag}=\psi^{\ddag}.
\eneq
Then
$$
R_{\phi, \psi}\in \text{Hom}(K_1(C), \overline{\rho_A(K_0(A))}).
$$

\end{lem}

\begin{proof}
Let $z\in K_1(C)$ be represented by a unitary $u.$ Then by
(\ref{BARAFF1}),
$$
\phi(u)\psi(u)^*\in CU(A).
$$
Suppose that $\{u(t): t\in [0,1]\}$ is a piece-wise smooth
continuous path in $U(A)$ such that $u(0)=\phi(u)$ and
$u(1)=\psi(u).$ Put $w(t)=\psi(u)^*u(t).$ Then
$w(0)=\psi(u)^*\phi(u)\in CU(A)$ and $w(1)=1_A.$ Thus
\beq
R_{\phi, \psi}(z)(\tau)&=&{1\over{2\pi}}\int_0^1
\tau({du(t)\over{dt}}u^*(t))dt\\
&=&{1\over{2\pi}}\int_0^1
\tau(\psi(u)^*{du(t)\over{dt}}u^*(t)\psi(u))dt\\
&=& {1\over{2\pi}}\int_0^1
\tau({dw(t)\over{dt}}w(t))dt
\eneq
for all $\tau\in T(A).$ By \ref{baraff},
$$
R_{\phi,\psi}(z)\in \overline{\rho_A(K_0(A))}.
$$
It follows that $R_{\phi,\psi}\in \text{Hom}(K_1(C),
\overline{\rho_A(K_0(A))}).$

\end{proof}

\begin{df}
{\rm Let $A$ be unital \CA, let $C\subset A$ be a unital \SCA.
Denote by ${\overline{\text{Inn}}}(C,A)$ the set of those
monomorphisms $\psi$ for which there exists a sequence of
unitaries $\{u_n\}\in A$ such that
$$
\psi(c)=\lim_{n\to\infty}u_n^*cu_n \tforal c\in C.
$$}

\end{df}

\begin{thm}{\rm (Theorem 4.2 of \cite{LN})}\label{NL}
Let $A$ be a unital \CA, let $C$ be a unital separable simple
\SCA\, of $A$ and denote by ${\imath}$ the embedding. Suppose that
$A$ has a positive element $b\in A$ with $sp(b)=[0,1],$ that $A$
has property (B1) and  property (B2) associated with $C$ (see 4.4 of \cite{LN}). For
any $\lambda\in {\rm Hom}(K_0(C), \overline{\rho_A(K_0(A))}),$ there
exists $\phi\in \overline{\rm{Inn}}(C,A)$ such that there are \hm
s $\theta_i: K_i(C)\to K_i(M_{\imath, \phi})$ with
$(\pi_0)_{*i}\circ \theta_i={\rm id}_{K_i(C)},$ $i=0,1,$ and the
rotation map $R_{\imath, \phi}: K_1(C)\to \text{Aff}(T(A))$ is
given by
\beq\label{NL-1}
R_{\imath,
\phi}(x)=\rho_A(c-\theta_1((\pi_0)_{*1}(x)))+\lambda\circ
(\pi_0)_{*1}(x)))\tforal x\in K_1(M_{\imath, \af}).
\eneq
In other words,
$$
[\phi]=[\imath]\,\,\,{\rm in}\,\,\, KK(C,A),
$$
and the rotation map $R_{\imath, \phi}: K_1(M_{\imath, \phi})\to
\text{Aff}(T(A))$ is given by
$$
R_{\imath, \phi}(a,b)=\rho_A(a)+\lambda(b)
$$
for some identification of $K_1(M_{\imath, \phi})$ with
$K_0(A)\oplus K_1(C).$

\end{thm}

\begin{proof}
This is proved in Theorem 4.2 of \cite{LN}. In the assumption of
Theorem 4.2 of \cite{LN}, it is assumed that $\rho_A(K_0(A))$ is
dense in $\text{Aff}(T(A)).$ However, in fact, it is
$\lambda(K_1(C))\subset \overline{\rho_A(K_0(A))}$ that is used.

\end{proof}

\begin{df}\label{KKM} Let $A$ be a
unital C*-algebra, and let $C$ be a unital separable \CA. Denote
by $\mathrm{Mon}_{asu}^e(C,A)$ the set of asymptotically unitary
equivalence classes of unital monomorphisms from $C$ into $A.$
Denote by $\small{{\boldsymbol{K}}}$ the map from
$\textrm{Mon}_{asu}^e(C, A)$ into ${KK}_e(C,A)^{++}$ defined by
$$
\phi\mapsto [\phi]\tforal \phi\in \mathrm{Mon}_{asu}^e(C,A).
$$
Let $\kappa\in {KK}_e(C,A)^{++}.$ Denote by $\langle \kappa
\rangle$ the classes of $\phi\in \mathrm{Mon}_{asu}^e(C,A)$ such
that $\small{{\boldsymbol{ K}}}(\phi)=\kappa.$

Denote by $KKUT_e(A,B)^{++}$ the set of triples $(\kappa,
\af,\gamma)$ for which $\kappa\in KK_e(A,B)^{++},$ $\af:
U(A)/CU(A)\to U(B)/CU(B)$ is a \hm\, and $\gamma: T(B)\to T(A)$ is an
affine continuous map and $\af,$  $\gamma$ and $\kappa$ are
compatible. Denote by $\boldsymbol{\mathfrak{K}}$ the map from
$\textrm{Mon}_{asu}^e(C, A)$ into ${KKUT}(C,A)^{++}$ defined by
$$
\phi\mapsto ([\phi],\phi^{\ddag}, \phi_T)\tforal \phi\in
\mathrm{Mon}_{asu}^e(C,A).
$$
Denote by $\langle \kappa, \af,\gamma \rangle $ the subset of
$\phi\in \mathrm{Mon}_{asu}^e(C,A)$ such that
$\boldsymbol{\mathfrak{K}}(\phi)=(\kappa,\,\af,\,\gamma).$

\end{df}

\begin{thm}\label{Mul}
Let $C$ and $A$ be two  unital separable simple amenable \CA s
with $TR(A)\le 1.$ Suppose that $\phi_1, \phi_2, \phi_3: A\to B$
are three unital monomorphisms for which
\beq\label{Mul-1}
&&[\phi_1]=[\phi_2]=[\phi_3]\,\,\,{\rm in}\,\,\, KK(A,B)\\
&&(\phi_1)_T=(\phi_2)_T=(\phi_3)_T
\eneq
Then
\beq\label{Mul-2}
\overline{R}_{\phi_1, \phi_2}+\overline{R}_{\phi_2,
\phi_3}=\overline{R}_{\phi_1,\phi_3}.
\eneq
\end{thm}

\begin{proof}
Let $\theta_1: K_1(C)\to K_1(M_{\phi_1, \phi_2})$ such that
$(\pi_0)_{*1}\circ \theta_1={\rm id}_{K_1(C)}$ and let $\theta_2:
K_1(C)\to K_1(M_{\phi_2, \phi_3})$ such that $(\pi_0)_{*1}\circ
\theta_2={\rm id}_{K_1(C)}.$ Define $\theta_3: K_1(C)\to K_1(M_{\phi_1,
 \af\circ \phi_2})$ as follows:

 Let $k>0$ be an integer and $u\in M_k(C)$ be a unitary.

 We may assume that there is a unitary $w(t)\in M_k(M_{\phi_1, \phi_2})$ such that
 \beq\label{path1}
 w(0)=\phi_1(u'),w(1)=\phi_2(u'),[u']=[u]\,\,\,{\rm in}\,\,\, K_1(A)\\
 \andeqn
 \theta_1([u])=[w(t)]\,\,\,{\rm in}\,\,\,K_1(M_{\phi_1,\phi_2})
 \eneq
 for some unitary $u'\in M_k(C).$
To simplify notation, without loss of generality, we may assume
that there are $h_1,h_2,...,h_n\in M_k(C)_{s.a.}$ such that
$$
u^*u'=\prod_{j=1}^n\exp(ih_j).
$$
Define $z(t)=u\prod_{j=1}^n\exp(h_jt)$ ($t\in [0,1]$). Consider
$\{\phi_1(z(t)): t\in [0,1]\}.$ Then
$$
\phi_1(z(0))=\phi_1(u)\andeqn \phi_1(z(1))=\phi_1(u').
$$
Moreover, (by Lemma 3.1 of \cite{Lnappn}, for example),
$$
\int_0^1\tau({d\phi_1(z(t))\over{dt}}\phi_1(z(t))^*)dt=\sum_{j=1}^n\tau(\phi_1(h_j))
$$
for all $\tau\in T(A).$ Consider $Z(t)=\phi_2(z(1-t)).$ Then
$$
Z(0)=\phi_2(u')\andeqn Z(1)=\phi_2(u).
$$
$$
\int_0^1\tau({d\phi_2(z(1-t))\over{dt}}\phi_2(z(1-t))^*)dt=-\sum_{j=1}^n\tau(\phi_2(h_j))
$$
for all $\tau\in T(A).$ Note that
$$
\tau(\phi_2(h_j))=\tau(\phi_1(h_j))\tforal \tau\in
T(B),\,\,\,j=1,2,...,n.
$$
It follows that
$$
\int_0^1\tau({d\phi_1(z(t))\over{dt}}\phi_1(u(t))^*)dt
+\int_0^1\tau({d\phi_2(z(1-t))\over{dt}}\phi_2(z(1-t))^*)dt=0
$$
for all $\tau\in T(A).$

Therefore,
 without loss of generality, we may assume that $u=u'$ in
(\ref{path1}). We may also assume that both paths are piecewise
smooth.

We may also assume that there is a unitary $s(t)\in
  M_k(M_{\phi_2, \phi_3})$ such that
  \beq
 s(0)=\phi_2(u),s(1)=\phi_3(u)\,\,\,{\rm in}\,\,\, K_1(A)\\
 \andeqn
 \theta_2([u])=[s(t)]\,\,\,{\rm in}\,\,\,K_1(M_{\phi_2,\phi_3}).
 \eneq

 Define
 $\theta_3([u])=[v],$ where
 \beq\label{path}
v(t)=\begin{cases} w(2t)\,\,\,\text{if}\,\,\, t\in
[0,1/2)\\
s(2(t-1/2))\,\,\,\text{if}\,\,\, t\in [1/2,1],
\end{cases}
\eneq
Thus $\theta_3$ gives a \hm\, from $K_1(A)$ to $K_1(M_{
\phi_1,\phi_3})$ such that $(\pi_0)_{*1}\circ \theta_3={\rm
id}_{K_1(A)}.$ Then
\beq
R_{\phi_1,\phi_3}(\theta_3([u]))(\tau)&=& {1\over{2\pi i}}\int_0^1
\tau({dv(t)\over{dt}}v(t)^*)dt\\
&=&{1\over{2\pi i}}\int_0^{1/2} \tau({dw(2t)\over{dt}}w(2t)^*)dt+\\
&&{1\over{2\pi i}}\int_{1/2}^{1}
\tau({ds(2(t-1/2))\over{dt}}s(2(t-1/2))^*)dt\\
&=& R_{\phi_1, \phi_2}\circ \theta_0([u])(\tau)+R_{
\phi_2,\phi_3}\circ \theta_1([u])(\tau)\\
\eneq
for all $\tau\in T(B).$ Thus (\ref{Mul-2}) holds.

\end{proof}

\begin{lem}\label{Group1}
Let $A$ and let $B$ be two unital separable simple amenable \CA s
with $TR(A)\le 1.$ Suppose that $\phi_1, \phi_2: A\to B$ are  two
unital monomorphisms such that
$$
[\phi_1]=[\phi_2]\,\,\,{\rm in}\,\,\, KK(A,B)\andeqn
(\phi_1)_T=(\phi_2)_T.
$$
Suppose that $(\phi_2)_T: T(B)\to T(A)$ is an affine
homeomorphism. Suppose also that $\af\in Aut(B)$ with
$$[\af]=[{\rm id}_B]\,\,\,{\rm in}\,\,\, KK(B,B)\andeqn \af_T={\rm id}_T.
$$
Then
\beq\label{Group1-0}
\overline{R}_{\phi_1,\af\circ \phi_2} =\overline{R}_{{\rm
id}_B,\af}\circ (\phi_2)_{*1}+\overline{R}_{\phi_1,\phi_2}
\eneq
in ${\rm Hom}(K_1(A), \text{Aff}(T(B)))/{\cal R}_0.$

\end{lem}

\begin{proof}
By \ref{Mul}, we compute that
\beq
\overline{R}_{\phi_1,\af\circ \phi_2} &=& \overline{R}_{\phi_1,
\phi_2}+\overline{R}_{\phi_2,\af\circ \phi_2}\\
&=&\overline{R}_{\phi_1,\phi_2}+\overline{R}_{{\rm id}_B,\af}\circ
(\phi_2)_{*1}.
\eneq

\end{proof}

\begin{thm}\label{MT2}
Let $C\in {\cal N}$ be a unital   simple \CA\, with
$TR(C)\le 1$ and let $A$ be a unital separable simple \CA\, with
$\mathrm{TR}(A)\le 1.$ Then the map $\boldsymbol{\mathfrak{K}}:
\mathrm{Mon}_{asu}^e(C,A)\to {KKUT}(C,A)^{++}$ is surjective.
Moreover, for each $(\kappa, \af, \gamma)\in {KKUT}(C,A)^{++}$,
there exists a bijection
$$
\eta: \langle \kappa,\af,\gamma \rangle \to \mathrm{Hom}({K}_1(C),
{\overline{\rho_A(K_0(A))}})/{\cal R}_0.
$$
\end{thm}

\begin{proof}
It follows from \ref{TEX} that $\boldsymbol{\mathfrak{K}}$ is
surjective.

Fix a triple $(\kappa,\af, \gamma)\in {KKT}(C,A)^{++}$ and choose
a unital monomorphism $\phi: C\to A$ such that $[\phi]=\kappa$,
$\phi^{\ddag}=\af$ and  $\phi_\mathrm{T}=\gamma.$
If $\phi_1: C\to A$ is another unital monomorphism such that
$\boldsymbol{\mathfrak{K}}(\phi_1)=\boldsymbol{\mathfrak{K}}(\phi).$
Then by \ref{BARAFF},
$$
\overline{R}_{\phi, \phi_1}\in \mathrm{Hom}({K}_1(C),
{\overline{\rho_A(K_0(A))}})/{\cal R}_0.
$$
Let $\lambda\in \mathrm{Hom}({K}_1(C), \overline{\rho_A(K_0(A))})$
be a \hm. It follows from \ref{CHITK} that $A$ has property (B1)
and property (B2) associated with $C.$   By applying \ref{NL}, we
obtain a unital monomorphism $\psi\in
{\overline{{\rm{Inn}}}}(\phi(C), A)$ with $[\psi\circ
\phi]=[\phi]$ in ${KK}(C,A)$ such that there exists a homomorphism
$\theta: {K}_1(C)\to K_1(M_{\phi, \psi\circ \phi})$ with
$(\pi_0)_{*1}\circ \theta={\rm{id}}_{{K}_1(C)}$ for which
$R_{\phi, \psi\circ \phi}\circ \theta=\lambda.$ Let
$\beta=\psi\circ \phi.$ Then $R_{\phi, \beta}\circ
\theta=\lambda.$ Note also since $\psi\in
{\overline{\rm{Inn}}}(\phi(C), A),$ $\beta^{\ddag}=\phi^{\ddag}$
and $\beta_\mathrm{T}=\phi_\mathrm{T}.$ In particular,
$\boldsymbol{\mathfrak{K}}(\beta)={\boldsymbol{\mathfrak{K}}}(\phi).$

Thus we obtain a well-defined  and surjective map $$\eta: \langle
[\phi], \phi^{\ddag}, \phi_T\rangle \to \mathrm{Hom}({K}_1(A),
\overline{\rho_A(K_0(A))})/{\cal R}_0.$$

 To see it is one to one, let
$\phi_1, \phi_2: C\to A$ be  two unital monomorphisms in $ \langle
[\phi], \phi^{\ddag},\phi_T\rangle$  such that
$$
\overline{R}_{\phi, \phi_1}=\overline{R}_{\phi, \phi_2}.
$$
Then, by \ref{Mul},
\beq
\overline{R}_{\phi_1,\phi_2}&=&\overline{R}_{\phi_1,\phi}+\overline{R}_{\phi,
\phi_2}\\
&=&-\overline{R}_{\phi, \phi_1}+\overline{R}_{\phi, \phi_2}=0.
\eneq
It follows from \ref{TM} that $\phi_1$ and $\phi_2$ are
asymptotically unitarily equivalent.

\end{proof}

\begin{df}
{\rm  Denote by $KKUT_e^{-1}(A,A)^{++}$ the subgroup of those
elements $\langle \kappa, \af, \gamma\rangle \in KKUT_e(A,A)^{++}$
for which $\kappa|_{K_i(A)}$ is an isomorphism $(i=0,1$), $\af$ is an
isomorphism and $\gamma$ is a affine homeomorphism.  Denote by
$\eta_{{\rm id}_A}=\eta|_{\langle [{\rm id}_A], {\rm
id}_A^{\ddag}, ({\rm id}_A)_T\rangle}.$

Denote by $\langle {\rm id}_A\rangle $ the class of those
automorphisms $\psi$ which are asymptotically unitarily equivalent
to ${\rm id}_A.$ Note that, if $\psi\in \langle {\rm id}_A\rangle
,$ then $\psi$ is {\it asymptotically inner}, i.e., there exists a
continuous path of unitaries $\{u(t): t\in [0,\infty)\}\subset A$
such that
$$
\psi(a)=\lim_{t\to\infty}u(t)^*au(t)\tforal a\in A.
$$

}
\end{df}

\begin{cor}\label{Inn}
Let $A\in {\cal N}$ be a unital simple  \CA\, with
$TR(A)\le 1.$ Then one
has the following short exact sequence:
\beq\label{Inn1}
0 \to {\rm Hom}(K_1(A), \overline{\rho_A(K_0(A))})/{\cal
R}_0\stackrel{\eta_{{\rm id}_A}^{-1}}{\to}{\rm Aut}(A)/\langle
{\rm id}_A\rangle \stackrel{{\boldsymbol{\mathfrak{K}}}}{\to}
KKUT_e^{-1}(A,A)^{++}\to 0.
\eneq

In particular, if $\phi, \psi\in Aut(A)$ such that
$$
{\boldsymbol{\mathfrak{K}}}(\phi)={\boldsymbol{\mathfrak{K}}}(\psi)={\boldsymbol{\mathfrak{K}}}({\rm
id}_A),
$$
Then
$$
\eta_{{\rm id}_A}(\phi\circ \psi)=\eta_{{\rm
id}_A}(\phi)+\eta_{{\rm id}_A}(\psi).
$$

\end{cor}

\begin{proof}
It follows from \ref{TEX} that,  for any $\langle \kappa, \af,
\gamma\rangle,$ there is a unital monomorphism $h: A\to A$ such
that ${\boldsymbol{\mathfrak{K}}}(h)=\langle \kappa, \af,
\gamma\rangle.$ The fact that $\kappa\in KK_e^{-1}(A,A)^{++}$
implies that there is $\kappa_1\in KK_e^{-1}(A,A)^{++}$ such that
$$
\kappa\times \kappa_1=\kappa_1\times \kappa=[{\rm id}_A].
$$
By \ref{TEX}, choose $h_1: A\to A$ such that
$$
{\boldsymbol{\mathfrak{K}}}(h)=\langle \kappa_1, \af^{-1},
\gamma^{-1}\rangle.
$$

It follows from Corollary 11.7 of \cite{Lnn1} that $h\circ h_1$
and $h\circ h_1$ are approximately unitarily equivalent. Applying
a standard approximate intertwining argument of G. A. Elliott, one
obtains two isomorphisms $\phi$ and $\phi^{-1}$ such that there is
a sequence of unitaries $\{u_n\}$ in $A$ such that
$$
\phi(a)=\lim_{n\to\infty}{\rm ad}\, u_{2n+1}\circ h(a)\andeqn
\phi^{-1}(a)=\lim_{n\to\infty}{\rm ad}\, u_{2n}\circ h_1(a)
$$
for all $a\in A.$ Thus $[\phi]=[h]$ in $KL(A,A)$ and
$\phi^{\ddag}=h^{\ddag}$ and $\phi_T=h_T.$ Then, as in the proof
of \ref{TEX}, there is $\psi_0\in {\overline{Inn}}(A,A)$ such that
$[\psi_0\circ \phi]=[{\rm id}_A]$ in $KK(A,A)$ as well as
$(\psi_0\circ \phi)^{\ddag}=h^{\ddag}$ and $(\psi_0\circ
\phi)_T=h_T.$  So we have $h\in Aut(A,A)$ such that
${\boldsymbol{\mathfrak{K}}}(h)=\langle \kappa, \af,
\gamma\rangle.$

Now let $\lambda\in {\rm Hom}(K_1(C), \overline{\text{Aff}(T(A))})/{\cal
R}_0.$  The proof \ref{MT2} says that there is $\psi_{00} \in
\overline{\text{Inn}}(A,A)$ (in place of $\psi$) such that
${\boldsymbol{\mathfrak{K}}}(\psi_{00}\circ {\rm
id}_A)={\boldsymbol{\mathfrak{K}}}({\rm id}_A)$ and
$$
\overline{R}_{{\rm id}_A, \psi_{00}}=\lambda.
$$
Note that $\psi_{00}$ is again an automorphism.

The last part of the lemma then follows from \ref{Group1}.

\end{proof}

\section{Strong asymptotic unitary equivalence}

Let $\phi_1, \phi_2: C\to A$ be two unital monomorphisms which are
asymptotically unitarily equivalent. A natural question is: Can
one find a continuous path of unitaries $\{u_t: t\in [0,
\infty)\}$ of $A$ such that
$$
u_0=1_A\andeqn \lim_{t\to\infty}{\rm ad}\, u_t\circ
\phi_1(a)=\phi_2(a)\tforal a\in A?
$$
It is known (see section 10 of \cite{Lnasym} ) that, in general,
unfortunately, the answer is negative.

\begin{df}{\rm (Definition 12.1 of \cite{Lnasym})}\,
{\rm Let $A$ and $B$ be two unital \CA s.  Suppose that $\phi_1,
\phi_2: A\to B$ are two unital \hm s. We say that $\phi_1, \phi_2$
are strongly asymptotically unitarily equivalent if there exists a
continuous path of unitaries $\{u_t: t\in [0, \infty)\}$ of $B$
such that
$$
u_0=1_B\andeqn \lim_{t\to\infty}{\rm ad}\, u_t\circ
\phi_1(a)=\phi_2(a)\tforal a.
$$
}

\end{df}

\begin{df}{\rm (Definition 10.2 of \cite{Lnasym} and see also \cite{Lnequ})}\label{Dsu}
Let $A$ be a unital \CA\, and $B$ be another \CA. Recall
(\cite{Lnequ}) that
$$
H_1(K_0(A), K_1(B))=\{x\in K_1(B): \phi([1_A])=x,\,\phi\in
{\rm Hom}(K_0(A), K_1(B))\}.
$$
\end{df}

Let us quote  the following:

\begin{prop}{\rm (Proposition 12.3 of \cite{Lnasym})}\label{Sup}
Let $A$  be a unital separable \CA\, and let $B$ be a unital \CA.
Suppose that $\phi: A\to B$ is a unital \hm\, and $u\in U(B)$ is a
unitary. Suppose that there is a continuous path of unitaries
$\{u(t): t\in [0,\infty)\}\subset B$ such that
\beq\label{sup-1}
u(0)=1_B\andeqn \lim_{t\to\infty}{\rm ad}\, u(t)\circ \phi(a)={\rm
ad}\, u\circ \phi(a)
\eneq
for all $a\in A.$  Then
$$
[u]\in H_1(K_0(A), K_1(B)).
$$

\end{prop}

\begin{lem}\label{SL1}
Let $C=\lim_{n\to\infty}(C_n, \psi_n),$ where $C_n$ is a finite
direct sum of \CA s in ${\cal C}_{00}$ and $\psi_n$ is unital and
injective, and let $A$ be a unital separable simple \CA\, with
$TR(A)\le 1.$ Suppose that $\phi_1, \phi_2: C\to A$ are two
monomorphisms such that there is an increasing sequence of finite
subsets ${\cal F}_n\subset C$ whose union is dense in $C,$ an
increasing sequence of finite subsets $\{{\cal P}_n\}$ of
$K_1(C)$ whose union if $K_1(C),$ a sequence of positive numbers $\{\dt_n\}$ with
$\sum_{n=1}^{\infty}\dt_n<1$ and a sequence of unitaries
$\{u_n\}\subset A,$ such that
\beq\label{SL1-1}
{\rm ad}\, u_n\circ \phi_1\approx_{\dt_n} \phi_2\,\,\,{\rm
on}\,\,\,{\cal F}_n \andeqn\\
\rho_A({\rm bott}_1(\phi_2, u_n^*u_{n+1})(x)=0\tforal x\in {\cal
P}_n.
\eneq
Then, we may further require that $u_n\in U_0(A),$
 if $H_1(K_0(C), K_1(A))=K_1(A).$

\end{lem}

\begin{proof}
Let $x_n=[u_n]$ in $K_1(A).$ Then, since $K_1(A)=H_1(K_0(C),
K_1(A)),$ there is a \hm\, $\kappa_{n,0}: K_0(C)\to K_1(A)$ such
that $\kappa_{n,0}([1_C])=-x_n.$
By the Universal Coefficient Theorem, there
is $\kappa_n\in KK(C,A)$ such that
\beq\label{sl1-1}
(\kappa_n)|_{K_0(C)}=\kappa_{n,0}\andeqn (\kappa_n)|_{K_1(C)}=0.
\eneq

There is, for each $n,$ a positive number $\eta_n<\dt_n,$  such
that
\beq\label{sl1-2}
{\rm ad}\, u_n\circ \phi_1\approx_{\eta_n}
\phi_2\,\,\,\text{on}\,\,\, {\cal F}_n.
\eneq
It follows from \ref{HITK} that, there is a unitary $w_n\in U(A)$
such that
\beq\label{sl1-3}
\|[\phi_2(a), w_n]\|&<&(\dt_n-\eta_n)/2\rforal a\in {\cal
F}_n\andeqn\\
\text{Bott}(\phi_2, w_n)&=&\kappa_n.
\eneq
Put $v_n=u_nw_n,$ $n=1,2,....$ Then, we have
$$
{\rm ad}\,v_n\circ
\phi_1\approx_{\dt_n}\phi_2\,\,\,\text{on}\,\,\,{\cal F}_n\andeqn
$$
$$
\rho_A(\text{bott}_1(\phi_2, v_n^*v_{n+1}))=0\andeqn
[v_n]=[u_n]-x_n=0.
$$

\end{proof}

\begin{thm}\label{ST1}
Let $C$ be as in \ref{SL1} and let $A$ be a unital separable
simple \CA\, with $TR(A)\le 1.$  Suppose that
$H_1(K_0(C),K_1(A))=K_1(A)$ and suppose that $\phi_1, \phi_2: C\to
A$ are two unital monomorphisms which are asymptotically unitarily
equivalent. Then there exists a continuous path of unitaries
$\{u(t): t\in [0, \infty)\}\subset A$ such that
$$
u(0)=1\andeqn \lim_{t\to\infty}{\rm ad}u(t)\circ
\phi_1(a)=\phi_2(a)\rforal a\in C.
$$
\end{thm}

\begin{proof}
By \ref{NecT}, we must have
$$
[\phi_1]=[\phi_2]\,\,\,\text{in}\,\,\, KK(C,A),\,\,\,
\overline{R}_{\phi_1,\phi_2}=0\andeqn
$$
$$
(\phi_1)_T=(\phi_2)_T.
$$

By \ref{SL1}, in the proof of \ref{TM}, we may assume that $v_n\in
U_0(A),$ $n=1,2,....$ It follows that $\xi_n([1_C])=0,$
$n=1,2,....$ Therefore $\kappa_n([1_C])=0.$ This implies that
$\gamma_n\circ \boldsymbol{\bt}([1_C])=0.$ Hence $w_n\in U_0(A).$
It follows that $u_n\in U_0(A).$ Therefore there is a continuous
path of unitary $\{U(t): t\in [-1,0]\}$ in $A$ such that
$$
U(-1)=1_A\andeqn U(0)=u(0).
$$
The theorem then follows.

\end{proof}

\begin{cor}\label{Scc2}
Let $C$ be as in \ref{SL1} and let $A$ be a unital separable
simple \CA\, with $TR(A)\le 1.$ Let $\phi: C\to B$ be a unital
monomorphism and let $u\in U(A).$ Then there exists a continuous
path of unitaries $\{U(t): t\in [0,\infty)\}\subset A$ such that
$$
U(0)=1_B\andeqn \lim_{t\to\infty}{\rm ad}\, U(t)\circ \phi(a)={\rm
ad}\, u\circ \phi(a)\rforal a\in C
$$
if and only if $[u]\in H_1(K_0(C), K_1(A)).$
\end{cor}

\begin{proof}
This follows from \ref{Sup} and the proof of \ref{ST1}.
\end{proof}

\begin{cor}\label{SCZ}
Let $C$ be as in \ref{SL1} with $K_0(C)=\Z\cdot [1_C]\oplus G$ and
let $B$ be a unital separable simple \CA\, with $TR(B)\le 1.$
Suppose that $\phi_1, \phi_2: C\to B$ are two unital monomorphisms
such that
\beq\label{scz1}
[\phi_1]&=&[\phi_2]\,\,\,\text{in}\,\,\,KK(C,B),\,\,\,\overline{
R}_{\phi_1,\phi_2}=0\andeqn\\
(\phi_1)_T&=&(\phi_2)_T\rforal \tau\in T(B).
\eneq
Then there exists a continuous path of unitaries $\{u_t:t\in [0,
\infty)\}$ of $B$ such that
$$
u_0=1_B\andeqn \lim_{t\to\infty}{\rm ad}\, u_t\circ
\phi_1(a)=\phi_2(a)\rforal a\in A.
$$

\end{cor}

\begin{proof}
 Let $x\in K_1(B).$ Then one defines $\gamma: K_0(C)\to K_1(B)$
by $\gamma([1_C])=x$ and $\gamma|_G=0.$ This implies that
$$
H_1(K_0(C), K_1(B))=K_1(B).
$$
Thus the corollary follows from \ref{ST1}.
\end{proof}

\section{Classification of separable simple amenable \CA s}

\begin{NN}
{\rm Let ${\mathfrak{p}}$ be a supernatural number. Denote by
$M_{\mathfrak{p}}$ the UHF-algebra associated with ${\mathfrak{p}}$
(see \cite{Dix}). Denote by $Q$ the UHF-algebra with $K_0(Q)=\Q$ and
$[1_Q]=1.$

}
\end{NN}

\begin{lem}\label{ucu}
Let $A$ be a unital separable simple \CA\, with $TR(A)\le 1$ and
let $\mathfrak{p}$ be a supernatural number of infinite type. Then
the \hm\, $\imath: a\mapsto a\otimes 1$ induces an isomorphism
from $U_0(A)/CU(A)$ to $U_0(A\otimes M_{\mathfrak{p}})/CU(A\otimes
M_{\mathfrak{p}}).$
\end{lem}

\begin{proof}
There are  sequences of positive integers $ \{m(n)\}$ and
$\{k(n)\}$ such that $A\otimes
M_{\mathfrak{p}}=\lim_{n\to\infty}(A\otimes M_{m(n)}, \imath_n),$
where
$$
\imath_n: M_{m(n)}(A)\to M_{m(n+1)}(A)
$$
is defined by $\imath(a)={\rm diag}(\overbrace{a,a,...,a}^{k(n)})$
for all $a\in M_{m(n)}(A),$ $n=1,2,....$ Note, by Theorem 5.8 of
\cite{Lnplms}, $TR(M_{m(n)}(A))\le 1.$ Let $j_n:
U(M_{m(n)}(A))/CU(M_{m(n)}(A)))\to
U(M_{m(n+1)}(A))/CU(M_{m(n+1)}(A))$ be defined by
$$
j_n({\bar u})=\overline{{\rm
diag}(u,\underbrace{1,1,...,1}_{k(n)-1})}\tforal u\in
U((M_{m(n)}(A)).
$$

It follows from Theorem 6.7 of \cite{Lnctr1} and Corollary 3.5 of \cite{Lnnhomp} that $j_n$ is
an isomorphism.
 By Lemma 6.6 of \cite{Lnctr1}, $U_0(M_{m(n)}(A))/CU(M_{m(n)}(A))$
is divisible. For each $n$ and $i,$ there is a unitary $U_i\in
M_{m(n+1)}(A)$ such that
$$
U_i^*E_{1,1}U_i=E_{i,i},\,\,\,i=2,3,...,k(n),
$$
where $E_{i,i}=\sum_{j=(i-1)m(n)+1}^{im(n)}e_{i,i}$  and
$\{e_{i,j}\}$ is a matrix unit for $M_{m(n+1)}.$ Then
$$
\imath_n(u)=u'U_2^*u'U_2\cdots U_{k(n)}^*u'U_{k(n)},
$$
where $u'={\rm diag}(u,\overbrace{1,1,...,1}),$ for all $u\in
M_{m(n)}(A).$ Thus
$$
\imath_n^{\ddag}({\bar u})=k(n)j_n(\bar u).
$$

It follows that
$\imath_n^{\ddag}|_{U_0(M_{m(n)}(A))/CU(M_{m(n)}(A))}$ is
injective, since $U_0(M_{m(n+1)}(A))/CU(M_{m(n+1)}(A))$ is torsion
free (see Theorem 6.11 of \cite{Lnctr1}). For each $z\in
U_0(M_{m(n+1)}(A)/CU(M_{m(n+1)}),$ there is a unitary $v\in
M_{m(n+1)}(A)$ such that
$$
j_n({\bar v})=z,
$$
since $j_n$ is an isomorphism. By the divisibility of
$U_0(M_{m(n)}(A)/CU(M_{m(n)}),$ there is $u\in M_{m(n)}(A)$ such
that
$$
\overline{u^{k(n)}}=\overline{u}^{k(n)}=\overline{v}.
$$
As above,
$$
\imath_n^{\ddag}({\bar u})=k(n)j_n(\bar v)=z.
$$
So $\imath_n^{\ddag}|_{U_0(M_{m(n)}(A))/CU(M_{m(n)}(A))}$ is a
surjective.  It follows that $\imath_{1,
\infty}^{\ddag}|_{U_0(M_{m(n)}(A))/CU(M_{m(n)}(A))}$ is also an
isomorphism. One then concludes that
$\imath^{\ddag}|_{U_0(A)/CU(A)}$ is an isomorphism.

\end{proof}

\begin{lem}\label{cu2}
Let $A$ and $B$ be two unital separable simple \CA s  with
$TR(A)\le 1.$ Let $\phi: A\to B$ be an isomorphism and let $\bt:
B\otimes M_{\mathfrak{p}}\to B\otimes M_{\mathfrak{p}}$ be an
automorphism such that $\bt_{*1}={\rm id}_{K_1(B\otimes
M_{\mathfrak{p}})}$ for some supernatural number $\mathfrak{p}$ of infinite type.
Then
$$\psi^{\ddag}(U(A)/CU(A))=
(\phi_0)^{\ddag}(U(A)/CU(A))=U(B)/CU(B),
$$
where $\phi_0=\imath\circ \phi,$ $\psi=\bt\circ \imath\circ \phi$
and where $\imath: B\to B\otimes M_{\mathfrak{p}}$ is defined by
$\imath(b)=b\otimes 1$ for all $b\in B.$ Moreover there is an
isomorphism $\mu: U(B)/CU(B)\to U(B)/CU(B)$ with
$\mu(U_0(B)/CU(B))\subset U_0(B)/CU(B)$ such that
$$
\imath^{\ddag}\circ \mu\circ \phi^{\ddag}=\psi^{\ddag}\andeqn
q_1\circ \mu=q_1,
$$
where $q_1: U(B)/CU(B)\to K_1(B)$ is the quotient map.
\end{lem}

\begin{proof}
Since  $\phi_{*1}$ is an isomorphism and $\bt_{*1}={\rm
id}_{K_1(B\otimes M_{\mathfrak{p}})},$ we have that
\beq\label{cu2-0}
(\psi)_{*1}(K_1(A))= (\phi_0)_{*1}(K_1(A)).
\eneq
Moreover, 
${\rm ker}(\imath\circ \phi)_{*1}={\rm
ker}(\imath\circ \psi)_{*1}.$

Let $q_1': U(B\otimes M_{\mathfrak{p}})/CU(B\otimes
M_{\mathfrak{p}})\to K_1(B\otimes M_{\mathfrak{p}})$  and $q_1'':
U(A)/CU(A)\to K_1(A)$ be the quotient maps. Define
$G=q_1'^{-1}((\phi_0)_{*1}(K_1(A))).$

It follows from \ref{ucu} that
\beq\label{cu1-1}
(\psi)^{\ddag}(U_0(A)/CU(A))=U_0(B\otimes
M_{\mathfrak{p}})/CU(B\otimes
M_{\mathfrak{p}})=(\phi_0)^{\ddag}(U_0(A)/CU(A)).
\eneq
Considering the following diagram:
$$
\begin{array}{cccccccc}
0 &\to &  U_0(A)/CU(A) &\to & U(A)/CU(A) & \to & K_1(A) & \to 0\\
& &\downarrow_{(\psi)^{\ddag}}& & \downarrow_{(\psi)^{\ddag}} &&
\downarrow_{\psi_{*1}}\\
0&\to& U_0(B\otimes M_{\mathfrak{p}})/CU(B\otimes
M_{\mathfrak{p}})&\to & G &\to & (\phi_0)_{*1}(K_1(A)) &\to 0
\end{array}
$$
Combining this with (\ref{cu1-1}), we have that
$$
\psi^{\ddag}(U(A)/CU(A))\subset (\phi_0)^{\ddag}(U(A)/CU(A)).
$$
We then conclude that
$$
\psi^{\ddag}(U(A)/CU(A))= (\phi_0)^{\ddag}(U(A)/CU(A)).
$$
Furthermore, we also have
\beq\label{cu1-2}
q_1'\circ \phi_0^{\ddag}=(\phi_0)_{*1}\circ q_1''\andeqn q_1'\circ
(\psi)^{\ddag}=\psi_{*1}\circ q_1''.
\eneq
It follows from Lemma 6.6 that $U_0(A)/CU(A)$ is divisible. There
is an injective \hm\, $\dt_1: K_1(A)\to U(A)/CU(A)$ such that
$q_1''\circ \dt_1={\rm id}_{K_1(A)}$  is the quotient map and we
write
$$
U(A)/CU(A)=U_0(A)/CU(A)\oplus \dt_1(K_1(A)).
$$
We also have an injective \hm\, $\dt_2: K_1(B)\to U(B)/CU(B)$ such
that $q_1\circ \dt_1={\rm id}_{K_1(B)}$ and we write
$$
U(B)/CU(B)=U_0(B)/CU(B)\oplus \dt_2(K_1(B)).
$$
We also have an injective \hm\, $\dt_3: K_1(B\otimes M_{\mathfrak{p}})\to U(B\otimes M_{\mathfrak{p}})/CU(B\otimes M_{\mathfrak{p}})$ such that $q_1'\circ \dt_3={\rm id}_{K_1(B\otimes M_{\mathfrak{p}})}.$
 Note that
$$
q_1'\circ \imath^{\ddag}\circ ({\rm id}_{U(B)/CU(B)}-\dt_2\circ
q_1)=0\andeqn q_1'\circ \imath^{\ddag}=\imath_{*1}\circ q_1.
$$
We also have
\beq\label{cu-3}
\dt_3\circ \imath_{*1}=\imath^{\ddag}\circ \dt_2.
\eneq

It follows that
$$
G=U_0((B\otimes M_{\mathfrak{p}})/CU(B\otimes
M_{\mathfrak{p}})\oplus \imath^{\ddag}\circ \dt_2\circ (\phi)_{*1}\circ
q_1''(U(A)/CU(A)).
$$

Let
$$j=(\imath^{\ddag}|_{U_0(B\otimes M_{\mathfrak{p}})/CU(B\otimes
M_{\mathfrak{p}})})^{-1}:U_0(B\otimes
M_{\mathfrak{p}})/CU(B\otimes M_{\mathfrak{p}})\to U_0(B)/CU(B).
$$
Define $\bt^{\ddag,0}: U(B)/CU(B)\to U_0(B)/CU(B)$ by
\beq\label{cu-9}
\bt^{\ddag,0}=j\circ (\bt^{\ddag}\circ \imath^{\ddag}-\dt_3\circ
q_1'\circ \bt^{\ddag}\circ \imath^{\ddag}).
\eneq
Note that $\bt^{\ddag,0}|_{U_0(B)/CU(B)}=j\circ\bt^{\ddag}\circ
\imath^{\ddag}|_{U_0(B)/CU(B)},$ since $\bt^{\ddag}$ maps
$U_0(B)/CU(B)$ to $U_0(B)/CU(B).$ Thus
\beq\label{cu-9+2}
\bt^{\ddag,0}\circ \phi^{\ddag}=j\circ
(\bt^{\ddag}\circ\imath^{\ddag}\circ \phi^{\ddag}-\dt_3\circ
q_1'\circ \bt^{\ddag}\circ \imath^{\ddag}\circ\phi^{\ddag}).
\eneq
Define $\mu: U(B)/CU(B)\to U(B)/CU(B)$ by
\beq\label{cu2-9+3}
\mu=\dt_2\circ q_1+\bt^{\ddag,0}.
\eneq
We compute that
\beq\label{cu2-9+4}
\imath^{\ddag}\circ \mu\circ \phi^{\ddag}&=& \imath^{\ddag}\circ
\dt_2\circ q_1\circ \phi^{\ddag}+\imath^{\ddag}\circ\bt^{\ddag,0}\circ \phi^{\ddag}\\
&=& \imath^{\ddag}\circ \dt_2\circ \phi_{*1}\circ q_1''\\
&&+\bt^{\ddag}\circ \imath^{\ddag}\circ \phi^{\ddag}-\dt_3\circ
q_1'\circ
\bt^{\ddag}\circ \imath^{\ddag}\circ\phi^{\ddag}\\
&=& \imath^{\ddag}\circ \dt_2\circ \phi_{*1}\circ q_1''\\
&&+\bt^{\ddag}\circ\imath^{\ddag}\circ \phi^{\ddag}-\dt_3\circ
\bt_{*1}\circ \imath_{*1}\circ
\phi_{*1}\circ q_1''\\
&=& \imath^{\ddag}\circ \dt_2\circ \phi_{*1}\circ q_1''\\
&&+\bt^{\ddag}\circ\imath^{\ddag}\circ \phi^{\ddag}-\dt_3\circ
\imath_{*1}\circ
\phi_{*1}\circ q_1''\\
&=&\bt^{\ddag}\circ\imath^{\ddag}\circ \phi^{\ddag}=\psi^{\ddag}.
\eneq
Moreover,
$$
q_1\circ \mu=q_1\circ \dt_2\circ q_1+q_1\circ \bt^{\ddag,0}=q_1.
$$
Since $\mu|_{U_0(B)/CU(B)}$ and $q_1\circ \mu
|_{\dt_2(U(B)/CU(B))}$ are both isomorphisms, one checks that
$\mu$ is an isomorphism.


\end{proof}

\begin{lem}\label{L10}
Let $A$ and let $B$ be two unital   simple amenable \CA s  in
${\cal N}$ with $TR(A)\le 1$ and $TR(B)\le 1.$  Suppose that
$\phi_1, \phi_2: A\to B$ are two isomorphisms such that
$[\phi_1]=[\phi_2]$ in $KK(A,B).$ Then there exists an
automorphism $\bt: B\to B$ such that $[\bt]=[{\rm id}_B]$ in
$KK(B,B)$ and $\bt\circ \phi_2$ is asymptotically unitarily
equivalent to $\phi_1.$ Moreover, if $H_1(K_0(A), K_1(B))=K_1(B),$
they are strongly asymptotically unitarily equivalent.
\end{lem}

\begin{proof}
It follows from \ref{MT2} that there is an automorphism $\bt_1:
B\to B$ satisfying the following:
\beq
[\bt_1]=[{\rm id}_B]\,\,\,{\rm in}\,\,\,KK(B,B),\\
\bt_1^{\ddag}=\phi_1^{\ddag}\circ (\phi_2^{-1}))^{\ddag}\andeqn
(\bt_1)_T=(\phi_1)_T\circ (\phi_2)_T^{-1}.
\eneq
By \ref{Inn}, there is automorphism $\bt_2\in Aut(B)$ such that
\beq
[\bt_2]=[{\rm id}_B]\,\,\, {\rm in}\,\,\, KK(B,B),\\
\bt_2^{\ddag}={\rm id}_B^{\ddag},\,\,\,(\bt_2)_T=({\rm
id}_B)_T\andeqn\\
 \overline{R}_{{\rm id}_B,\bt_2}=-\overline{R}_{\phi_1,\bt_1\circ \phi_2}\circ
(\phi_2)_{*1}^{-1}.
\eneq
Put $\bt=\bt_2\circ \bt_1.$
 It follows that
\beq
[\bt\circ \phi_2]=[\phi_1]\,\,\,{\rm in}\,\,\, KK(A,B),\\
(\bt\circ \phi_2)^{\ddag}=\phi_1^{\ddag}\andeqn (\bt\circ
\phi_2)_T=(\phi_1)_T.
\eneq
Moreover, by \ref{Group1},
\beq
\overline{R}_{\phi_1,\bt\circ \phi_2}&=&\overline{R}_{{\rm
id}_B,\bt_2}\circ
(\phi_2)_{*1}+\overline{R}_{\phi_1,\bt_1\circ\phi_2}\\
&=&(-\overline{R}_{\phi_1,\bt_1\circ \phi_2}\circ
(\phi_2)_{*1}^{-1})\circ (\phi_2)_{*1}
+\overline{R}_{\phi_1,\bt_1\circ\phi_2}=0.
\eneq
It follows from \ref{MT2} that $\bt\circ \phi_2$ and $\phi_1$ are
asymptotically unitarily equivalent.

In the case that $H_1(K_0(A), K_1(B))=K_1(B),$ it follows from
\ref{ST1} that $\bt\circ \phi_2$ and $\phi_1$ are strongly
asymptotically unitarily equivalent.

\end{proof}

\begin{lem}\label{l1}
Let $A$ and let $B$ be two  unital simple amenable \CA\, in ${\cal
N}$ with $TR(A),\,TR(B)\le 1$ and let $\phi: A\to B$ be an
isomorphism. Suppose that $\bt\in Aut(B\otimes M_{\mathfrak{p}})$
for which
$$[\bt]=[{\rm id}_{B\otimes
M_{\mathfrak{p}}}]\,\,\,{\rm in}\,\,\,KK(B\otimes
M_{\mathfrak{p}},B\otimes M_{\mathfrak{p}}) \andeqn \bt_T=({\rm
id}_{B\otimes M_{\mathfrak{p}}})_T
$$
for some supernatural number $\mathfrak{p}$ of infinite type.

Then there exists an automorphism $\af\in Aut(B)$ with
$[\af]=[{\rm id}_{B}]$ in $KK(B,B)$ such that $ \imath\circ
\af\circ \phi $ and $\bt\circ \imath\circ \phi$ are asymptotically
unitarily equivalent, where $\imath: B\to B\otimes
M_{\mathfrak{p}}$ is defined by $\imath(b)=b\otimes 1$ for all
$b\in B.$
\end{lem}

\begin{proof}
It follows from \ref{cu2} that there is an isomorphism $\mu:
U(B)/CU(B)\to U(B)/CU(B)$ such that
$$
\imath^{\ddag}\circ \mu\circ \phi^{\ddag}=(\bt\circ \imath\circ
\phi)^{\ddag}.
$$
Note that $\imath_T: T(B\otimes M_{\mathfrak{p}})\to T(B)$ is an
affine homeomorphism.

It follows from \ref{MT2} that there is an automorphism $\af: B\to
B$ such that
\beq\label{l1-1}
&&[\af]=[{\rm id}_B]\,\,\,{\rm in}\,\,\,KK(B,
B),\\
&&\af^{\ddag}=\mu,\,\,\, \af_T=(\bt\circ \imath\circ \phi)_T\circ
((\imath\circ
\phi)_T)^{-1}=({\rm id}_{B\otimes M_{\mathfrak{p}}})_T\andeqn\\
&&\overline{R}_{{\rm
id}_B,\af}(x)(\tau)=-\overline{R}_{\bt\circ\imath\circ \phi,\,
\imath\circ \phi}(\phi_{*1}^{-1}(x))(\imath_T(\tau))\tforal x\in
K_1(A)
\eneq
and for all $\tau\in T(B).$

Denote by $\psi=\imath\circ \af\circ \phi.$ Then we have, by
\ref{Group1},
\beq
&&[\psi]= [\imath\circ \phi]=[\bt\circ\imath\circ \phi]\,\,\,{\rm
in}\,\,\, KK(A, B\otimes M_{\mathfrak{p}})\\
&&\psi^{\ddag}=\imath^{\ddag}\circ\mu\circ \phi^{\ddag}=(\bt\circ
\imath\circ \phi)^{\ddag},\\
&&\psi_T=(\imath\circ \af\circ \phi)_T=(\imath\circ
\phi)_T=(\bt\circ \imath\circ \phi)_T. \eneq Moreover, for any $x\in
K_1(A)$ and $\tau\in T(B\otimes M_{\mathfrak{p}})$ (by \ref{Group1})

 \beq
\overline{R}_{\bt\circ \imath\circ \phi, \psi}(x)(\tau)&=&
\overline{R}_{\bt\circ \imath\circ\phi,
\imath\circ\phi}(x)(\tau)+\overline{R}_{\imath, \imath\circ
\af}\circ
\phi_{*1}(x)(\tau)\\
&=&\overline{R}_{\bt\circ \imath\circ\phi,
\imath\circ\phi}(x)(\tau)+\overline{R}_{{\rm id}_B, \imath\circ \af}\circ
\phi_{*1}(x)(\imath_T^{-1}(\tau))\\
&=&\overline{R}_{\bt\circ \imath\circ\phi,
\imath\circ\phi}(x)(\tau)-\overline{R}_{\bt\circ \imath\circ\phi,
\imath\circ\phi}(\phi_{*1}^{-1})(\phi_{*1}(x))(\tau)=0
\eneq
It follows from \ref{MT2}  that $\imath\circ \af\circ \phi$ and
$\bt\circ \imath\circ \phi$ are asymptotically unitarily
equivalent.

\end{proof}

\begin{NN}
{\rm Let $A$ be a unital separable simple \CA. By the Elliott
invariant we mean the 6-tuple:
$$
Ell(A)=(K_0(A), K_0(A)_+, [1_A], K_1(A), T(A),\rho_A),
$$
where $\rho_A: T(A)\to S_{[1]}(K_0(A))$ (where $S_{[1]}(K_0(A))$ is
the state space of $K_0(A)$) is a surjective affine continuous map
such that $\rho_A(\tau)([p])=\tau(p)$ for any projection $p\in
A\otimes {\cal K}.$

 Let $B$ be another unital separable simple \CA. We say that
there is an isomorphism
$$
\Gamma: Ell(A)\to Ell(B)
$$
if there is an  order isomorphism $\kappa_0: (K_0(A),
K_0(A)_+,[1_A])\to (K_0(B), K_0(B)_+, [1_B]),$  there is an
isomorphism $\kappa_1: K_1(A)\to K_1(B)$ and there is an affine
homeomorphism $\mu: T(A)\to T(B)$ for which
$$
\mu^{-1}(\tau)(p)=\rho_B(\tau)(\kappa_0([p]))
$$
for all projection $p\in M_k(A)$ (for all $k\ge 1$) and all tracial
states $\tau\in T(B).$

}
\end{NN}

\begin{thm}\label{CMT1}
Let $A$ and $B$ be two unital separable  simple
\CA s in ${\cal N}.$  Suppose that there is an isomorphism
$$
\Gamma: Ell(A)\to Ell(B).
$$
Suppose that, for some pair of relatively prime supernatural
numbers $\mathfrak{p}$ and $\mathfrak{q}$ of infinite type such
that $M_{\mathfrak{p}}\otimes M_{\mathfrak{q}}\cong Q,$
$TR(A\otimes M_{\mathfrak{p}})\le 1,$ $ TR(B\otimes
M_{\mathfrak{p}})\le 1,$ $TR(A\otimes M_{\mathfrak{q}})\le 1$ and
$TR(B\otimes M_{\mathfrak{q}}))\le 1.$ Then,
$$
A\otimes {\cal Z}\cong B\otimes {\cal Z}.
$$
\end{thm}

\begin{proof}
Note that $\Gamma$ induces an isomorphism
$$
\Gamma_{\mathfrak{p}}: Ell(A\otimes M_{\mathfrak{p}})\to
Ell(B\otimes M_{\mathfrak{p}}).
$$
 Since $TR(A\otimes M_{\mathfrak{p}})\le 1$ and
$TR(B\otimes M_{\mathfrak{p}})\le 1,$ by Theorem 10.10 of
\cite{Lnctr1}, there is an isomorphism $\phi_{\mathfrak{p}}:
A\otimes M_{\mathfrak{p}}\to B\otimes M_{\mathfrak{p}}.$ Moreover,
(by the proof of Theorem 10.4 of \cite{Lnctr1}),
$\phi_{\mathfrak{p}}$ carries $\Gamma_{\mathfrak{p}}.$  For
exactly the same reason, $\Gamma$ induces an isomorphism
$$
\Gamma_{\mathfrak{q}}:Ell(A\otimes M_{\mathfrak{q}})\to
Ell(B\otimes M_{\mathfrak{q}})
$$
and there is an isomorphism $\psi_{\mathfrak{q}}: A\otimes
M_{\mathfrak{q}}\to B\otimes M_{\mathfrak{q}}$ which induces
$\Gamma_{\mathfrak{q}}.$

 Put $\phi=\phi_{\mathfrak{p}}\otimes {\rm
id}_{M_{\mathfrak{q}}}: A\otimes Q\to B\otimes Q$ and
$\psi=\psi_{\mathfrak{q}}\otimes {\rm id}_{M_{\mathfrak{p}}}:
A\otimes Q\to B\otimes Q.$

Note that
$$
(\phi)_{*i}=(\psi)_{*i}\,\,{\rm (} i=0,1 {\rm )} \andeqn
\phi_T=\psi_T.
$$
(they are induced by $\Gamma$). Note that $\phi_T$ and $\psi_T$
are affine homeomorphisms. Since $K_{*i}(B\otimes Q)$ is
divisible, we in fact have $[\phi]=[\psi]$ (in $KK(A\otimes Q,
B\otimes Q)$). It follows from \ref{L10} that there is an
automorphism $\bt: B\otimes Q\to B\otimes Q$ such that
$$
[\bt]=[{\rm id}_{B\otimes Q}]\,\,\,KK(B\otimes Q, B\otimes Q)
$$
as well as  $\phi$ and $\bt\circ \psi$ are asymptotically unitarily
equivalent.  Since $K_1(B\otimes Q)$ is divisible, $H_1(K_0(A\otimes
Q), K_1(B\otimes Q))=K_1(B\otimes Q).$ It follows that $\phi$ and
$\bt\circ \psi$ are strongly asymptotically unitarily equivalent.
Note also in this case
$$
\bt_T=({\rm id}_{B\otimes Q})_T.
$$
Let $\imath: B\otimes M_{\mathfrak{q}}\to B\otimes Q$ defined by
$\imath(b)=b\otimes 1$ for $b\in B.$ We consider the pair $\bt\circ
\imath\circ \psi_{\mathfrak{q}}$ and $\imath \circ
\psi_{\mathfrak{q}}.$ By applying \ref{l1}, there exists an
automorphism $\af: B\otimes M_{\mathfrak{q}}\to B\otimes
M_{\mathfrak{q}}$ such that $\imath\circ \af\circ
\psi_{\mathfrak{q}}$ and $\bt\circ \imath\circ \psi_{\mathfrak{q}}$
are asymptotically unitarily equivalent (in $M(B\otimes Q)$). So
they are strongly asymptotically unitarily equivalent. Moreover,
$$
[\af]=[{\rm id}_{B\otimes M{\mathfrak{q}}}]\,\,\,{\rm in}\,\,\,
KK(B\otimes M_{\mathfrak{q}},B\otimes M_{\mathfrak{q}}).
$$

We will show that $\bt\circ \psi$ and $\af\circ
\psi_{\mathfrak{q}}\otimes {\rm id}_{M_{\mathfrak{p}}}$ are strongly
asymptotically unitarily equivalent. Define $\bt_1=\bt\circ
\imath\circ\psi_{\mathfrak{q}}\otimes {\rm id}_{M_{\mathfrak{p}}}:
B\otimes Q\otimes M_{\mathfrak{p}}\to B\otimes Q\otimes
M_{\mathfrak{p}}.$ Let $j: Q\to Q\otimes M_{\mathfrak{p}}$ defined
by $j(b)=b\otimes 1.$ There is an isomorphism $s:
M_{\mathfrak{p}}\to M_{\mathfrak{p}}\otimes M_{\mathfrak{p}}$ with
$({\rm id}_{M_{\mathfrak{q}}}\otimes s)_{*0}=j_{*0}.$ In this case
$[{\rm id}_{M_{\mathfrak{q}}}\otimes s]=[j].$ Since
$K_1(M_{\mathfrak{p}})=0.$ By \ref{TM}, ${\rm
id}_{M_{\mathfrak{q}}}\otimes s$ is strongly asymptotically
uniatrily equivalent to $j.$ It follows that $\af\circ
\psi_{\mathfrak{q}}\otimes {\rm id}_{M_{\mathfrak{p}}}$ and
$\bt\circ \imath\circ \psi_{\mathfrak{q}}\otimes {\rm
id}_{M_{\mathfrak{p}}}$ are strongly asymptotically unitarily
equivalent. Consider the \SCA\, $C=\bt\circ \psi(1\otimes
M_{\mathfrak{p}})\otimes M_{\mathfrak{p}}\subset B\otimes Q\otimes
M_{\mathfrak{p}}.$ In $C,$  $\bt\circ \phi|_{1\otimes
M_{\mathfrak{p}}}$ and $j_0$ are strongly asymptotically unitarily
equivalent, where $j_0: M_{\mathfrak{p}}\to C$ is defined by
$j_0(a)=1\otimes a$ for all $a\in M_{\mathfrak{p}}.$ There exists a
continuous path of unitaries $\{v(t): t\in [0,\infty)\}\subset C$
such that \beq\label{CM1-1} \lim_{t\to\infty}{\rm ad}\, v(t)
\circ\bt\circ \phi(1\otimes a)=1\otimes a\tforal a\in
M_{\mathfrak{p}}. \eneq It follows that $\bt\circ \psi$ and $\bt_1$
are strongly asymptotically unitarily equivalent. Therefore
$\bt\circ \psi$ and $\af\circ \psi_{\mathfrak{q}}\otimes {\rm
id}_{M_{\mathfrak{p}}}$ are strongly asymptotically unitarily
equivalent. Finally, we conclude that $\af\circ
\psi_{\mathfrak{q}}\otimes {\rm id}_{\mathfrak{p}}$ and $\phi$ are
strongly asymptotically unitarily equivalent. Note that $\af\circ
\psi_{\mathfrak{q}}$ is an isomorphism which induces
$\Gamma_{\mathfrak{q}}.$

Let $\{u(t): t\in [0,1)\}$ be a continuous path of unitaries in
$B\otimes Q$ with $u(0)=1_{B\otimes Q}$ such that
$$
\lim_{t\to\infty}{\rm ad}\, u(t)\circ \phi(a)=\af\circ
\psi_{\mathfrak{q}}\otimes {\rm id}_{M_{\mathfrak{q}}}(a)\tforal
a\in A\otimes Q.
$$
One then obtains a unitary suspended isomorphism which lifts
$\Gamma$ along $Z_{p,q}$ (see \cite{W}). It follows from Theorem
7.1 of \cite{W} that $A\otimes {\cal Z}$ and $B\otimes {\cal Z}$
are isomorphic.

\end{proof}

\begin{df}\label{Class}
{\rm Denote by ${\cal A}$ the class of those unital  simple
\CA s $A$  in ${\cal N}$ for which $TR(A\otimes
M_{\mathfrak{p}})\le 1$ for any supernatural number
${\mathfrak{p}}$ of infinite type.

Of course ${\cal A}$  contains all unital simple amenable \CA s
with tracial rank no more than one  which satisfy the UCT. It
contains the Jiang-Su algebra ${\cal Z}.$ It is also known to
contain all unital simple \CA s which are locally type I with
unique tracial state.
Moreover, it contains all unital simple ${\cal Z}$-stable
ASH-algebras $A$ such that $T(A) =S_{[1]}(K_0(A)),$ where $S_{[1]}(K_0(A))$ is the state space of $K_0(A).$  In these
three cases, the tensor products of these unital simple \CA s with
any infinite dimensional UHF-algebras actually have tracial rank
zero (see \cite{W} and \cite{Lnappn}).

There are of course  unital simple \CA s $A$ for which
$A\otimes M_{\mathfrak{p}}$ has tracial rank one (not zero). For example,
all  unital simple AH-algebras have this property.
The
following corollary states that \CA s in ${\cal A}$ can be
classified up to ${\cal Z}$-stably isomorphism by their Elliott
invariant.

}

\end{df}

\begin{cor}\label{CM1}
Let $A$ and $B$ be two \CA s in ${\cal A}.$ Then $A\otimes {\cal
Z}\cong B\otimes {\cal Z}$ if and only if $Ell(A\otimes {\cal
Z})\cong Ell(B\otimes {\cal Z}).$
\end{cor}

\begin{proof}
This follows from \ref{CMT1} immediately.

\end{proof}

\begin{thm}\label{Cperm}

 {\rm (i)}\, Every unital hereditary \SCA\, of a
\CA\, in ${\cal A}$ is in ${\cal A};$

{\rm (ii)}\, If $A\in {\cal A},$ then $M_n(A)\in {\cal A}$ for
integer $n\ge 1;$

 {\rm (iii)}\, If $A\in {\cal A},$ then $A\otimes {\cal Z}\in {\cal
 A};$

{\rm (iv)} If $A$ and $B$ are in ${\cal A},$ then $A\otimes B\in
{\cal A},$ and

{\rm (v)}\, If $\{A_n\}\subset {\cal A}$ and
$A=\lim_{n\to\infty}(A_n, \phi_n),$ then $A\in {\cal A},$

{\rm (vi)} Every unital simple AH-algebra is in ${\cal A}.$

{\rm (v)} ${\cal Z}\in {\cal A}.$
\end{thm}

\begin{proof}

{\rm (i)}\, Let $A\in {\cal A}.$ Then, for any supernatural number
$\mathfrak{p},$ $TR(A\otimes M_{\mathfrak{p}})\le 1.$  Let $p\in
A$ be a projection. Then $pAp\otimes M_{\mathfrak{p}}$ is a unital
hereditary \SCA\, of $A\otimes M_{\mathfrak{p}}.$ Therefore, by 5.3
of \cite{Lnplms}, $TR(pAp\otimes M_{\mathfrak{p}})\le 1.$ Thus
$pAp\in {\cal A}.$

{\rm (v)} Suppose that $A_n\in {\cal A}$ and
$A=\lim_{n\to\infty}(A_n, \phi_n)$ and $\mathfrak{p}$ is a
supernatural number. Then $TR(A_n\otimes M_{\mathfrak{p}})\le 1.$ It
follows that  $TR(A\otimes M_{\mathfrak{p}})\le 1.$

{\rm (iv)}. Suppose that $A,\, B\in {\cal A}$ and suppose that
$\mathfrak{p}$ is a supernatural number. Note that
$M_{\mathfrak{p}}\cong M_{\mathfrak{p}}\otimes M_{\mathfrak{p}}.$
It follows that
$$
(A\otimes B)\otimes M_{\mathfrak{p}}\cong (A\otimes
M_{\mathfrak{p}})\otimes (B\otimes M_{\mathfrak{p}}).
$$
Since $TR(A\otimes M_{\mathfrak{p}})\le 1$ and $TR(B\otimes
M_{\mathfrak{p}})\le 1,$ by 10.10 and 10.9  of \cite{Lnctr1}, $A\otimes
M_{\mathfrak{p}}$ and $B\otimes M_{\mathfrak{p}}$ are unital
simple AH-algebra with no dimension growth. It is easy to see that
$(A\otimes M_{\mathfrak{p}})\otimes (B\otimes M_{\mathfrak{p}})$
is also a unital simple AH-algebra with no dimension growth. It
follows that $(A\otimes B)\otimes M_{\mathfrak{p}}\in {\cal A}.$

{\rm (iii)}\, Note that $M_n, \, {\cal Z}\in {\cal A}.$ Thus (iii)
follows from (iv).

{\rm (vi)}\, Suppose that $A$ is a unital simple AH-algebra.  Then
it is easy to see that $A\otimes M_{\mathfrak{p}}$  has very slow
dimension growth (in the sense of Gong) for any supernatural numbers
$\mathfrak{p}.$ It follows from Theorem 2.5 of \cite{Lnctr1} that
$TR(A\otimes M_{\mathfrak{p}})\le 1.$

{\rm (v)} ${\cal Z}\in {\cal A}$ follows from the fact that ${\cal Z}\otimes M_{\mathfrak{p}}$ is
approximately divisible and whose projections separate the traces. In fact that ${\cal Z}\otimes M_{\mathfrak{p}}$
has tracial rank zero.

\end{proof}

\begin{cor}\label{MTF}
Let $A$ and $B$ be two unital simple AH-algebras. Then $A\otimes
{\cal Z}\cong B\otimes {\cal Z}$ if $Ell(A\otimes {\cal Z})\cong
Ell(B\otimes {\cal Z}).$
\end{cor}

\begin{proof}
This follows from (vi) of \ref{Cperm} and \ref{CM1} immediately.

\end{proof}

One may compare the following  with Theorem 3.4 of \cite{TW2}.

\begin{cor}\label{CF}
Let $A$ be a unital simple infinite dimensional AH-algebra. Then the
following are equivalent:

{\rm (1)}\, $A$ is ${\cal Z}$-stable,

{\rm (2)} $A$ is approximately divisible,

{\rm (3)} $TR(A)\le 1,$

 {\rm (4)} $A$ is isomorphic to a unital
simple AH-algebra with very slow dimension growth,

{\rm (5)}  $A$ is isomorphic to a unital simple AH-algebra with no
dimension growth.

\end{cor}

\begin{proof}
The implication (5) to (4) follows from the definition. That (2)
implies (1) follows from Theorem 2.3 of \cite{TW}, and that (4)
implies (3) was proved in Theorem 2.5 of \cite{Lnctr1}. By
\cite{EGLa}, (5) implies (2).  It follows from Theorem 10.10  of
\cite{Lnctr1} that (3) implies (5).  It remains to show that (1)
implies (5).

Let $A$ be a unital simple AH-algebra which is ${\cal Z}$-stable.
Then, by Theorem 3.6 of \cite{TW}, there is a unital simple
AH-algebra $B$ with no dimension growth such that $Ell(A)=Ell(B).$
It follows from the implication that (5) implies (2) and that (2)
implies (1) that $B$ is also ${\cal Z}$-stable. Therefore, by
\ref{MTF}, $A\cong B.$ Consequently, (1) implies (5).

\end{proof}

\begin{cor}
Let $A$ be a unital simple infinite dimensional AH-algebra. Then
$A\otimes {\cal Z}$ is isomorphic to a unital simple AH-algebra with
no dimension growth.

\end{cor}

\begin{proof}
The proof of this is contained in the proof of the implication
from (1) to (5).
\end{proof}

\begin{NN}
{\rm It is known that all unital simple ASH-algebras $A$ whose
projections separate the traces are in ${\cal A}$ (see \cite{W}).
In fact, in this case, $A\otimes M_{\mathfrak{p}}$ has tracial
rank zero (see \cite{W}). Moreover, any unital simple \CA\, with unique tracial state which
is an inductive limit of type I \CA s is  in ${\cal A}$ (see also
\cite{Lnappn}). In a subsequent paper (\cite{LN2}), we will show that
 all  unital simple  inductive limits of dimension drop
algebras studied by Jiang and Su (\cite{JS}), as well as those
unital simple inductive limits of dimension drop circle algebras
studied in \cite{Myg}  are in ${\cal A}.$ These give examples of
unital simple \CA s $A$ for which $A\otimes M_{\mathfrak{p}}$ has
tracial rank one but not zero. Other unital simple ASH-algebras
whose $K_0$-groups are not Riesz groups are also shown to be in
${\cal A}.$ The range of invariants of ${\cal A}$ will be discussed.
}
\end{NN}

\end{document}